\documentclass[a4paper,11pt]{article}
\usepackage[usenames]{color}
\usepackage[utf8]{inputenc}
\usepackage{amsmath}
\usepackage{amssymb}
\usepackage{amsfonts}
\usepackage{amsthm,amscd}
\usepackage[utf8]{inputenc}
\newtheorem{theorem}{Theorem}[section]
\newtheorem{lemma}{Lemma}[section]
\newtheorem{corollary}{Corollary}[section]
\newtheorem{proposition}{Proposition}[section]
\theoremstyle{definition}
\newtheorem{definition}{Definition}[section]
\newtheorem{example}{Example}[section]

\theoremstyle{remark}
\newtheorem{remark}[theorem]{Remark}

\numberwithin{equation}{section}

\title{On cosymplectic dynamics}
\author{\scshape 
	S. Tchuiaga	$^{}$\thanks{tchuiagas@gmail.com,\ Department of Mathematics,  
		The University of Buea, South West Region,Cameroon}, F. Houenou	$^{}$\thanks{rdjeam@gmail.com, \ Department of Mathematics,  
		The University of Abomey-Calavi, Benin},  and P. Bikorimana	$^{}$\thanks{pierrebikorimana@gmail.com, \ Institut de Mathematiques et des Science Physiques, Benin} }

\definecolor{couleurliens}{rgb}{1.0,0.,0.} 
\definecolor{couleurliensref}{rgb}{0.,0.,1.} 
\definecolor{couleurliensurl}{rgb}{.3,.4,.3} 
\usepackage[pdftex,colorlinks=true,urlcolor=couleurliensurl,linkcolor=couleurliens,citecolor=couleurliensref]{hyperref}

\begin{document}
	\maketitle\large 
\begin{abstract} 
Cosymplectic geometry can be viewed as an odd dimensional counterpart of symplectic geometry. Likely in the symplectic case, a related property which is preservation of closed forms $\omega$ and $\eta$,  
refers to the theoretical possibility of further understanding a cosymplectic manifold $(M, \omega, \eta)$ from its group of diffeomorphisms. 
In this  paper we study the structures of the group of cosymplectic diffeomorphisms and the group of almost cosymplectic diffeomorphisms of a  cosymplectic manifold $(M, \omega, \eta)$ in threefold: first of all, we study  cosymplectics counterpart of the Moser 
isotopy method, a proof of a cosymplectic version of Darboux theorem follows, and we define and present the features of the space of almost cosymplectic vector fields (resp. cosymplectic vector fields), this set forms a Lie group whose Lie 
algebra is the group of all almost cosymplectic diffeomorphisms (resp. all cosymplectic diffeomorphisms); $(II)$ we prove by a direct method that the identity component in the group of all cosymplectic diffeomorphisms is $C^0-$closed in the group $Diff^\infty(M)$ (a rigidity result: without appealing to similar result from symplectic geometry), while in the almost cosymplectic case, we prove that the Reeb vector field determines the almost cosymplectic nature of the $C^0-$limit $\phi$  of a sequence of almost cosymplectic diffeomorphisms (a rigidity result). A sufficient condition (based on Reeb's vector field) which guarantees that $\phi$ is a cosymplectic diffeomorphism is given (a flexibility condition), and also an attempt to the study
 almost cosymplectic (resp. cosymplectic) counterpart of flux geometry follows: this gives rise to a group homomorphism whose  
kernel is path connected; and $(III)$ we study the almost cosymplectic (resp. cosymplectic)  analogues of Hofer geometry and Hofer-like geometry: the group of almost co-Hamiltonian diffeomorphisms (resp. co-Hamiltonian diffeomorphisms) carries two bi-invariant norms, the cosymplectic analogues of the usual symplectic capacity-inequality theorem are derived and the cosymplectic analogues of a result that was proved by Hofer-Zehnder follow. This includes further interesting results that will be useful for a subsequent development of the theory. 
\end{abstract}


\vspace{1cm}
\section{Introduction}
The study of cosymplectic manifolds can be motivated by the seek  of a mathematical framework in which an odd dimensional counterpart of the usual symplectic geometry can be studied. 
Cosymplectic manifolds have been at the center of interest 
of several authors starting from  the works of Liberman \cite{P.L} to the works of Li \cite{H-L}.  
An interesting result due to Li \cite{H-L} shows how from a cosymplectic structure, one can always generate a symplectic structure, vise-versa. This seems to suggest that some cosymplectic properties can be derived from the usual symplectic geometry. But one can count with  fingertips the cosympletic analogues of some deep results from symplectic topology while many of them still unknown. For instance, there is any cosymplectic formulation to  neither Arnold conjecture nor  energy-capacity-inequality in symplectic geometry. 
Besides,  let us remark that the existence of an isomorphism between  vector fields space  and  the space of $1-$forms is a fundamental aspect which is common to both  symplectic and cosymplectic  geometries.  In symplectic geometry context, the corresponding isomorphism is used to investigate various structures of the group of symplectomorphisms and rigidity results from the space of closed $1-$forms of a symplectic manifold.

  As far as we know, almost nothing is known about the structure of the group of cosymplectic 
diffeomorphisms of a cosymplectic manifold. In fact,  one doesn't know how to fragment an arbitrary  co-Hamiltonian vector field of a cosymplectic manifold with respect to an open cover of the  manifold;   how to explain or describe accurately  the dynamics of a cosymplectic manifold from its group of cosymplectic diffeomorphisms ; how to quantify the energy needed by a cosymplectic diffeomorphism in order to displace a given open subset of the corresponding manifold;  how rigid is the cosymplectic  structure with respect to the $C^0-$limit and so on.\\ 
The goal of the present paper is to use the methods from symplectic topology to characterize, and then to study several subgroups of diffeomophisms of a cosymplectic manifold: several well-known results from symplectic topology have their cosymplectic counterparts, and structures of various transformation groups of a cosymplectic manifold follows.\\  

We organize this paper as follows. \\ 
  
   In Section \ref{SC0},  we recall the definition of cosymplectic vector space, cosymplectic linear group, and cosymplectic manifolds. In Subsection  \ref{SC00},  we claim and prove the following weak stability 
   result in cosymplectic geometry setting: Theorem \ref{thm00}, Theorem \ref{thm01}, Theorem \ref{thm02}, and Theorem \ref{thm4}.  In Subsection  \ref{Darboux-lem}, we prove Lemma \ref{Darboux1}, and use it to prove the cosymplectic Darboux theorem (Theorem  \ref{thm1}).\\ 
   Section \ref{SC00-1} deals with vector fields of a cosymplectic manifold. In Subsection \ref{SC011}, we define and study cosymplectic vector fields. In Subsection \ref{SC011-A}, we introduce and study the algebra of almost cosymplectic vector fields, prove Proposition \ref{ADecom-0} which is the almost cosymplectic analogue of  Proposition \ref{Decom-0}.\\ 
   Section \ref{SC00-2} deals with the study of almost cosymplectic diffeomorphisms, cosymplectic diffeomorphisms, almost cosymplectic isotopies, and cosymplectic isotopies. Lemma \ref{pushfoward} shows how the Reeb vector field of a cosymplectic manifold transforms under the push forward by a cosymplectic diffeomorphism (resp. by an almost cosymplectic diffeomorphism), while Lemma \ref{Gpushfoward} is a slight generalization of Lemma \ref{pushfoward} between two arbitrary 
   cosympletic manifolds, and Proposition \ref{pro4-A} characterizes the Lie algebras of some cosympletic and  almost cosymplectic subspaces. In Subsection \ref{rem1} we show how does cosymplectic  (resp. almost cosymplectic) geometry varies under the composition and inversion cosymplectic  (resp. almost cosymplectic) isotopies. In Subsection \ref{Trans}, we compare cosymplectic isotopies with symplectic isotopies, and characterize periodic orbits in co-Hamiltonian dynamical systems. In Subsection \ref{Transs}, we compare almost cosymplectic isotopies with symplectic isotopies 
   (Proposition \ref{Trasit-1}, Proposition \ref{Trasit-2}, Proposition \ref{Trasit-3}, and Proposition \ref{Trasit-4}), 
   prove a  theorem showing that the Reeb vector field determines the almost cosymplectic nature of a uniform limit of a sequence of almost cosymplectic diffeomorphisms (Theorem \ref{Theo-Al1}), and characterize periodic orbits in almost co-Hamiltonian dynamical systems.\\  Section \ref{Flux} is devoted to the introduction of cosymplectic flux geometries. Here we define and study two group homomorphisms with respect to cosymplectic isotopies and almost cosymplectic 
   isotopies: two topological invariants.  \\ 
    In Section \ref{Hofer-N}, we define and study the cosymplectic setting of Hofer and Hofer-like geometries with respect to the group of all cosymplectic diffeomorphisms (resp. almost cosymplectic diffeomorphisms) isotopic to the identity map. Here,  we first start from  a comparison of the uniform sup norm of a closed $1-$form and that of its pull-back with respect to a projection map, next we  study the co-Hofer norms,  co-Hofer-like norms, establish the cosymplectic setting of the energy-capacity-inequality from symplectic geometry, discuss on displacement energies in the fibers of a certain Cartesian product, and prove that the group of all cosympletic diffeomorphisms isotopic to the identity map is $C^0-$closed inside the group of all smooth diffeomorphisms (Theorem \ref{thm2}, Theorem \ref{maint1}, Theorem \ref{thm3}, and Theorem \ref{closure}). Also, we  study the almost co-Hofer norms,  almost co-Hofer-like norms, establish the almost cosymplectic setting of the energy-capacity-inequality from symplectic geometry, discuss on the displacement energy in fibers  (Theorem \ref{thm5A}, Theorem \ref{thm6A}, and Theorem \ref{maint1-A}). 
  Section \ref{Illus} contains some illustrations of our constructions.

\section{Preliminaries}\label{SC0}
\subsection{Cosymplectic vector spaces}
A bilinear form on a vector space $V$ is a map $b : V\times V \longrightarrow \mathbb{R}$ which is linear in each variable. When $b(u, v) = - b(v,u)$, for all $u, v\in V$, then $b$ is called antisymmetric (or skew symmetric). 
\begin{theorem}\label{stand-0}({\bf Standard form for antisymmetric bilinear maps})
Let $b$ be an skew-symmetric bilinear map on $V$. Then there is a basis
$u_1,\cdots, u_k, e_1,\cdots, e_n, f_1,\cdots, f_n$ of $V$ such that
\begin{itemize}
\item $b(u_i, v) = 0$, for all $i$ and for all $v$ ;
\item $b(e_i, e_j) = 0 = b(f_i, f_j)$, for all $i, j$ ; 
\item $b(e_i, f_j) = \delta_{ij}$, for all $i, j$.
\end{itemize}
\end{theorem}

Given any non-trivial linear map $\psi:V\longrightarrow \mathbb{R}$, together with a bilinear map\\ $b : V\times V \longrightarrow \mathbb{R}$, one defines a linear map 
\begin{eqnarray}\begin{array}{ccclccc}
\widetilde{I}_{\psi, b}: & V  & \longrightarrow &V^*\\
 &Y&\longmapsto &\widetilde{I}_{\psi, b}(Y):=\imath_Y b + \psi(Y) \psi
\end{array}\nonumber\end{eqnarray}
so that $\widetilde{I}_{\psi, b}(Y)(X) =b(Y,X) + \psi(Y)\psi(X)$, for all $X, Y \in V$.
\begin{definition}
\begin{enumerate}
\item A pair $(b, \psi)$ consisted of an  antisymmetric bilinear map\\ $b : V\times V \longrightarrow \mathbb{R}$ and a non-trivial linear map $\psi:V\longrightarrow \mathbb{R}$ is a called cosymplectic couple if the map  $ \tilde{I}_{\psi, b}$ is a bijection. 
\item A cosymplectic vector space is a triple $(V, b, \psi)$ where  $ V$ is a vector space and $(b, \psi)$ is a  
	cosymplectic   structure on $V$. 
\end{enumerate}
\end{definition}
\begin{proposition}\label{Dim}
	Let $(V, b, \psi)$ be a cosymplectic vector space. Then, $\dim (V) = 2n + 1$. 
\end{proposition}

\begin{proof} 
	 By Theorem \ref{stand-0}, we have $\dim (V) = 2n + k$. So, it is enough to show that $ k = 1$. Let 
	 $u_1,\dots, u_k, e_1,\dots, e_n, f_1,\dots, f_n$ be a basis of $V$ as in Theorem \ref{stand-0}. Since, 
	 $b(u_i, v) = 0$, for all $1\leq i \leq k$, and for all $v$, we derive that  $ \tilde{I}_{\psi, b}(u_i) = \psi(u_i)\psi$, for each $i$. This implies that 
	 \begin{equation}\label{reduc}
	 u_i = \tilde{I}_{\psi, b}^{-1}(\psi(u_i)\psi) = \psi(u_i)\tilde{I}_{\psi, b}^{-1}(\psi) =: \psi(u_i)\xi,
	 \end{equation}
	 for each $i$, where $\xi:= \tilde{I}_{\psi, b}^{-1}(\psi)$. Since  $u_1,\dots, u_k, e_1,\dots, e_n, f_1,\dots, f_n$ is a basis of $V$, then (\ref{reduc}) implies that 
	 $ k = 1$.  This completes the proof.
\end{proof}

\begin{remark}\label{Dim1}	Let $(V, b, \psi)$ be a cosymplectic vector space. From the proof of Proposition \ref{Dim}, one can always assume that $ \psi(\xi) = 1$ (of course after normalization, if necessary), and we also have $b(\xi, v) = 0$, for all $v\in V$. We shall call the vector $\xi$, the Reeb vector of  $(V, b, \psi)$. 
\end{remark}
\begin{remark}\label{Dim2}	Let $(V, b, \psi)$ be a cosymplectic vector space of dimension $(2n +1)$. Then, the $(2n +1)-$multi-linear map $B$, defined by 
	$B:=\psi\wedge \underbrace{b\wedge\cdots \wedge b}_{n-factors},$ is non-trivial. Indeed, by Proposition \ref{Dim}, let 
	$\xi, e_1,\cdots, e_n, f_1,\cdots, f_n$ be a basis of $V$ as in Theorem \ref{stand-0}, and set 
	$v_1 = \xi,$ $v_{i +1} = e_i$, and $v_{n + i + 1} = f_i$,   \qquad $ 1\leqslant i\leqslant n$. We have 
\begin{eqnarray}
B( v_1,\cdots, v_n, v_1,\cdots, v_{2n + 1}) &=&  (\psi\wedge b\wedge\cdots \wedge b)(v_1,\cdots, v_{2n + 1})\nonumber\\
&=& \sum_{\sigma\in S_{(2n +1)}}sgn(\sigma)\psi(v_{\sigma(1)})(b\wedge\cdots \wedge b)(v_{\sigma(2)},\cdots v_{\sigma(2n + 1)})\nonumber\\
&=& \sum_{\sigma\in S_{(2n +1)}, \sigma(1) = 1}sgn(\sigma)(b\wedge\cdots \wedge b)(v_{\sigma(2)},\cdots v_{\sigma(2n + 1)})\nonumber\\
&=& \varepsilon \cdot1\cdot 2\cdots (n-1)\cdot n
\end{eqnarray}
with $\varepsilon \in \{-1, 1\}$, where $S_{(2n +1)}$ stands for the set of all permutations of $ (2n +1)$ elements.  
\end{remark}

\subsection{The cosymplectic group} 
Let $F$ be a linear map form a  vector space $E$ into another vector space $V$ and let $\Omega$ be any $k-$multi-linear map on $V$. The pullback of
$\Omega$ by $F$, often denoted $F^\ast\Omega$, is the $k$-form on $E$ defined by:
$F^\ast\Omega(X_1,\cdots,X_k) = \Omega\big(F(X_1),\cdots,F(X_k)\big),$
for all $X_1,\cdots, X_k\in E$. \\ 

Given any cosymplectic vector space $(V, b, \psi)$ we shall denote by $\rm{Cosymp}(V, b, \psi)$ the following set: 
$$\rm{Cosymp}(V, b, \psi) := \big\{F\in GL(V): F^\ast b = b \ \text{and}\  F^\ast(\psi) = \psi\big\}.$$
Note that, under the composition of maps, $\big( \rm{Cosymp}(V, b, \psi), \circ \big)$ is a group  called the linear cosymplectic group. 
\subsection{Cosymplectic Manifolds}
Let $M$ be a smooth  manifold. An almost cosymplectic structure on $M$ is a pair $( \omega,\eta)$ consisting of a $2-$form  $\omega$ and a  $1-$form $\eta$  such that for each $x\in M$, the triple $ (T_xM, \omega_x, \eta_x)$ is an almost cosymplectic vector space.  Therefore, a cosymplectic structure on $M$ is any almost cosymplectic structure $( \omega,\eta)$ on $M$ such that $d\eta = 0$ and $d\omega = 0$. We shall write  $(M, \omega, \eta)$ to mean that $M$ has a cosymplectic structure $(\omega,\eta)$. For further details, we refer to the 
 \cite{P.L, H-L} and references therein.\\ 

In particular, Remark \ref{Dim2} tells us that any cosymplectic manifold $(M,\omega,\eta)$ is oriented with respect to the volume form 
$\eta\wedge\omega^n$, while by Remark \ref{Dim1} any cosymplectic manifold $(M,\omega,\eta)$ admits a vector field $\xi$ called the 
Reeb vector field such that $ \eta(\xi) = 1$, and $\imath_\xi\omega = 0$.

\begin{remark}
Not all odd dimensional manifolds has a cosymplectic structure. For instance, let $M^{(2k +1)}$ be any $(2k +1)-$dimensional closed manifold with $k\neq 0$ such that $H^\ast(M^{(2k +1)},\mathbb{R})$ 
denotes its $\ast-$th de Rham group with real coefficients. Therefore, if $H^1(M^{(2k +1)},\mathbb{R}) = 0$, or  $H^2(M^{(2k +1)},\mathbb{R})  = 0$, then  
$M^{(2k +1)} $ has no cosymplectic structure. In particular, since  $H^1(S^{(2k +1)},\mathbb{R}) = 0$, then the unit spheres $S^{(2k +1)} $ have no cosymplectic  structures, for any integer  $k$: A consequence of the usual Stockes' theorem. 
\end{remark}

\subsection{Cosymplectic Structure}\label{SC00} 

In order to  well describe in more details some features of  cosymplectic manifolds we shall need the following result due to H. Li \cite{H-L}. 

\begin{lemma}\label{lem-3}(\cite{H-L})
	Let $M$ be a manifold and $\eta$, $\omega$ be two differential forms on $M$ with degrees $1$ and $2$ respectively. Consider
	$\widetilde M = M\times \mathbb{R}$ equipped with the $2-$form\\ $\tilde\omega: = p^\ast(\omega) + p^\ast(\eta)\wedge \pi_2^\ast (du)$  where $u$ is the coordinate function on $\mathbb{R}$ and $p: \widetilde M\longrightarrow M,$ and $\pi_2: \widetilde M\longrightarrow \mathbb{R},$ are canonical projections. Then,  $(M,\omega, \eta)$ is a  cosymplectic manifold if  and only if  $(\widetilde M, \tilde\omega)$ is a  symplectic manifold.
\end{lemma}

\subsubsection{Cosymplectic Stability}
Gathering Lemma \ref{lem-3}  with the usual Moser stability result in symplectic geometry  setting, we claim and prove the following weak stability 
result in cosymplectic geometry setting. 
\begin{theorem}\label{thm00}
	Let $M$ be a smooth compact  $(2n +1)$-dimensional  manifold  admitting two cosymplectic structures  $( \omega_0,\eta_0)$ and $(\omega_1,\eta_1)$ with 
	$\eta_0$ (resp. $\omega_0$) cohomologous to $\eta_1$ (resp. $\omega_1$). Assume that for each $t\in [0,1]$, the map $I_{\eta_t,\omega_t}$ is an isomorphism over the module $C^\infty(M)$,  where $\eta_t = (1-t)\eta_0 + t\eta_1$ and $\omega_t = (1-t)\omega_0 + t\omega_1$. Then, there exist a smooth family of maps $\{\phi_t\}$ from $M$ to itself and a smooth family of real functions $\{f_t\}$ with $\phi_0 = id_M$ and $f_0 = 1$ such that 
	$\phi_t^\ast(\eta_t)= f_t\eta_0 $ and  $\phi_t^\ast(\omega_t)= \omega_0$, for all $t\in [0,1]$.
\end{theorem}
\begin{proof}  The proof of Theorem \ref{thm00}  is the adaptation of  the Moser path method in symplectic geometry  and  strongly rests on  Lemma \ref{lem-3}. 
	Consider {\color{cyan}$\widetilde M = M\times \mathbb{S}^1$} equipped with the $2$ form $\tilde\omega: = p^\ast(\omega) + p^\ast(\eta)\wedge \pi_2^\ast (du)$ where $u$ is the coordinate function on $\mathbb{R}$ and $p: \widetilde M\rightarrow M,$ and {\color{green}$\pi_2: \widetilde M\rightarrow \mathbb{S}^1$} 
	are the canonical projections on each factor of $\widetilde M$ respectively. Then, by  Lemma \ref{lem-3}, we have that $(\widetilde M, \tilde\omega)$ is a  symplectic manifold. Since the family  of closed $2-$forms $\tilde\omega_t: = p^\ast(\omega_t) + p^\ast(\eta_t)\wedge \pi_2^\ast (du)$ satisfies the assumption of the symplectic Moser's stability theorem (namely the forms $\tilde \omega_t $ defines the same cohomology class), then there exists an symplectic isotopy $\Psi=\{\psi_t\}$ such that $\psi_t^{\ast}(\tilde\omega_t) = \tilde\omega_0$, for each $t$. Now, since $p$ is surjective, if we {\color{blue} fix $l\in \mathbb{S}^1$} and denote by $S_l$ a smooth section of $p$ such that $S_l(x) = (x,l)\in\widetilde M $, for all $x\in M$ we can set  $\phi_t := p\circ \psi_t\circ S_l$ for each $t$. 
	Then $\phi_t$ is a smooth  map from $M$ into itself for each $t$. Compute, 
	\begin{eqnarray}
	\phi_t^\ast(\omega_t)  &= &S^\ast_l \big[\psi_t^\ast (p^\ast(\omega_t))\big] = S^\ast_l \big[\psi_t^\ast \big(\tilde\omega_t -p^\ast(\eta_t)\wedge \pi_2^\ast (du) \big)\big] \nonumber\\
	&= &S^\ast_l \big[\tilde\omega_0 - \psi_t^\ast \big(p^\ast(\eta_t)\wedge \pi_2^\ast (du) \big)\big] ,\nonumber\\
	&= &S^\ast_l(\tilde\omega_0) - S^\ast_l \big[\psi_t^\ast \big(p^\ast(\eta_t)\wedge \pi_2^\ast (du) \big)\big] ,\nonumber\\
	& =&  \omega_0, \quad \text{for each } \ t,
	\end{eqnarray}
	because  we have 
	\begin{eqnarray}
	S^\ast_l (\tilde \omega_0)  &= &
	S^\ast_l \big[p^\ast(\omega_0) + p^\ast(\eta_0)\wedge \pi_2^\ast (du) \big]  \nonumber\\
	&= & \omega_0 + \eta_0\wedge d(u\circ S_l) ,\nonumber\\
	&= & \omega_0+0,
	\end{eqnarray}
	and $ S^\ast_l\big[\psi_t^\ast (p^\ast(\eta_t)\wedge \pi_2^\ast (du) )\big]   = 
	S^\ast_l\psi_t^\ast (p^\ast(\eta_t))\wedge 0 = 0$. 
	On the other hand,   since   $p_\ast(\partial_u) = 0$, where $\partial_u := \frac{\partial}{\partial u}$, we derive from
	$ - p^\ast(\eta_t) = \imath_{\partial_u}\tilde\omega_t$, that \\
	$-\phi_t^\ast(\eta_t) :=( S_l^\ast\circ \psi_t^\ast \circ p^\ast)(\eta_t) =  S_l^\ast(\psi_t^\ast( \imath_{\partial_u}\tilde\omega_t))$ 
	and  compute, 
	\begin{eqnarray}
	\psi_t^\ast( \imath_{\partial_u}\tilde\omega_t) &=& \imath_{(\psi_t^{-1})_\ast(\partial_u)}\psi_t^\ast\tilde\omega_t\nonumber\\ &=& \imath_{(\psi_t^{-1})_\ast(\partial_u)}\tilde\omega_0,\nonumber\\
	& = & (\imath_{(\psi_t^{-1})_\ast(\partial_u)}p^\ast\omega_0) + 
	(\imath_{(\psi_t^{-1})_\ast(\partial_u)}p^\ast\eta_0)\wedge \pi_2^\ast (du)- p^\ast (\eta_0)\wedge  (\imath_{(p\circ\psi_t^{-1})_\ast(\partial_u)} \pi_2^\ast (du)),\nonumber
	\end{eqnarray}
	i.e., 
	$\psi_t^\ast( \imath_{\partial_u}\tilde\omega_t)  = (\imath_{(\psi_t^{-1})_\ast(\partial_u)}p^\ast\omega_0) + (\imath_{(\psi_t^{-1})_\ast(\partial_u)}p^\ast\eta_0)\wedge \pi_2^\ast (du) 
	- p^\ast (\eta_0)\wedge  (\imath_{(p\circ\psi_t^{-1})_\ast(\partial_u)}\pi_2^\ast (du)).$ 
	Composing  the above equality with $ S_l^\ast$ gives \\
	$ \phi_t^\ast(\eta_t) =	S_l^\ast (\imath_{(\psi_t^{-1})_\ast(\partial_u)}\pi_2^\ast (du)) \eta_0 - S_l^\ast(\imath_{(\psi_t^{-1})_\ast(\partial_u)}p^\ast\omega_0)$. 
	Hence, defining two smooth families of mappings $\{h_t\}$ and $\{k_t\}$ such that  $h_t\circ p = p\circ \psi_t^{-1}$, and $k_t\circ \pi_2 = \pi_2\circ \psi_t^{-1}$ for each $t$, then we derive that 
	\begin{eqnarray}
	-\phi_t^\ast(\eta_t) &=& S_l^\ast(\imath_{(h_t\circ p)_\ast(\partial_u)}\omega_0) -S_l^\ast (\imath_{(\psi_t^{-1})_\ast(\partial_u)} \pi_2^\ast (du)) \eta_0,\nonumber\\
	& = & 0 - (\pi_2\circ S_l)^\ast (\imath_{(k_t)_\ast(\partial_u)} du) \eta_0,\nonumber\\
	&= & - (\imath_{(k_t)_\ast(\partial_u)} du)(l) \eta_0.
	\end{eqnarray}
	Take $ f_t:= \imath_{(k_t)_\ast(\partial_u)} du(l)$,  for each $t$. 
	It is clear that $f_0 =  (\imath_{(k_0)_\ast(\partial_u)} du)(l) = 1$. 
\end{proof}
\begin{theorem}[$\omega-$Stability theorem]\label{thm01} 
	Let $M$ be a smooth compact manifold  of dimension $(2n +1)$, admitting two cosymplectic structures  $(\eta, \omega_0)$ and $(\eta, \omega_1)$ with 
	$\omega_0$ cohomologous to  $\omega_1$. Assume that for each $t\in [0,1]$, $I_{\eta,\omega_t}$ is an isomorphism 
	over the module $C^\infty(M)$ 
	where $\omega_t = (1-t)\omega_0 + t\omega_1$. Then, there exists a smooth isotopy
	$\Phi=\{\phi_t\}$  such that $\phi_t^\ast(\omega_t)= \omega_0$ and $\phi_t^\ast(\eta)= \eta$, 
	for all $t\in [0,1]$.
\end{theorem}
\begin{proof}
	We shall adapt the proof of similar result from symplectic geometry. Suppose that there exists a smooth isotopy
	$\{\phi_t\}$ from $M$ to itself such that $\phi_t^\ast(\omega_t)= \omega_0$  and $\phi_t^\ast(\eta)= \eta$, 
	for all $t\in [0,1]$. This implies that if $v_t(\phi_t(x)):= \frac{d}{dt}\phi_t(x)$, for each $x\in M$, then we must have 
	$\mathcal{L}_{v_t}(\omega_t) + \frac{d}{d t}\omega_t = 0$ and $ \mathcal{L}_{v_t}(\eta) = 0$. This implies that, for all $t\in [0,1]$ 
	\begin{equation}\label{Stab-0}
	\imath_{v_t}\omega_t + \alpha = 0,
	\end{equation}
	where $\alpha$ is such that $\omega_1- \omega_0 = d\alpha$ and 
	\begin{equation}\label{Stab-1}
	\eta(v_t) = 0.
	\end{equation}
	Conversely, suppose that one can find a smooth family of vector fields $ \{v_t\}$ which satisfies (\ref{Stab-0}) and (\ref{Stab-1}), then its generating isotopy 
	$\{\psi_t\}$ will satisfy $\psi_t^\ast(\omega_t)= \omega_0$, and $\psi_t^\ast(\eta)= \eta$, 
	for all $t\in [0,1]$. So, it will be enough to solve (\ref{Stab-0}) and (\ref{Stab-1}). To that end, we observe that  (\ref{Stab-0}) suggests that
	if $\xi$ is the Reeb vector field, then $\alpha(\xi) = 0$ and with this information,  the equations (\ref{Stab-0}) and (\ref{Stab-1}) are equivalent to :
	\begin{equation}\label{Stab-2}
	\widetilde  I_{\eta, \omega_t}(v_t) = -\alpha,
	\end{equation}
	for all $t\in [0,1]$. From the non-degeneracy of $\widetilde  I_{\eta, \omega_t} $, one easily  solves (\ref{Stab-2}) to obtain $  \{v_t\}$. 
\end{proof}

\begin{theorem}[$\eta-$Stability theorem]\label{thm02}
	Let $M$ be a $(2n +1)$-dimensional  smooth compact manifold  admitting two cosymplectic structures  $(\omega,\eta_0)$ and $(\omega,\eta_1)$ with 
	$\eta_0$ cohomologous to  $\eta_1$. Assume that for each $t\in [0,1]$, $\widetilde I_{\eta_t,\omega}$ is an isomorphism over the module $C^\infty(M)$ 
	where $\eta_t = (1-t)\eta_0 + t\eta_1$. Then, there exists a smooth isotopy $\{\phi_t\}$  such that $\phi_t^\ast(\omega)= \omega$ and 
	$\phi_t^\ast(\eta_t)= \eta_0$, for all $t\in [0,1]$.
\end{theorem}
\begin{proof}We shall adapt the proof of Theorem \ref{thm01}. In fact here, $v_t := \widetilde I_{\eta_t, \omega}^{-1}(-f\eta_t)$, with\\ $\eta_1 -\eta_0 = df$ \, for some smooth function $f$. \end{proof}
\begin{theorem}[Stability theorem $I$]\label{thm03} 
	Let $M$ be a smooth compact manifold  of dimension $(2n +1)$, admitting two cosymplectic structures  $( \omega_0,\eta_0)$ and $( \omega_1,\eta_1)$ with 
	$\eta_0$ (resp. $\omega_0$) cohomologous to $\eta_1$ (resp. $\omega_1$). Assume,  for each $t\in [0,1]$,that  $\widetilde I_{\eta_t,\omega_t}$ is an isomorphism 
	over the module $C^\infty(M)$ where $\eta_t = (1-t)\eta_0 + t\eta_1$ and $\omega_t = (1-t)\omega_0 + t\omega_1$. Then, there exists a smooth isotopy
	$\{\phi_t\}$ such that $\phi_t^\ast(\omega_t)= \omega_0$ and $\phi_t^\ast(\eta_t)= \eta_0$, for all $t\in [0,1]$. 
\end{theorem}
\begin{proof} 
	Suppose that there exists a smooth isotopy $\{\phi_t\}$ from $M$ to itself such that\\ $\phi_t^\ast(\omega_t)= \omega_0$, and $\phi_t^\ast(\eta_t)= \eta_0$, 
	for all $t\in [0,1]$. This implies that if $v_t(\phi_t(x)):= \frac{d}{dt}\phi_t(x)$, for each $x\in M$, then we must have $\mathcal{L}_{v_t}(\omega_t) + \frac{d}{d t}\omega_t = 0$, and $ \mathcal{L}_{v_t}(\eta_t) +  \frac{d}{d t}\eta_t = 0$. This implies that, 
	\begin{equation}\label{Stab-12}
	\eta_t(v_t) + f = 0,
	\end{equation}
	and 
	\begin{equation}\label{Stab-02}
	\imath_{v_t}\omega_t + \alpha = 0,
	\end{equation}
	where $\omega_1- \omega_0 = d\alpha$, for all $t\in [0,1]$. 
	Conversely, suppose that one can find a smooth family of vector fields $ \{v_t\}$ which satisfies (\ref{Stab-02}) and (\ref{Stab-12}), then its generating isotopy 
	$\{\psi_t\}$ will satisfy $\psi_t^\ast(\omega_t)= \omega_0$ and $\psi_t^\ast(\eta_t)= \eta_0$, for all $t\in [0,1]$. So, it will be enough to solve (\ref{Stab-02}) and (\ref{Stab-12}). First of all, note that from (\ref{Stab-02}), it follows that :  if $\xi_t := \widetilde I_{\eta_t, \omega_t}^{-1}(\eta_t)$, then $\alpha(\xi_t) = 0$, for each $t$. This condition suggests that (\ref{Stab-02}) and (\ref{Stab-12}) are equivalent to :
	\begin{equation}\label{Stab-22}
	\widetilde I_{\eta_t, \omega_t}(v_t) = -\alpha - f\eta_t,
	\end{equation}
	for all $t\in [0,1]$. From the non-degeneracy of $ \widetilde I_{\eta_t, \omega_t} $, one can solve (\ref{Stab-22}) to obtain $  \{v_t\}$. 
\end{proof}
\begin{theorem}[Stability theorem $II$]\label{thm4}
	Let $M$ be a smooth closed manifold  of dimension $(2n +1)$, admitting two cosymplectic structures  $( \omega_0,\eta_0)$ and $(\omega_1,\eta_1)$. Assume that  $\{\eta_t\}$ (resp. $\{\omega_t\}$) is a  smooth family of closed $1-$forms (resp. closed $2-$forms) with endpoints $\eta_0$ and $\eta_1$ 
	(resp. $\omega_0$ and $\omega_1$ )  making $(M, \omega_t,\eta_t)$  a cosymplectic manifold for each $t$, and satisfying
	\begin{itemize}
		\item $\dfrac{\partial}{\partial t}[\omega_t ] =\left[\dfrac{\partial}{\partial t}\omega_t \right] =  [d\alpha_t ] = 0$, 
		\item $\dfrac{\partial}{\partial t}[\eta_t ] =\left[\dfrac{\partial}{\partial t}\eta_t\right]  =[df_t] = 0$, and 
		\item $d(\alpha_t(\xi_t))  = 0$,  
	\end{itemize}
	for each $t$, with $\xi_t: = \widetilde  I_{\eta_t,\omega_t}^{-1}(\eta_t)$. 
	Then,  there exists a smooth isotopy
	$\{\phi_t\}$ such that\\ $\phi_t^\ast(\omega_t) = \omega_0$, and $\phi_t^\ast(\eta_t)= \eta_0$, 
	for all $t\in [0,1]$.
\end{theorem}
\begin{proof} From $\dfrac{\partial}{\partial t}\left[\eta_t \right] =\left[\dfrac{\partial}{\partial t}\eta_t \right] = 0$, we derive that $ \dfrac{\partial}{\partial t}\eta_t  = df_t$, where $f_t$, is a smooth family of smooth functions on $M$, whereas  $\dfrac{\partial}{\partial t}\omega_t =  d\alpha_t$, for each $t$. So, set $\xi_t: = \widetilde I_{\eta_t,\omega_t}^{-1}(\eta_t)$, for each $t$, and from the non-degeneracy of $ \widetilde I_{\eta_t, \omega_t}$, 
	there exists a unique family of vector field $\{v_t\}$ such that 
	\begin{equation}\label{Stab-33}
	\widetilde I_{\eta_t, \omega_t}(v_t) + \alpha_t + f_t\eta_t= 0,
	\end{equation}
	for each $t$. Note that with this assumption, it is not hard to see that for such vector field $\{v_t\}$, applying $ \xi_t$ in both side of  (\ref{Stab-33}) implies that 
	$ \eta_t(v_t) + \alpha_t(\xi_t) + f_t = 0 .$  So, 
	if $\{\rho_t\}$ is the generating isotopy of $\{v_t\}$, then we have\\ 
	$\frac{d}{dt} (\rho_t^\ast(\eta_t)) = \rho_t^\ast(d(\eta_t(v_t)) + df_t) = -\rho_t^\ast(d(\alpha_t(\xi_t))) = 0,$ and\\  
		$\frac{d}{dt} (\rho_t^\ast(\omega_t)) = \rho_t^\ast(d\imath_{v_t}\omega_t + d\alpha_t) = -\rho_t^\ast(d(\alpha_t(\xi_t))\wedge\eta_t)  = 0,$
	for each $t$. 
\end{proof}
\subsection{Cosymplectic Darboux theorem}\label{Darboux-lem}
\begin{lemma}\label{Darboux1}
	Let $M$ be a smooth manifold  of dimension $(2n +1)$. Let $Q$ be a compact sub-manifold of  $M$, let $\eta_0$, $\eta_1\in \mathcal{Z}^1(M)$ and $\omega_0$, $\omega_1\in \mathcal{Z}^2(M)$ such that $(\omega_0,\eta_0)$ and $(\omega_1,\eta_1)$ induce two cosymplectic structures on a neighborhood of $Q$ with $\omega_0 (q) =\omega_1(q)$ and $\eta_0 (q) =\eta_1(q)$, for all $q\in Q$. Then, there exist open neighborhoods $\mathcal{U}_0$ and  $\mathcal{U}_1$ of $Q$ and a local diffeomorphism $\psi:\mathcal{U}_0\longrightarrow\mathcal{U}_1$  such that $\psi^\ast(\omega_1) = \omega_0$, 
	$\psi^\ast(\eta_1) = \eta_0$ and $\psi_{|Q} = id$. 
\end{lemma}
\begin{proof} We shall adapt similar proof from symplectic geometry. 
	Let $\eta_0$, $\eta_1\in \mathcal{Z}^1(M)$ and $\omega_0, \omega_1\in \mathcal{Z}^2(M)$ be as in the above assumption. We want to show that there exists a neighborhood $\mathcal{U}$ of $Q$ on which we have $ \omega_1-\omega_0 = d\alpha,$ and $ \eta_1-\eta_0 = df,$ where $\alpha$ is a $1-$form on $M$, and $f$ is a smooth function on $M$.  To do so, we shall adapt the technique used to prove similar result from symplectic geometry. 
	\begin{itemize}
		\item {\bf Step (1):} Fix a Riemannian metric $g$ on $M$, and identify the normal bundle $\mathbf{N}Q$ of $Q$ with the orthogonal complement $Q^{\perp}$. Then, the exponential map $exp: \mathbf{N}Q\longrightarrow M$ is a diffeomorphism on some 
		$\mathbf{B}_\delta :=\big\{(q, v)\in \mathbf{N}Q: \| v\|_g<  \delta\big\},$ 
		for $\delta$ sufficiently small. We set $\mathcal{U} : =  exp(\mathbf{B}_\delta) $ (such n $\delta > 0$ exists because $Q$ is compact). 
		For each $0\leqslant t \leqslant 1$, define $\phi_t(exp((q, v))) :=  exp((q, tv)).$ For $t> 0$, we have that $\phi_t$ realizes a diffeomorphism 
		from $\mathcal{U}$, onto its image in  $\mathcal{U}$. Moreover, we have $ \phi_1 = id$, $\phi_0(\mathcal{U}) = Q$, and $\phi_t$ restricted to $Q$ is the 
		identity of $Q$. Set 
		$\tau := \omega_1 - \omega_0,$ and 
		$\sigma := \eta_1 -\eta_0,$ and for each $0< t \leqslant 1$, define
		$Y_t(\phi_t(x)) := \dfrac{d}{d t}(\phi_t(x)),$
		for all $x\in M$. Now, for any $s> 0$, compute 
		$$\phi_1^\ast(\tau) -  \phi_s^\ast(\tau) = \int_s^1 \frac{d}{d u} (\phi_u^\ast(\tau))du = d\left(\int_s^1 \phi_u^\ast (\iota(Y_u)\tau) du \right) =: d\alpha_s.$$
		Similarly, one obtains 
		$$\displaystyle \phi_1^\ast(\sigma) -  \phi_s^\ast(\sigma) = \int_s^1 \frac{d}{d u} (\phi_u^\ast(\sigma))du = d\left(\int_s^1\sigma(Y_u)\circ\phi_u du \right) =: df_s.$$
		For vector fields $u$ and $ v$ on $\mathcal{U}$, we have 
		$$ \phi_s^\ast(\tau)(u,v) = \tau(d\phi_s(u), d\phi_s(v))\circ \phi_s  \longrightarrow \tau(d\phi_0(u), d\phi_0(v))\circ \phi_0 = 0,$$ 
		as $s\longrightarrow 0^+$ since   $\phi_0(\mathcal{U}) = Q$, and $\tau = 0$ on each $ T_qM$,  for each $q\in Q$. This implies that 
		$\tau = d\left( (\phi^{-1}_1)^\ast(\alpha_0)\right) =: d\beta,$ and similarly, one obtains $\sigma = d \left( f_0\circ\phi^{-1}_1\right).$
		
		\item {\bf Step (2):}
		Note that the restrictions of $\tau$ and $\sigma$ to $ T_qM$ are trivial. For instance,
		set  $\eta_t = \eta_0 + t (\eta_1 - \eta_0)$, and $\omega_t = \omega_0 + t (\omega_1 - \omega_0)$ for each $t$ and derive\\ that,
		$\frac{\partial}{\partial t}\omega_t  = d\beta \quad \text{and}\quad  \frac{\partial}{\partial t}\eta_t = \sigma.$ 
		For each  $t$, consider $I_t:= \widetilde I_{\eta_0,\omega_0} + t (\widetilde I_{\eta_1,\omega_1} - \widetilde I_{\eta_0,\omega_0}).$  Since non-degeneracy is an open condition, then $I_t$
		is non-degenerate  in a smaller neighborhood of $Q$: for each $t$, the couple $(\eta_t, \omega_t)$, induces a cosymplectic structure on that neighborhood. 
		We have to show that if 
		$ \xi_t: = I_t^{-1}(\eta_t)$, then $d(\beta(\xi_t)) = 0$.
		From $ \imath_{\xi_t}\omega_t = 0$, we derive that $td(\beta(\xi_t)) = - \mathcal{L}_{\xi_t}\omega_0 $, for each $t$. 
		We claim that for $0< t\leqslant1$, we 
		have $\mathcal{L}_{\xi_t}\omega_0 = 0.$\\
		{\sf Proof of the claim }:  Assume by contradiction that there exists $s\in ]0, 1]$ such that 
		$\mathcal{L}_{\xi_s}\omega_0  \neq 0$. In particular, since the two forms $\mathcal{L}_{\xi_s}\omega_0$, and  $\mathcal{L}_{\xi_s}\omega_1$ 
		agree on $Q$, then we could have  $\mathcal{L}_{\xi_s}\omega_1 + (1-s)\mathcal{L}_{\xi_s}\omega_0  \neq 0,$ 
		on a neighborhood of $Q$. 
		That is, 
		\begin{equation}\label{Lie-1}
		0=  \mathcal{L}_{\xi_s}\omega_s = s\mathcal{L}_{\xi_s}\omega_1 + (1-s)\mathcal{L}_{\xi_s}\omega_0 \neq 0,
		\end{equation}
		on that neighborhood of $Q$ since $\xi_s$ is a cosymplectic vector field on a neighborhood of $Q$. Formula (\ref{Lie-1}) yields a contradiction then the claim follows. From   $\eta_t = \eta_0 + t (\eta_1 - \eta_0)$ and $\omega_t = \omega_0 + t (\omega_1 - \omega_0)$,  it follows from  the above arguments together with {\bf step (2)} that, $\dfrac{\partial}{\partial t}\omega_t  = d\beta$, 
	 $\dfrac{\partial}{\partial t}\eta_t = \sigma $,  
			and  $d(\beta(\xi_t)) = 0$, for each $t$. From the stability theorem (Theorem \ref{thm4}), we can find a smooth family of vector field $\{X_t\}_t$ defined on a neighborhood of $Q$ such that 
		$ I_t(X_t) + \beta + (f_0\circ\phi^{-1})\eta_t = 0, \qquad \text{ for each } \quad t.$
		\item {\bf Step (3):}
		It is enough to show that there exists a small neighborhood $\mathcal{U}_0$ of $Q$ containing $\mathcal{U}$, on which the
		family of diffeomorphisms $\psi_t$ defined by the ODE:\\  $\dot{\psi}_t = X_t\circ\psi_t,$ and $\psi_0 = id,$ is well-defined.
		By contradiction, assume that no such  neighborhood exists. Then there would exist  sequences $\{x_i\}\subset \mathcal{U}$ and $\{s_i\}\subset [0, 1)$  such that: $\psi_t(x_i)$ is defined for 
			$t\in [0, 1]$,  $\psi_{s_i}(x_i)\in \partial \mathcal{U}$, and as $i\longrightarrow \infty$,  
			$x_i$ belongs to the closure of $Q$. 
		Thus, a subsequence  $\{q_i\}$  of $\{x_i\}$ 
		converges to $q_0\in Q$, and  a subsequence $\{t_i\}$ of $\{s_i\}$ converges to $t_0 \in [0,1]$, whereas the orbit $[0, t_i]\ni t \mapsto\psi_t(q_i)$ converges 
		to the orbit $[0, t_0]\ni t \mapsto\psi_t(q_0)$. Since for each  $t\in [0, t_0]$, we have that $\psi_{t}(q_0)\in \partial \mathcal{U}$ which is compact, then $\psi_{t_0}(q_0)\in \partial \mathcal{U}$. This is a contradiction, because $X_t$ being trivial on $Q$, imposes that $\psi_t(q_0) = q_0$, for 
		each $t\in [0, t_0]$. Therefore, we take $\mathcal{U}_1 = \psi_1( \mathcal{U}_0)$, and $\psi := \psi_1$. 
	\end{itemize}
\end{proof}

\begin{theorem}(Cosymplectic-Darboux theorem)\label{thm1}. 
	On a cosymplectic manifold $(M, \omega,\eta)$ of dimension $(2n +1)$, any point $x\in M$ has a local coordinate system $(z, x_1,\cdots x_n, y_1, \cdots, y_n)$ in which $\eta$ (resp.  $\omega$) agrees with:  $\bar \eta := dz$ (resp. $\displaystyle \bar\omega := \sum_{i=1}^n dx_i\wedge dy_i$). 
\end{theorem}
\begin{proof} Since for each $q\in M$, we have that $(T_qM, \omega_q, \eta_q)$  is a cosymplectic vector space, we fix a Riemannian metric near $q$
	and  choose a basis $\xi_q, e_1,\cdots, e_n, f_1,\cdots, f_n$ of $T_qM$ as in Theorem \ref{stand-0}. Then, we define \\ 
	$\displaystyle \phi: \mathbb{R}^{(2n +1)}\longrightarrow M, (z, x_1,\cdots, x_n, y_1,  y_2,\cdots, y_n) 
	\mapsto exp_q(z \xi_q + \sum_{i= 1}^{n}\left( x_ie_i + y_if_i\right)).$\\
	The map $\phi$ is a diffeomorphism from a neighborhood $\mathcal{U}$ of zero in $ \mathbb{R}^{(2n +1)}$ onto a neighborhood 
	of $q\in M$: Thus, $(\mathcal{V}, \tilde\phi)$  defines a chart centered at $q$, with $ \tilde \phi :=\phi^{-1}$, and 
	$\mathcal{V} := \phi(\mathcal{U})$. Since the differential of the map $exp_q$  at the
	origin is the identity map and  the  basis $\xi_q, e_1,\cdots, e_n, f_1,\cdots, f_n$ is chosen as in Theorem \ref{stand-0}, it follows that 
	$(\tilde\phi^\ast(\bar\eta), \tilde\phi^\ast(\bar \omega))$ and  $(\eta, \omega)$ are two cosymplectic structures on $\mathcal{U}$  such that 
	$(\tilde\phi^\ast(\bar\eta))_q$  and $\eta_q$ (resp.   $(\tilde\phi^\ast(\bar \omega))_q$ and $\omega_q$) agrees on $T_qM$.  The theorem follows
	immediately by applying Lemma \ref{Darboux1} with $Q = \{q\}$. This completes the proof. 
\end{proof}
\subsection{The $C^0-$topology}
Let $Homeo(M)$ denote the group of all homeomorphisms of $M$ equipped with the $C^0-$ 
compact-open topology. This is the 
metric topology induced by the following distance
\begin{equation}
	d_0(f,h) = \max(d_{C^0}(f,h),d_{C^0}(f^{-1},h^{-1})),
\end{equation}
where 
\begin{equation}
	d_{C^0}(f,h) =\sup_{x\in M}d (h(x),f(x)).
\end{equation}
On the space of all continuous paths $\lambda:[0,1]\rightarrow Homeo(M)$ such that $\lambda(0) = id_M$, 
we consider the $C^0-$topology as the metric topology induced by the following metric  
\begin{equation}
	\bar{d}(\lambda,\mu) = \max_{t\in [0,1]}d_0(\lambda(t),\mu(t)).
\end{equation}
\section{Vector fields of a cosymplectic manifold}\label{SC00-1}
Let $\Omega^1(M)$ (resp. $\chi(M)$ ) be the space of all $1-$forms (resp. of all smooth vector fields) of a cosymplectic manifold $(M, \omega,\eta)$. The cosymplectic structure induces an isomorphism of $C^\infty(M,\mathbb{R})-$modules 
\begin{eqnarray}\begin{array}{ccclcc}
\widetilde I_{\eta, \omega}: & \chi(M) &\longrightarrow &\Omega^1(M) \\
& X &\longmapsto & \widetilde I_{\eta, \omega}(X)=\eta(X)\eta + \imath_X\omega.\nonumber
\end{array}\end{eqnarray}
The vector field $\xi:= \widetilde I_{\eta, \omega}^{-1}(\eta)$ is called the Reeb vector field of $(M, \omega,\eta)$ and is characterized by :  
$ \eta(\xi) = 1$ and $\imath_\xi\omega = 0$.
\begin{proposition}[$( \omega,\eta)-$decomposition)]\label{dec1}
	Let $(M, \omega,\eta)$ be a  cosymplectic manifold. Then, any vector field $X$ on $M$ decomposes in a unique way as : 
	$ X = X_\omega + X_\eta,$ 
	where  $$X_\omega := \widetilde I_{\eta, \omega}^{-1}(\imath_X\omega) \quad  \text{ and } \quad X_\eta :=  \widetilde I_{\eta, \omega}^{-1}(\eta(X)\eta).$$
\end{proposition}
\subsection{Cosymplectic vector fields}\label{SC011} 
In this subsection we study those vector fields $X$ of a  cosymplectic manifold $(M, \eta, \omega)$ whose generating flow $ \Phi_X$ preserve
 the forms $\eta$, and $\omega$. 

\begin{definition}
	Let $(M, \omega,\eta)$ be a  cosymplectic manifold. A vector field $X$ is said to be  cosymplectic if $\mathcal{L}_{X}\eta = 0$
	and  $\mathcal{L}_{X}\omega = 0$.
\end{definition}
We shall denote by $\chi_{\eta, \omega}(M)$ the space of all  cosymplectic vector fields of $(M, \omega,\eta)$.
\begin{corollary}\label{cor1}
	Let $(M,\omega,\eta)$ be a  cosymplectic manifold. For any $X\in \chi_{\eta, \omega}(M)$, the  $1-$form 
	$ \widetilde I_{\eta, \omega}(X)$ is closed. 
\end{corollary}

We have the following fact.

\begin{lemma}\label{flux-geo}
	Let $(M, \eta, \omega)$ be a cosymplectic manifold.  Consider the symplectic manifold $\widetilde M = M\times \mathbb{R}$ equipped with the symplectic form $\Omega: = p^\ast(\omega) + p^\ast(\eta)\wedge \pi_2^\ast (du)$ where $u$ is the coordinate function on $\mathbb{R}$, $p: \widetilde M\rightarrow M,$ and $\pi_2: \widetilde M\rightarrow \mathbb{R}$ are projection maps. Let $\alpha$ be any closed $1-$form on $M$, and set $ X_\alpha := \tilde\Omega^{-1}(p^\ast(\alpha))$, where $\tilde\Omega$ is the isomorphism induced by the symplectic form from the space of all vector fields on $\widetilde{M}$ onto the space of all $1-$forms on $\widetilde{M}$. Then, the vector field $Y_\alpha : = p_\ast(X_\alpha)$ is a cosymplectic vector field, if and only if, $  d \left(  \left( du((\pi_2)_\ast(X_\alpha) \right)(1)\right) = \alpha(\xi)\eta$, where $\xi$ is the Reed vector field of $(M, \eta, \omega)$. 
\end{lemma}
\begin{proof}
	Consider the symplectic manifold $\widetilde M = M\times \mathbb{R}$ equipped with the symplectic form $\Omega: = p^\ast(\omega) + p^\ast(\eta)\wedge \pi_2^\ast (du)$ where $u$ is the coordinate function on $\mathbb{R}$, $p: \widetilde M\rightarrow M,$ and\\ $\pi_2: \widetilde M\rightarrow \mathbb{R}$, are projection maps. Let $\alpha$ be any closed $1-$form on $M$, and put\\ $ X_\alpha := \tilde\Omega^{-1}(p^\ast(\alpha))$. 
	We compute $$ \iota(X_\alpha)\Omega = p^\ast\left(\iota( p_\ast(X_\alpha))\omega \right) +  p^\ast\left(\iota( p_\ast(X_\alpha))\eta \right) \pi_2^\ast(du) - \left( du((\pi_2)_\ast(X_\alpha))\right) \circ\pi_2 p^\ast(\eta),$$
	and since by definition we have 
	$\iota(X_\alpha)\Omega = p^\ast(\alpha) $, then we derive that 
	$$ p^\ast\left(\iota( p_\ast(X_\alpha))\omega \right) +  p^\ast\left(\iota( p_\ast(X_\alpha))\eta \right) \pi_2^\ast(du) - \left( \left( du((\pi_2)_\ast(X_\alpha))\right) \circ\pi_2\right)  p^\ast(\eta) = p^\ast(\alpha).$$
	We apply the vector field $\frac{\partial}{\partial u}$ to both sides of the above equality to obtain:\\  
	$ 0 +  p^\ast\left(\iota( p_\ast(X_\alpha))\eta \right) - 0 = 0$ since $p_\ast(\frac{\partial}{\partial u}) = 0$, and 
	$(\pi_2)_\ast(\frac{\partial}{\partial u}) = \frac{\partial}{\partial u}$: this gives $ \eta(Y_\alpha) = 0$. 
	On the other hand, fix $l\in \mathbb{R}$, and let $S_l$ be the corresponding section of the projection $p$. 
	Composing the  equality 
	$$ p^\ast\left(\iota( p_\ast(X_\alpha))\omega \right) +  p^\ast\left(\iota( p_\ast(X_\alpha))\eta \right) \pi_2^\ast(du) - \left( du((\pi_2)_\ast(X_\alpha))\right) \circ\pi_2 p^\ast(\eta) = p^\ast(\alpha),$$
	in both sides with respect to $S_l^\ast$, yields: 
	\begin{equation}\label{Const-1}
	\iota(Y_\alpha)\omega   - \left( du((\pi_2)_\ast(X_\alpha))\right)(l) \eta = \alpha.
	\end{equation}
	Therefore, it follows from (\ref{Const-1}) that $\left( du((\pi_2)_\ast(X_\alpha)\right)(l) = \alpha(\xi)$,\\ and so 
	$d\left( \iota(Y_\alpha)\omega\right) =    d\left( \left( \big( du((\pi_2)_\ast(X_\alpha))(l)\big) \right) \eta  + \alpha\right) = 0,$ whenever\\   
	$ d \big( \left( du((\pi_2)_\ast(X_\alpha))\right)(l)\big) = \alpha(\xi)\eta$:
	 which implies that $\mathcal{L}_{Y_\alpha}\omega = 0$, and  $\mathcal{L}_{Y_\alpha}\eta = 0 $. Conversely, if $Y_\alpha$ is cosymplectic, then from (\ref{Const-1}) we derive that $ d \big( \left( du((\pi_2)_\ast(X_\alpha)\right)(l) \eta\big) = 0$, which implies that  
	 $  du((\pi_2)_\ast(X_\alpha))(l)\big) = \alpha(\xi)\eta$. 
\end{proof}
However, we do not know whether for any $\alpha\in \mathcal{Z}^1(M)$, the vector 
field $X := \widetilde I_{\eta, \omega}^{-1}(\alpha)$ is a cosymplectic vector field or not. Due to that fact, let us consider the set $C^{ste}(M)$  consists of all constant function on $M$, and put 
\begin{equation}
\mathcal{Z}^1_\xi(M) : = \big\{\beta\in \mathcal{Z}^1(M) \ ; \ \beta(\xi)\in C^{ste}(M)\big\}	
\end{equation}
where $\xi$ is the Reeb vector field. This set is non-empty, since from $\eta(\xi) = 1$, we derive that $\eta \in \mathcal{Z}^1_\xi(M)$. Also, 
for any vector field $X$ on $M$ such that $d( \imath_X \omega) = 0 $, we have 
$( \imath_X \omega)(\xi) = 0,$
i.e., $ \imath_X \omega \in \mathcal{Z}^1_\xi(M) $.
\begin{proposition}\label{0pro-1}
	Let $(M, \omega,\eta )$ be a compact cosymplectic manifold. Let $\alpha\in \mathcal{Z}^1_\xi(M)$ and let
	$X := \widetilde I_{\eta, \omega}^{-1}(\alpha)$. If $\{\psi_t\}$ is the flow generated by $X$ then for each $t$
	we have : $\psi^\ast_t(\omega) = \omega$ and $\psi^\ast_t(\eta) = \eta$.
\end{proposition}
\begin{proof} Since $d \widetilde I_{\eta, \omega}(X) = d\alpha = 0$, then   $\mathcal{L}_X(\omega) = -\mathcal{L}_X(\eta)\wedge\eta$. But, 
	$\imath_X \omega + \eta(X)\eta = \alpha$, and $\alpha\in \mathcal{Z}^1_\xi(M)$ imply that $\eta(X) = \alpha(\xi) = cte$. So, 
	$d(\eta(X)) = d(\alpha(\xi)) = 0$ and this implies that $\mathcal{L}_X(\eta) =  d(\eta(X)) = 0$, and 
	$\mathcal{L}_X(\omega) = -\mathcal{L}_X(\eta)\wedge\eta = 0$. This concludes the proof. 
\end{proof}

\begin{remark}\label{flux-geo-rem} The map 
	\begin{eqnarray}\begin{array}{ccccccccc}\label{0pro} 
	\sharp_{\eta, \omega} : &   \chi_{\eta, \omega}(M) &\longrightarrow  &\mathcal{Z}^1_\xi(M)\\
	&  X &\longmapsto &\widetilde I_{\eta, \omega}(X)
	\end{array}\nonumber\end{eqnarray}
	is a linear isomorphism. Furthermore, for each closed $1-$form $\alpha$, the vector field\\ $Z_\alpha: = Y_\alpha -\alpha(\xi)\xi$, satisfies 
	$  \widetilde I_{\eta, \omega}(Z_\alpha) = \alpha$,   $\mathcal{L}_{Z_\alpha}\eta = d \left(  \iota(Y_\alpha)\eta  - \alpha(\xi)\iota(\xi)\eta\right) = 0$, and\\ $\mathcal{L}_{Z_\alpha}\omega =  d\left( \alpha(\xi) \right)$.\\ 
\end{remark}

Here is a consequence of the $(\omega, \eta)-$decomposition of vector fields on a cosymplectic manifold.

\begin{proposition}\label{Decom-0}
	Let $X $ be a cosymplectic vector field with $(\omega, \eta)-$decomposition\\ $ X = X_\omega + X_\eta$. Then,
	\begin{enumerate}
		\item the vector fields $ X_\omega$ and $ X_\eta $ are cosymplectic,
		\item  when it exists, the flow $\Phi_\omega = \{\phi_\omega^t\}$ (resp. $\Phi_\eta =  \{\phi_\eta^t\}$) generated by the vector field 
		$ X_\omega $ 
		(resp. $ X_\eta $) preserves the cosymplectic structure,
		\item $[ X_\omega , X_\eta ] = 0 $, 
		\item $ \phi_\eta^s \circ \phi_\omega^t = \phi_\omega^s \circ \phi_\omega^t$, for each $s, t$, 
		\item the flow $\Phi$ generated by $X$ decomposes as: $\Phi= \Phi_\omega\circ\Phi_\eta$, and $\Phi= \Phi_\eta\circ\Phi_\omega$, and 
		\item for any $Y\in \chi_{\eta, \omega}(M)$, we have $[X,Y]\in \chi_{\eta, \omega}(M)$.
	\end{enumerate}
\end{proposition}

\begin{proof} For $(1)$, let $X\in \chi_{\eta, \omega}(M)$. Since $\mathcal{L}_X  \omega = 0 $ and $ X = X_\omega + X_\eta$, then  
		$ \mathcal{L}_{X_\eta}\omega = -\mathcal{L}_{X_\omega}\omega.$ 
		We claim that $ \mathcal{L}_{X_\eta}\omega  = 0$. In fact, since $\widetilde I_{\eta, \omega}(X_\omega) = \iota(X_\omega)\omega $, and 
		$ \widetilde I_{\eta, \omega}(X_\eta) = \eta(X)\eta$, then we compose in both sides with respect to the Reeb vector $\xi$ gives 
		$\eta(X_\omega) = 0 $, and $ \eta(X_\eta) = \eta(X)$. Thus, it follows that $\mathcal{L}_{X_\omega}\omega = 0$, $\mathcal{L}_{X_\omega}\eta = 0 $, 
		$\mathcal{L}_{X_\eta}\omega  = 0$, and $\mathcal{L}_{X_\eta}\eta = 0$.  For $(2)$, we derive from the first item that  when it exists, 
	 the $1-$parameter family of diffeomorphisms $\Phi_\omega = \{\phi_\omega^t\}$, preserves the cosymplectic structure. From the formula\\  
		$\mathcal{L}_{(\psi_\omega^t)^{-1}_\ast( X_\eta)}\eta =  (\psi_\omega^t)^\ast\left(\mathcal{L}_{X_\eta}\left( (\psi_\omega^t)^{-1}\right)^\ast(\eta )\right)  =  (\psi_\omega^t)^\ast\left(\mathcal{L}_{X_\eta}\eta \right),$
		we derive that  when it exists,  $1-$parameter family of diffeomorphisms $\Psi_\eta $, also preserves that cosymplectic structure. For $(3)$, 
		by the mean of the formula 
		$\imath _{[X_\omega,X_\eta]} = \mathcal{L}_{X_\omega}\circ\imath_{X_\eta}  - \imath_{X_\eta} \circ\mathcal{L}_{X_\omega}$, we compute: 
		\begin{eqnarray}
		\begin{array}{ccccccccc}
		\imath _{[X_\omega,X_\eta]}\omega &=&  \mathcal{L}_{X_\omega}\left( \imath_{X_\eta}\omega\right)   - \imath_{X_\eta} \left( \mathcal{L}_{X_\omega}\omega \right)  \\
		& = &  0 - 0,\\ \nonumber
		\end{array}\nonumber
		\end{eqnarray}
		and  
		\begin{eqnarray}
		\begin{array}{ccccccccc}
		\imath _{[X_\eta,X_\omega]}\eta &=&  \mathcal{L}_{X_\eta}\left( \imath_{X_\omega}\eta\right)   - \imath_{X_\omega} \left( \mathcal{L}_{X_\eta}\eta\right)  \\
		& = & \mathcal{L}_{X_\omega}\left( 0\right) - \imath_{X_\omega} \left( \mu\eta\right),\\ \nonumber
		& = & 0 - 0,\\ \nonumber
		\end{array}\nonumber
		\end{eqnarray}
		an then, we derive that $\widetilde I_{\eta, \omega}([X_\eta,X_\omega]) =  	\imath _{[X_\omega,X_\eta]}\omega -  \left( \imath _{[X_\eta,X_\omega]}\eta\right)\eta  
		= 0 - 0$. 
		The non-degeneracy of $ \widetilde I_{\eta, \omega}$, implies that $[X_\eta,X_\omega] = 0 $. 
		 For $(6)$, let $X, Y\in \chi_{\eta, \omega}(M)$. From the formulas\\ $\imath_{[X,Y]} = \mathcal{L}_X\circ\imath_Y  - \imath_Y \circ\mathcal{L}_X$, and 
		$ d\circ \mathcal{L}_X =  \mathcal{L}_X \circ d$,  we derive that  $ \mathcal{L}_{[X, Y]} \omega = d(\imath_{[X,Y]}\omega) = 0$\\ and 
		$ \mathcal{L}_{[X, Y]} \eta = d(\imath_{[X,Y]}\eta) = 0.$ Hence, 
		$[X,Y]\in \chi_{\eta, \omega}(M)$.  
	
\end{proof}

\begin{definition}
	Let $(M, \omega,\eta)$ be a cosymplectic manifold. An element $Y\in \chi_{\eta, \omega}(M)$ is called 
	\begin{itemize}
		\item  $\eta-$vector field if the $1-$form $\imath_Y\omega = 0$, 
		\item   $\omega-$vector field if the function $\eta(Y)$ is trivial. 
	\end{itemize}
\end{definition}
\begin{example}\label{Ex22}
	Let $Y$ be any cosymplectic vector field. Then in the decomposition,\\ 
	$ Y= Y_\eta + Y_\omega,$ 
	we have that $ Y_\eta$ is a $\omega-$vector field while $Y_\omega$ is an $\eta-$vector field. 
\end{example}
\begin{definition}
	Let $(M, \omega,\eta)$ be a cosymplectic manifold. An element $Y\in \chi_{\eta, \omega}(M)$ is called a co-Hamiltonian vector field if the $1-$form $\widetilde I_{\eta, \omega}(Y)$ is exact. 
\end{definition}
We shall denote by $ham_{\eta, \omega}(M)$ the space of all co-Hamiltonian vector fields of $(M, \eta, \omega)$. 
The following proposition was found in \cite{B-O}.
\begin{proposition}(\cite{B-O})\label{pro2}
	Let $(M, \omega,\eta)$ be a cosymplectic manifold. For any $X, Y\in \chi_{\eta, \omega}(M)$, we have $[X,Y]\in ham_{\eta, \omega}(M)$.
\end{proposition}
\begin{example}\label{Ex3} $ $ 	Let $(M, \omega,\eta)$ be a cosymplectic manifold. The Reeb vector field $\xi$ of $(M, \omega,\eta)$ is a cosymplectic 
		vector field since, $\imath_\xi\omega = 0$, implies $\mathcal{L}_\xi(\omega) = 0$, and also $\eta(\xi) = 1$, implies 
		$ \mathcal{L}_\xi(\eta) = d\left( \eta(\xi)\right) = d (1) = 0$. 
	
\end{example}


\subsection{Almost cosymplectic vector fields}\label{SC011-A} 
In this subsection, we define and study those vector fields $X$ of a cosymplectic manifold $(M, \eta, \omega)$ whose generating flow $ \Phi_X$ preserve the forms 
$\eta$, and $\omega$ up to a multiplicative smooth function. 
\begin{definition}
	Let $(M, \omega,\eta)$ be a  cosymplectic manifold. A vector field $X$ is said to be an almost cosymplectic if   $\mathcal{L}_{X}\omega = 0$, and there is a smooth function $\mu_X$ on $M$,  non-identically trivial such that  $\mathcal{L}_{X}\eta = \mu_X\eta$. 
\end{definition}
We shall denote by $\mathcal{A}\chi_{\eta, \omega}(M)$ the space of all  almost cosymplectic vector fields of $(M, \omega,\eta)$. By convention 
we assume that the only almost cosymplectic vector field $X$ with $\mu_X = 0$ is the trivial vector field.\\ 
In the above definition, for any almost cosymplectic vector field $X$ such that $\mathcal{L}_{X}\eta = \mu\eta$,
 the smooth  function $ \mu$, is uniquely determined as: $ \mu = \xi(\eta(X)).$ 
\begin{proposition}\label{pro7-0}
	Let $(M, \omega,\eta)$ be a cosymplectic manifold. For any $X, Y\in \mathcal{A}\chi_{\eta, \omega}(M)$, we have $[X,Y]\in \mathcal{A}\chi_{\eta, \omega}(M)$.
\end{proposition}
\begin{proof} Let $X, Y\in \mathcal{A}\chi_{\eta, \omega}(M)$. From the formulas $\imath _{[X,Y]} = \mathcal{L}_X\circ\imath_Y  - \imath_Y \circ\mathcal{L}_X$, and\\ 
	$ d\circ \mathcal{L}_X =  \mathcal{L}_X \circ d$,  we derive that  
	$ \mathcal{L}_{[X, Y]} \omega = d\imath_{[X,Y]}\omega= 0$ and 
	\begin{eqnarray}
	\mathcal{L}_{[X, Y]} \eta &=& d\imath _{[X,Y]}\eta = d\big( \mu_Y\eta(X) -\mu_X\eta(Y) \big)\nonumber\\
	&=& \eta(X) d\mu_Y + \mu_Yd\eta(X) - \eta(Y) d\mu_X - \mu_Xd\eta(Y)\nonumber\\
	&=& (d\mu_Y)(X)\eta + \mu_Y\mu_X\eta - (d\mu_X)(Y)\eta - \mu_X\mu_Y\eta,
	\end{eqnarray}
	since $ d (\eta(X)) = \mathcal{L}_{X}\eta = \mu_X\eta$ and $ d (\eta(Y)) = \mathcal{L}_{Y}\eta = \mu_Y\eta$.  Hence, 
	$$\mathcal{L}_{[X, Y]} \omega = (d\mu_Y)(X)\eta - (d\mu_X)(Y)\eta = f_{X, Y}\eta,$$
	with $$f_{X, Y} : = \big((d\mu_Y)(X)  - (d\mu_X)(Y)\big) \in C^\infty(M),$$ i.e., 
	$[X,Y]\in \mathcal{A}\chi_{\eta, \omega}(M)$.\end{proof}  

\begin{definition}
	Let $(M, \omega,\eta)$ be a cosymplectic manifold. An element $Y\in \mathcal{A}\chi_{\eta, \omega}(M)$ is called an almost  co-Hamiltonian vector field if 
	the $1-$form $\imath_Y\omega$ is exact. 
\end{definition}
We shall denote by $\mathcal{A}ham_{\eta, \omega}(M)$ the space of all almost co-Hamiltonian vector fields of $(M, \omega,\eta)$. 

\begin{proposition}\label{pro2-0}
	Let $(M, \omega,\eta)$ be a cosymplectic manifold. For any $X, Y\in \mathcal{A}\chi_{\eta, \omega}(M)$, we have $[X,Y]\in \mathcal{A}ham_{\eta, \omega}(M)$.
\end{proposition}
\begin{proof}Since $[X,Y]\in \mathcal{A}\chi_{\eta, \omega}(M)$, we derive that 
	$\imath_{[X,Y]}\omega  = \mathcal{L}_X\circ\imath_Y \omega  = d ( \pm\omega(X, Y)).$
\end{proof}
\begin{corollary}\label{cor1A}
	Let $(M,\omega,\eta)$ be a  cosymplectic manifold. For any $X\in \mathcal{A}\chi_{\eta, \omega}(M)$, the  $1-$form 
	$ \widetilde I_{\eta, \omega}(X)$ is closed. 
\end{corollary}
\begin{proof} For each $X\in \mathcal{A}\chi_{\eta, \omega}(M)$, since $\mathcal{L}_X = \imath_X \circ d + d\circ \imath_X$, we derive from the equalities
	$\mathcal{L}_X(\omega) = 0 $ and $\mathcal{L}_X(\eta) = \mu_X\eta $ that $d(\imath_X\omega) = 0$ and $ d(\eta(X)) =\mu_X\eta  $. Thus, 
	$$ d(\widetilde I_{\eta, \omega}(X)) = d(\imath_X\omega) + d(\eta(X)\eta) = d(\imath_X\omega) + d(\eta(X))\wedge \eta = 0 + 0.$$
\end{proof}
Here is a consequence of the $(\omega, \eta)-$decomposition of vector fields on a cosymplectic manifold.
\begin{proposition}\label{ADecom-0}
	Let $Y $ be an almost cosymplectic vector field such that $\mathcal{L}_Y(\eta) = \mu\eta $ with $(\omega, \eta)-$decomposition  $ Y = Y_\omega + Y_\eta$. Then, 
	\begin{enumerate}
		\item the vector field $ Y_\omega$ 	(resp. $ Y_\eta $) is cosymplectic (resp. almost cosymplectic),
		\item the flow $\Psi_\omega = \{\psi_\omega^t\}$ (resp. $\Psi_\eta =  \{\psi_\eta^t\}$) generated by the vector field 
		$ Y_\omega $ 
		(resp. $ Y_\eta $) is cosymplectic (resp. almost cosymplectic 
		with $\mathcal{L}_{\dot\psi_\eta^t}\eta = \mu\eta$) when it exists, 
		\item $[ Y_\omega , Y_\eta ] = 0 $, 
		\item $ \psi_\eta^s \circ \psi_\omega^t = \psi_\omega^s \circ \psi_\omega^t$, for each $s, t$, 
		and 
		\item the flow $\Psi$ generated by $Y$ decomposes as $\Psi= \Psi_\omega\circ\Psi_\eta$, and $\Psi= \Psi_\eta\circ\Psi_\omega$.
	\end{enumerate}
\end{proposition}
\begin{proof} Let $Y $ be an almost cosymplectic vector field such that $\mathcal{L}_Y(\eta) = \mu\eta $ with\\ $(\omega, \eta)-$decomposition  $ Y = Y_\omega + Y_\eta$. For $(1)$, we derive from $\widetilde I_{\eta, \omega}(Y_\omega) = \iota(Y)\omega $, by composing with respect to the Reeb vector field that 
	$ \eta(Y_\omega) = 0$, and so we also have $ \iota(Y_\omega)\omega = \iota(Y)\omega$, which show that $Y_\omega$ is a cosymplectic vector field. Similarly, 
	from  
	$\widetilde I_{\eta, \omega}(Y_\eta) = \eta(Y)\eta $, we derive that 
	$ \eta(Y_\eta) =  \eta(Y)$, and so we also have $ \iota(Y_\eta)\omega = 0$. This implies that 
	$\mathcal{L}_{Y_\eta}\omega = 0 $, and  $\mathcal{L}_{Y_\eta}\eta = \mu\eta$. On the other hand, to prove $(2)$, we derive from the first item 
	that when the isotopy generated by $Y_\omega$ exists, then the latter is  a cosymplectic isotopy. From the formula 
	$\mathcal{L}_{(\psi_\omega^t)^{-1}_\ast( Y_\eta)}\eta =  (\psi_\omega^t)^\ast\left(\mathcal{L}_{Y_\eta}\left( (\psi_\omega^t)^{-1}\right)^\ast(\eta )\right)  =  (\psi_\omega^t)^\ast\left(\mathcal{L}_{Y_\eta}\eta \right),$
	we derive that $\Psi_\eta $, is an almost cosymplectic isotopy when it exists. By the mean of the formula 
	$\imath _{[Y_\omega,Y_\eta]} = \mathcal{L}_{Y_\omega}\circ\imath_{Y_\eta}  - \imath_{Y_\eta} \circ\mathcal{L}_{Y_\omega}$, we compute: 
	\begin{eqnarray}
	\begin{array}{ccccccccc}
	\imath _{[Y_\omega,Y_\eta]}\omega &=&  \mathcal{L}_{Y_\omega}\left( \imath_{Y_\eta}\omega\right)   - \imath_{Y_\eta} \left( \mathcal{L}_{Y_\omega}\omega \right)  \\
	& = &  0 - 0,\\ \nonumber
	\end{array}\nonumber
	\end{eqnarray}
	and  
	\begin{eqnarray}
	\begin{array}{ccccccccc}
	\imath _{[Y_\eta,Y_\omega]}\eta &=&  \mathcal{L}_{Y_\eta}\left( \imath_{Y_\omega}\eta\right)   - \imath_{Y_\omega} \left( \mathcal{L}_{Y_\eta}\eta\right)  \\
	& = & \mathcal{L}_{Y_\omega}\left( 0\right) - \imath_{Y_\omega} \left( \mu\eta\right),\\ \nonumber
	& = & 0 - 0,\\ \nonumber
	\end{array}\nonumber
	\end{eqnarray}
	an then, we derive that $	\widetilde I_{\eta, \omega}([Y_\eta,Y_\omega]) = 	\imath _{[Y_\omega,Y_\eta]}\omega -  \left( \imath _{[Y_\eta,Y_\omega]}\eta\right)\eta  = 0 - 0 $. 
	The non-degeneracy of $ \widetilde I_{\eta, \omega}$, implies that $[Y_\eta,Y_\omega] = 0 $. 
\end{proof}

\section{Diffeomorphisms of a cosymplectic manifold}\label{SC00-2}

\begin{definition} Let $(M, \omega,\eta)$ be a cosymplectic manifold. 
	\begin{enumerate}
		\item A diffeomorphism $\phi : M\longrightarrow M$ is said to be an almost cosymplectic diffeomorphism (or almost cosymplectomorphism) if : $\phi^\ast(\omega)= \omega$ and there exists a smooth function $f\in C^\infty(M)$ such that $\phi^\ast(\eta)= e^f\eta$. 
		\item  A diffeomorphism $\phi : M\longrightarrow M$ is said to be a  cosymplectic diffeomorphism (or cosymplectomorphism) if : $\phi^\ast(\eta)= \eta$ and $\phi^\ast(\omega)= \omega$.  
	\end{enumerate}
\end{definition}

We shall denote by $\rm{\mathcal{A}Cosymp}_{\eta, \omega}(M)$ the space of all almost cosymplectomorphisms of $(M,\omega,\eta, )$ and by $\text{Cosymp}_{\omega,\eta}(M)$ the space of all cosymplectomorphisms of $(M, \omega, \eta)$.

\begin{definition} Let $(M, \omega,\eta)$ be a cosymplectic manifold. 
	An isotopy $\Phi =\{\phi_t\}$ is said to be an almost cosymplectic (resp.  cosymplectic) if for each time $t$, we have 
	$\phi_t\in \rm{\mathcal{A}Cosymp}_{\eta, \omega}(M)$ (resp. $\phi_t\in  \text{Cosymp}_{\eta, \omega}(M)$). 
\end{definition}
We shall denote by  $\mathcal{A}Iso_{\eta, \omega}(M)$ (resp. $Iso_{\eta, \omega}(M)$) the space of all almost cosymplectic (resp. 
the space of all cosymplectic) isotopies of $(M, \eta, \omega)$.\\ 


We then define the following important subgroups: 
\begin{equation*}
\mathcal{A}G_{\eta, \omega}(M) := ev_1\left(\mathcal{A}Iso_{\eta, \omega}(M)\right)  \quad \text{and}\quad G_{\eta, \omega}(M) := ev_1\left(Iso_{\eta, \omega}(M)\right),
\end{equation*}
and  equip both groups $ \mathcal{A}G_{\eta, \omega}(M)$ and $G_{\eta, \omega}(M) $ with the $C^\infty-$compact-open topology \cite{Hirs76}.

\begin{lemma}\label{pushfoward}
	Let $(M,\omega, \eta)$ be a cosymplectic manifold and let $\xi$ be its Reeb vector field. We have the following properties.
	\begin{enumerate}
		\item If $\phi\in G_{\eta, \omega}(M)$, then \ $\phi_\ast(\xi) = \xi.$
		\item If $\psi\in  \mathcal{A}G_{\eta, \omega}(M)$ with $\psi^\ast(\eta) = e^f\eta$, then \ $\psi_\ast(\xi) = e^{f\circ\psi^{-1}}\xi.$
	\end{enumerate}
\end{lemma}
\begin{proof} 
	Since for all  diffeomorphism $\varphi\in \text{Diff}(M)$, we have\\ $\widetilde I_{\eta, \omega} (\varphi_\ast(\xi)) = (\varphi^{-1})^\ast(\imath_\xi \varphi^\ast(\omega)) + (\varphi^{-1})^\ast(\imath_\xi  \varphi^\ast(\eta))\eta,$ then  we derive that: 
	\begin{enumerate}
		\item If $\phi\in G_{\eta, \omega}(M)$, then for $\varphi =\phi$, we have 
		\begin{eqnarray}
		\widetilde I_{\eta, \omega} (\phi_\ast(\xi)) &=& (\phi^{-1})^\ast(\imath_\xi \phi^\ast(\omega)) + (\phi^{-1})^\ast(\imath_\xi \phi^\ast(\eta))\eta,\nonumber\\
		&  = & (\phi^{-1})^\ast(\imath_\xi\omega) + (\phi^{-1})^\ast(\imath_\xi\eta)\eta\nonumber\\
		& =&  \eta\nonumber\\ 
		& =&\widetilde  I_{\eta, \omega} (\xi).
		\end{eqnarray}
		
		Thus, $\widetilde  I_{\eta, \omega} (\phi_\ast(\xi)) = \widetilde I_{\eta, \omega} (\xi)$ which  implies that $\phi_\ast(\xi) = \xi$, since $ \widetilde I_{\eta, \omega}$ is non-degenerate.  
		\item If $\psi\in  \mathcal{A}G_{\eta, \omega}(M)$ with $\psi^\ast(\eta) = e^f\eta$, then 
		\begin{eqnarray}
		\widetilde I_{\eta, \omega} (\psi_\ast(\xi)) &=& (\phi^{-1})^\ast(\imath_\xi\psi^\ast(\omega)) + (\psi^{-1})^\ast(\imath_\xi \psi^\ast(\eta))\eta,\nonumber\\
		& =& (\psi^{-1})^\ast(\imath_\xi\omega) + (\psi^{-1})^\ast(e^f\imath_\xi\eta)\eta\nonumber\\
		& = & 0 +  e^{f\circ\psi^{-1}}\eta \nonumber\\
		& = & \widetilde I_{\eta, \omega} (e^{f\circ\psi^{-1}}\xi).
		\end{eqnarray}
		Thus, $\widetilde  I_{\eta, \omega} (\psi_\ast(\xi)) =\widetilde  I_{\eta, \omega} (e^{f\circ\psi^{-1}}\xi)$ from which we derive   
		$\psi_\ast(\xi) = e^{f\circ\psi^{-1}}\xi$, since $\widetilde I_{\eta, \omega}$ is non-degenerate.   
	\end{enumerate}
\end{proof}

The following result generalizes Lemma \ref{pushfoward} and its proof comes immediately from the proof of Lemma \ref{pushfoward}. 

\begin{lemma}\label{Gpushfoward}
	Let $(M_i, \omega_i, \eta_i)$, $i = 1, 2$ be two cosymplectic manifolds and let $\xi_i$, $i = 1, 2$ be their  Reeb vector fields respectively. 
	We have the following properties.
	\begin{enumerate}
		\item If $\phi\in \text{Diff}(M_1, M_2)$ such that $ \phi^\ast(\omega_2) = \omega_1$, and $ \phi^\ast(\eta_2) = \eta_1$, then $\phi_\ast(\xi_1) = \xi_2.$
		\item If $\psi\in \text{Diff}(M_1, M_2)$ such that $ \psi^\ast(\omega_2) = \omega_1$, and $ \psi^\ast(\eta_2) = e^h\eta_1$, then 
		$\psi_\ast(\xi_1) = e^{h\circ\psi^{-1}}\xi_2.$\\
	\end{enumerate}
\end{lemma}

\begin{definition}
	Let $(M, \omega,\eta)$ be a cosymplectic manifold. An isotopy $\Psi:=\{\psi_t\}$ is called an almost co-Hamiltonian isotopy, if
	for each $t$, the vector field $\dot\psi_t$ is an almost co-Hamiltonian vector field, i.e., $\dot\psi_t\in \mathcal{A}ham_{\eta, \omega}(M)$, for each $t$. 
\end{definition}
We shall denote by $\mathcal{A}H_{\eta, \omega}(M)$ the space of all almost co-Hamiltonian isotopies of $(M, \omega,\eta)$, and put
\begin{equation}
\mathcal{A}Ham_{\eta, \omega}(M) := ev_1\left(\mathcal{A}H_{\eta, \omega}(M)\right). 
\end{equation}
The elements of the set $\mathcal{A}H_{\eta, \omega}(M)$ are called almost co-Hamiltonian diffeomorphisms of $(M, \omega,\eta)$. 

\begin{definition}
	Let $(M, \omega,\eta)$ be a cosymplectic manifold. A isotopy $\Psi:=\{\psi_t\}$ is called a co-Hamiltonian isotopy, if
	for each $t$, the vector field $\dot\psi_t$ is a co-Hamiltonian vector field, i.e., $\dot\psi_t\in ham_{\eta, \omega}(M)$, for each $t$. 
\end{definition}
We shall denote by $H_{\eta, \omega}(M)$ the space of all  co-Hamiltonian isotopies of $(M, \omega,\eta)$, and put
\begin{equation}
Ham_{\eta, \omega}(M) := ev_1\left(H_{\eta, \omega}(M)\right). 
\end{equation}
The elements of the set $Ham_{\eta, \omega}(M)$ are called co-Hamiltonian diffeomorphisms of $(M, \omega,\eta)$. 

\begin{proposition}\label{pro4-A}
	Let $(M, \omega, \eta)$ be a  cosymplectic manifold. The following properties hold.
	\begin{enumerate}
		\item The set $\mathcal{A}Ham_{\eta, \omega}(M)$ is a Lie group whose Lie algebra is the space $ \mathcal{A}ham_{\eta, \omega}(M)$.
		\item The set $\mathcal{A}Ham_{\eta, \omega}(M)$ is a normal subgroup in the group $ \mathcal{A}G_{\eta, \omega}(M)$. 
		\item The set $Ham_{\eta, \omega}(M)$ is a Lie group whose Lie algebra is the space $ham_{\eta, \omega}(M)$.
		\item  The set $Ham_{\eta, \omega}(M)$ is a normal subgroup in the group $G_{\eta, \omega}(M)$. 
	\end{enumerate}
\end{proposition}
\subsection{Cosymplectic and almost cosymplectic flows}\label{rem1}
 To that end, we shall need the following fact : to any smooth isotopy  
 $\Phi =\{\phi_t\}$  with  $\phi_0 = id_M$, it is attached a smooth family of smooth vector fields $\{\dot{\phi}_t\}$, defined by
 \begin{equation}
 \boxed{\dot\phi_t := X_t\circ\phi_t^{-1}}
 \end{equation}
  where 
 \begin{equation}
 X_t (x):= \frac{d\phi_t(x)}{d t},
 \end{equation}
for each $t$ and for all $x\in M$. Furthermore, we have 
\begin{equation}
	 \boxed{\phi_t^\ast\left(\widetilde I_{\eta, \omega}(\dot\phi_t)\right)  = 	\phi_t^\ast\left( \widetilde I_{\eta, \omega}(X_t)\right) }
\end{equation}
for each $t$, provided  $\Phi =\{\phi_t\}$ is cosymplectic. 
\begin{enumerate}
	\item Let $\{\phi_t\}\in H_{\eta, \omega}(M)$ such that $\widetilde I_{\eta, \omega}(\dot\phi_t) = dF_t$, for each $t$ and smooth function $F_t$. Then, from the relation 
	$\dot{\phi}_{-t} = - (\phi_t^{-1})_\ast(\dot{\phi}_t),$  
	for each $t$ and 
	$ \imath_{(\phi_t^{-1})_\ast\dot{\phi}_t}\alpha = -\phi_t^\ast\big(\imath_{(\dot{\phi}_t\circ\phi_t^{-1})}\alpha\big),$ 
	for all $p-$form $\alpha$, we derive that 
	\begin{equation}\label{equ1}
	\widetilde I_{\eta, \omega}(\dot{\phi}_{-t}) = -\phi_t^\ast(\imath_{\dot{\phi}_t}\omega) -\phi_t^\ast(\eta(\dot{\phi}_t))\eta = -\phi_t^\ast(\widetilde I_{\eta, \omega}(\dot\phi_t))= d(-F_t\circ\phi_t), \quad \text{for all } \quad t.
	\end{equation}
	Hence, $\{\phi_t^{-1}\}\in H_{\eta, \omega}(M)$ and $\widetilde I_{\eta, \omega}(\dot{\phi}_{-t}) =  d(-F_t\circ\phi_t), $ where $\phi_t^{-1} =:\phi_{-t}$, for each $t$. 
	\item If  $\Phi_F = \{\phi_t\}$ is a co-Hamiltonian isotopy such that $\widetilde I_{\eta, \omega} (\dot\phi_t) = dF_t$, for all $t$, then for all 
	$\rho \in G_{\eta,  \omega}(M)$, the isotopy $ \Psi =\{\psi_t\}$ with  $\psi_t :=\rho^{-1} \circ \phi_t\circ \rho$ is also  co-Hamiltonian : in fact, 
	from $\dot { \psi}_t = \rho_\ast^{-1}(\dot\phi_t)$, we derive that 
	$$ \widetilde I_{\eta, \omega} (\dot { \psi}_t) = \rho^\ast (\widetilde I_{\eta, \omega} (\dot\phi_t)) = d(F_t\circ\rho), \quad \text{for each}\  t. $$ 
	\item Similarly, if  $\{\psi_t\}$ and  $\{\phi_t\}$ are two elements of $H_{\eta, \omega}(M)$ such that $\widetilde I_{\eta, \omega}(\dot\phi_t) = dF_t$ and 
	$\widetilde I_{\eta, \omega}(\dot\psi_t) = dK_t$, for each $t$, then we have 
	\begin{equation}\label{equ2}
	\widetilde I_{\eta, \omega}(\dot{\overbrace{\phi_t\circ\psi_t}}) = d(F_t + K_t\circ\phi_t^{-1}),
	\quad \quad \text{for each}\  t. 
	\end{equation}
	\item Let $\{\phi_t\}\in Iso_{\eta, \omega}(M)$. Then, for each $t$, we have  
	$$\widetilde I_{\eta, \omega}(\dot{\phi}_{-t}) = -\phi_t^\ast(\widetilde I_{\eta, \omega}(\dot{\phi}_t)) .$$
	\item If  $\{\psi_t\}$ and  $\{\phi_t\}$ are two elements of $ Iso_{\eta, \omega}(M)$, then 
	for each $t$, we have 
	$$\widetilde I_{\eta, \omega}(\dot{\overbrace{\phi_t\circ\psi_t}}) = \widetilde I_{\eta, \omega}(\dot{\phi}_t) + (\phi_t^{-1})^\ast(\widetilde I_{\eta, \omega}(\dot{\psi}_t)).$$
	\item Let $\{\phi_t\}\in \mathcal{A}Iso_{\eta, \omega}(M) $ such that $\mathcal{L}_{\dot\phi_t}\eta = \mu_t\eta$, (or $\phi_t^\ast(\eta) = e^{f_t}\eta$), 
	for each $t$. Compute,  
	\begin{eqnarray}\label{equ1-I}
	\mathcal{L}_{\dot{\phi}_{-t}}\eta &=& \phi_{t}^\ast \left( \frac{d}{dt}\left( (\phi_t^{-1})^\ast(\eta)\right) \right)  
	=  \phi_{t}^\ast \left( \frac{d}{dt}\left(e^{-f_t\circ\phi_t^{-1}}\eta\right) \right)  \\
	&=&  \phi_{t}^\ast \left( \eta \frac{d}{dt}\left(e^{- f_t\circ\phi_t^{-1}}\right) \right)=  (-(\dot{\overbrace{f_t\circ\phi_t^{-1}}})\circ\phi_t)\eta,
	\end{eqnarray} 
	i.e., $ \mathcal{L}_{\dot{\phi}_{-t}}\eta = \vartheta_t\eta,$ for all $t$, with  $ \vartheta_t := -(\dot{\overbrace{f_t\circ\phi_t^{-1}}})\circ\phi_t.$ 
	\item If  $\{\psi_t\}$ and $\{\phi_t\}$ are two elements of $  \mathcal{A}Iso_{\eta, \omega}(M)$ such that $\mathcal{L}_{\dot\phi_t}\eta = \mu_t\eta$,\\ (or $\phi_t^\ast(\eta) = e^{f_t}\eta$), and 
	$\mathcal{L}_{\dot\psi_t}\eta = \mu'_t\eta$, (or $\psi_t^\ast(\eta) = e^{q_t}\eta$),
	for each $t$, then we have 
	\begin{eqnarray}\label{equ2-I}
	\mathcal{L}_{\dot{\overbrace{\phi_t\circ\psi_t}}}\eta &=&  ( (\phi_t\circ\psi_t)^{-1})^\ast \left( \frac{d}{dt}\left( (\phi_t\circ\psi_t)^\ast\eta\right)   \right),\nonumber\\
	&=&  ( (\phi_t\circ\psi_t)^{-1})^\ast 	\left( \frac{d}{dt}\left( e^{f_t\circ\psi_t}e^{q_t}\eta \right) \right),\nonumber\\
	&=&  \left( \dot{\overbrace{f_t\circ\psi_t}} + \dot{q}_t \right)\circ(\phi_t\circ\psi_t)^{-1} \eta, \quad \text{ for each} \ t. 
	\end{eqnarray}
	Thus, 
	\begin{equation}\label{equ2-Iv}
	\mathcal{L}_{\dot{\overbrace{\phi_t\circ\psi_t}}}\eta = \varrho_t\eta,
	\end{equation}
	with  $ \varrho_t : = \left( \dot{\overbrace{f_t\circ\psi_t}} + \dot{q}_t \right)\circ(\phi_t\circ\psi_t)^{-1},$ 
	for all $t$.
	\item 	If  $\Phi = \{\phi_t\}$ is an almost co-Hamiltonian isotopy such that $ \imath_{\dot\phi_t}\omega = dF_t$, for all $t$, and \\
	$\mathcal{L}_{\dot\phi_t}\eta = \mu_t\eta$, (or $\phi_t^\ast(\eta) = e^{f_t}\eta$), 
	for each $t$, then  for all $\rho \in \mathcal{A}G_{\eta,  \omega}(M)$ such that\\ $\rho^\ast(\eta) = e^{f^\rho}\eta$, 
	the isotopy $ \Psi =\{\psi_t\}$ with $ \psi_t := \rho^{-1} \circ \phi_t\circ \rho$ is an almost co-Hamiltonian: in fact, 
	from $\dot { \psi}_t = \rho_\ast^{-1}(\dot\phi_t)$, we derive that 
	$ \imath_{\dot  \psi_t}\omega = \rho^\ast ( \imath_{\dot\phi_t}\omega) = d(F_t\circ\rho),$ 
	for each $t$. On the other hand, we have 
	\begin{eqnarray}
	\mathcal{L}_{\rho_\ast^{-1}(\dot\phi_t)}\eta &=& \mathcal{L}_{\dot  \psi_t}\eta   =  (\psi_t^{-1})^\ast 
	\left( \frac{d}{dt}\left(( \rho^{-1} \circ \phi_t\circ \rho)^\ast(\eta) \right) \right)\nonumber\\
	&=&  (\psi_t^{-1})^\ast  \frac{d}{dt}\left( e^{ -f^\rho\circ\rho^{-1} \circ \phi_t\circ \rho}e^{f_t\circ\rho}e^{f^\rho}\eta \right) , \qquad
	\text{for all} \ t.
	\end{eqnarray}
	Hence, 
	$ \mathcal{L}_{\rho_\ast^{-1}(\dot\phi_t)}\eta =  H_t(\rho)\eta,$
	with 
	\begin{eqnarray} 
	H_t(\rho) &:=& e^{ -f^\rho\circ\rho^{-1} \circ \phi_t\circ \rho}e^{f_t\circ\rho}e^{f^\rho}\frac{d}{dt}\left( e^{ -f^\rho\circ\rho^{-1} \circ \phi_t\circ \rho}e^{f_t\circ\rho}e^{f^\rho}\eta \right)\circ ( \rho^{-1} \circ \phi_t\circ \rho)^{-1},\nonumber\\
	&=&  \left(-df^\rho((\rho^{-1})_\ast(\dot{\phi}_t )) + \dot f_t\right) \circ (\rho^{-1}\circ\phi_t\circ \rho)^{-1}\quad \text{for each} \ t.
	\end{eqnarray} 
	\item For any isotopy $\Phi = \{\phi_t\},$ we shall denote by $C_{\Phi,\eta}^t$ the smooth function  $x\longmapsto\eta(\dot \phi_t)(\phi_t(x))$. 
	If  $\Phi = \{\phi_t\}$ and $\Psi = \{\psi_t\}$ are two elements of $Iso_{\eta, \omega}(M)$ such that 
	$\mathcal{L}_{\dot\phi_t}\eta = \mu_t\eta$, (or $\phi_t^\ast(\eta) = e^{f_t}\eta$),	for each $t$, then we have 
	\begin{eqnarray}C_{\Phi\circ\Psi, \eta}^t &=& \eta(\dot \phi_t + (\phi_t)_\ast(\dot \psi_t))\circ(\phi_t\circ\psi_t)\nonumber\\
	&=&  C_{\Phi, \eta}^t\circ\psi_t + \eta((\phi_t)_\ast(\dot \psi_t))\circ(\phi_t\circ\psi_t)\nonumber\\
	&=&  C_{\Phi, \eta}^t\circ\psi_t + \big((\phi_t^{-1})^\ast( e^{f_t}\eta(\dot \psi_t))\big)\circ(\phi_t\circ\psi_t)\nonumber\\
	&=&  C_{\Phi, \eta}^t\circ\psi_t + e^{f_t\circ\psi_t}  C_{\Psi, \eta}^t, \qquad \text{for each} \  t \nonumber
	\end{eqnarray} namely 
	\begin{equation}\label{rfunc}  
	C_{\Phi\circ\Psi, \eta}^t =  C_{\Phi, \eta}^t\circ\psi_t + e^{f_t\circ\psi_t}  C_{\Psi, \eta}^t,\quad \text{for all} \  t.
	\end{equation}
	\item So, from \eqref{rfunc}, we derive that if $\Phi = \{\phi_t\}$ is an almost cosymplectic isotopy such that $\phi_t^\ast(\eta) = e^{f_t}\eta$, 
	for each $t$, then, we have  $$ C_{\Phi^{-1},\eta}^t =  - e^{f_t}  C_{\Phi, \eta}^t\circ \phi_t^{-1},\quad \text{for all} \ t.$$ 	
\end{enumerate}
\subsection{Cosymplectic geometry and Symplectic geometry}\label{Trans}
\subsubsection{Co-Hamiltonian dynamical systems}\label{CHDS}
\begin{itemize}
	\item  Let  $\mathfrak{N}([0,1]\times M\ ,\mathbb{R})$ denotes the vector space of all smooth functions $F$ defined on $ [0,1]\times M$ such that $\displaystyle \int_MF_t\eta\wedge\omega^n = 0,$ for all $t\in [0,1]$. We shall call such functions, time-dependent co-normalized function. It is clear that any $\{\phi_t\}\in H_{\eta, \omega}(M)$ determines a unique co-normalized  function $ F\in \mathfrak{N}([0,1]\times M\ ,\mathbb{R})$ such that $
	\widetilde I_{\eta, \omega}(\dot{\phi}_t) = dF_t,$ 
	for all $t$.   
	\item Let $X$ be a co-Hamiltonian vector field of $(M, \omega,\eta)$, and let $\Phi_X$ be its flow. Since 
	$\widetilde I_{\eta, \omega}(X) = d G$, for some $G\in C^\infty(M,\mathbb{R})$, we derive that for each $p\in M$, we have
	$$\dfrac{d}{ dt}G(\phi_X^t(p)) = dG(X)(\phi_X^t(p)) = \eta(X)^2(\phi_X^t(p)) \geqslant 0, \ \text{  for all}\  t \ \text {and for all} \  p\in M$$
	According to the above equality,  along the orbit $t\mapsto\phi_X^t(p)$,  the energy function $G$, increases with time, and we have $$\displaystyle G(\phi_X^t(p)) = G(p)  + \int_0^t \eta(X)^2(\phi_X^s(p))ds,$$
	for all $t$, and for all $p\in M$. It seems that, in general, not as in symplectic geometry (where the Hamiltonian is constant 
	along the orbits of its generating flow), in the cosymplectic case, we do not yet have such a conservation property.\\
	Furthermore, it seems that in such a dynamics system, given a co-Hamiltonian vector field $X$ of $(M, \omega,\eta)$, 
	the orbit $t\mapsto\phi_X^t(p)$ is  $\sigma-$periodic if and only if,\\ $\displaystyle \int_0^\sigma \eta(X)^2(\phi_X^s(p))ds = 0$  and since the map 
	$s\mapsto \eta(X)^2(\phi_X^s(p))$, is non-negative and continuous, then we must have $\eta(X)(\phi_X^s(p)) = 0,$ for all $s\in [0,\sigma]$. \\
\end{itemize}

\subsubsection*{Co-Hamiltonian equations:}
Let $(z,q_1,\dots,q_n, p_1,\dots,p_n)$ be the cosymplectic-Darboux coordinates system such that\\ $ \omega = \Sigma_{i = 1}^n dq_i\wedge dp_i$, and 
$\eta = dz$. Let $\{H_t\}$ be a smooth family of smooth functions on $M$ such that 
$$ 
\left\{
\begin{array}{l}
\dot p_i =-\frac{\partial}{ \partial q_i}H_t, \\
\dot q_i = \frac{\partial}{ \partial p_i}H_t, \\
\dot z =\frac{\partial}{ \partial z}H_t . \\
\end{array}
\right. 
$$  
It is clear that if $X_t := \tilde I^{-1}_{\eta, \omega}(dH_t),$ then the one-parameter family of local diffeomorphisms $ \Phi : = \{\phi_t\}$ generated 
by $ X_t$, consists of co-Hamiltonian diffeomorphisms whenever\\ $\frac{\partial}{ \partial z}H_t = cte(t)$. Conversely, if $ \Phi : = \{\phi_t\}$ is a co-Hamiltonian 
isotopy such that 
$ \tilde I_{\eta, \omega}(\dot{\phi}_t) = dL_t $, then 
$$ 
(CHE)\left\{
\begin{array}{l}
\dot p_i =-\frac{\partial}{ \partial q_i}L_t, \\
\dot q_i = \frac{\partial}{ \partial p_i}L_t, \\
\dot z =\frac{\partial}{ \partial z}L_t = cte(t). \\
\end{array}
\right. 
$$ \\

Lemma \ref{lem-3} tells to us that: given a smooth manifold $M$, a $1-$form $\eta$, and  a $2-$form $\omega$ on $M$, if we denote by 
$\mathbb{S}^1$ unit circle equipped with the coordinate function $\theta$, and we consider smooth manifold $\widetilde M := M\times \mathbb{S}^1$, equipped with the $2-$form $\widetilde\omega := p^\ast\omega + p^\ast\eta\wedge \pi_2^\ast (d\theta)$ where $\pi_2: \widetilde M\rightarrow \mathbb{S}^1$, and $p: \widetilde M\rightarrow M$ are projection maps, then  $(M, \omega,\eta)$ is a  cosymplectic manifold, if and only if,  $(\widetilde M, \tilde\omega)$ is a  symplectic manifold. We have the following facts: 
\begin{enumerate}
	\item[{\sf Fact 1}:] For any cosymplectic  isotopy $\Phi = \{\phi_t\}$, one defines an isotopy $\tilde\Phi = \{\tilde\phi_t\}$ of  the symplectic manifold $(\widetilde M, \widetilde\omega)$  as follows: For each $t$, 
	\begin{eqnarray}\begin{array}{ccclc}
	\tilde\phi_t : & M\times \mathbb{S}^1 & \longrightarrow  & M\times \mathbb{S}^1\\
	&(x, \theta) & \longmapsto& \big(\phi_t(x), \mathcal{R}_{\Lambda_t(\Phi)}(x, \theta) \big),
	\end{array}\nonumber\end{eqnarray}
	where $ \mathcal{R}_{\Lambda_t(\Phi)}(x, \theta) := \theta - \int_0^tC_{\Phi, \eta}^s(x) ds \mod{2\pi}.$
	Furthermore, if we consider the canonical projection  $p: \widetilde M\rightarrow M,$ then for each $t$, we have the commutative diagram 
	\begin{eqnarray}
	\begin{array}{cccccccccc}
	\widetilde M & \stackrel{\tilde \phi_t}{\longrightarrow} & \widetilde M\\[.1cm]
	p \downarrow &   & \downarrow p\\[.1cm]
	M & \stackrel{ \phi_t}{\longrightarrow}  &  M,
	\end{array}\nonumber 
	\end{eqnarray}
	namely $ p\circ\tilde{\phi_t} = \phi_t\circ p$. The isotopy $\widetilde\Phi = \{\widetilde\phi_t\}$, is in fact symplectic: since we have  
	\begin{eqnarray} 
	\tilde{\phi_t}^\ast(\widetilde\omega) &=& (p\circ\tilde{\phi_t})^\ast\omega + (p\circ\tilde{\phi_t})^\ast\eta\wedge \tilde{\phi_t}^\ast(\pi_2^\ast (d\theta))\nonumber\\
	&=& (\phi_t\circ p)^\ast\omega + (\phi_t\circ p)^\ast\eta\wedge \pi_2^\ast (d\theta)\nonumber\\
	&=& p^\ast\omega + p^\ast\eta\wedge \pi_2^\ast (d\theta) \nonumber\\
	& =& \widetilde\omega.
	\end{eqnarray}
	\item[{\sf Fact 2}:]  For any cosymplectic  isotopy $\Phi = \{\phi_t\}$, one defines an isotopy $\tilde\Phi = \{\tilde\phi_t\}$ as in {\bf Fact 1}, and we can easily compute,  
	$\dot {\tilde{\phi_t}} = \dot{\phi}_t - p^\ast(C_{\Phi, \eta}^t ) \dfrac{\partial}{\partial \theta},$ 
	for each $t$, which implies that 
	\begin{equation}\label{Hofer-trans}
	\imath_{\dot {\tilde{\phi_t}}}\widetilde\omega = p^\ast( \imath_{\dot {\phi}_t }\omega)  +   p^\ast( C_{\Phi, \eta}^t) \pi_2^\ast (d\theta) +    p^\ast(C_{\Phi, \eta}^t\eta), \ \text{ for each} \ t.
	\end{equation}
\end{enumerate}

  \subsection{Almost cosymplectic structure and Symplectic structure}\label{Transs}

  \begin{proposition}\label{Trasit-1}
  	Let $(M, \eta, \omega)$ be a cosymplectic manifold. 
  	If $\Psi = \{\phi_t\}$ is any  almost cosymplectic isotopy such that 
  	$\psi_t^\ast(\eta) = e^{f_t}\eta$ for all $t$, then the isotopy $\tilde\Psi = \{\tilde\psi_t\}$ defined by 
  	$$ \tilde\psi_t: M\times \mathbb{S}^1 \longrightarrow M\times \mathbb{S}^1,(x, \theta) \mapsto (\psi_t(x), \theta e^{-f_t(x)} ),$$
  	is a  symplectic isotopy of  the symplectic manifold $(\tilde M, \tilde\omega)$. 
  \end{proposition}
  {\it Proof.} Assume $\Psi = \{\psi_t\}$ to be any almost cosymplectic isotopy such that 
  $\psi_t^\ast(\eta) = e^{f_t}\eta$ for all $t$, and consider $p: \tilde M\rightarrow M,$ and $\pi_2: \tilde M\rightarrow \mathbb{S}^1,$ are projection maps. For each $t$, we have $p\circ \tilde{\psi}_t = \psi_t\circ p$, and also
  \begin{eqnarray} 
  \tilde{\psi}_t^\ast(\tilde{\omega}) &=&  \tilde{\psi}_t^\ast (p^\ast(\omega) + p^\ast(\eta)\wedge \pi^\ast_2(d\theta)) \nonumber\\
  &=&(p\circ\tilde{\psi}_t)^\ast\omega +  (p\circ\tilde{\psi}_t)^\ast\eta\wedge (\pi_2\circ\tilde{\psi}_t)^\ast d\theta \nonumber\\
  &=&  (\psi_t\circ p)^\ast\omega +  (\psi_t\circ p)^\ast\eta\wedge e^{-f_t\circ p}\pi_2^\ast d\theta \nonumber\\
  &=&  p^\ast((\psi_t)^\ast\omega) +  p^\ast((\psi_t)^\ast\eta) \wedge e^{-f_t\circ p}\pi_2^\ast d\theta \nonumber\\
  &=&  p^\ast(\omega) +  e^{f_t\circ p}p^\ast(\eta) \wedge e^{-f_t\circ p}\pi_2^\ast (d\theta)
  \nonumber\\
  &=& \tilde{\omega}.
  \end{eqnarray}
  
  Thus, $\tilde\Psi :=  \{(\psi_t\circ p, \pi_2 e^{-f_t\circ p} )_t\},$ is a symplectic isotopy of  $(\tilde M, \tilde\omega)$. $ \blacksquare$
  \begin{proposition}\label{Trasit-2}
  	Let $(M, \eta, \omega)$ be a  cosymplectic manifold. 
  	If $\Psi = \{\phi_t\}$ is any  almost cosymplectic isotopy such that 
  	$\psi_t^\ast(\eta) = e^{f_t}\eta$ for all $t$, then we have 
  	$$  d\left( \pi_2\eta(\dot{\psi_t})\circ p\right)  = \left( \eta(\dot{\psi_t})\circ p\right) \pi_2^\ast(d\theta)  + \left(\dot f_t\circ pe^{-f_t\circ p}\pi_2\right) p^\ast(\eta),\ \text{ for each} \ t.$$
  \end{proposition}
  \begin{proof}
   Assume $\Psi = \{\psi_t\}$ to be any almost cosymplectic isotopy such that 
  $\psi_t^\ast(\eta) = e^{f_t}\eta$ for all $t$, and consider $p: \tilde M\rightarrow M,$ and $\pi_2: \tilde M\rightarrow \mathbb{S}^1,$ to be projection maps. 
  From the previous proposition (Proposition \ref{Trasit-1}) the isotopy $\tilde\Psi :=  \{(\psi_t\circ p, \pi_2 e^{-f_t\circ p} )_t\},$ is a symplectic isotopy of  $(\tilde M, \tilde\omega)$. That is, the $1-$form $\iota (\dot{\tilde{\psi}}_t)\tilde{\omega}$ is closed for each $t$. On the other hand, we have\\ 
  $\dot{\tilde{\psi}}_t = \dot{\psi}_t -  \left(\dot f_t\circ pe^{-f_t\circ p}\pi_2 \right)\frac{\partial}{\partial\theta},$ 
  which implies that\\ 
  $\iota (\dot{\tilde{\psi}}_t)\tilde{\omega} = p^\ast(\iota (\dot{\psi}_t)\omega)  + \eta(\dot{\psi_t})\circ p\pi_2^\ast(d\theta)  + \left(\dot f_t\circ pe^{-f_t\circ p}\pi_2\right) p^\ast(\eta),$
  for each $t$. Therefore, differentiating the above equality in both sides gives:
  $$ d \iota (\dot{\tilde{\psi}}_t)\tilde{\omega} = d \left( p^\ast(\iota (\dot{\psi}_t)\omega)\right)   + d\left( \eta(\dot{\psi_t})\circ p\pi_2^\ast(d\theta)  + \left(\dot f_t\circ pe^{-f_t\circ p}\pi_2\right) p^\ast(\eta)\right) ,$$
  i.e.,
  $0 = 0  + d\left( \eta(\dot{\psi_t})\circ p\pi_2^\ast(d\theta)  + \left(\dot f_t\circ pe^{-f_t\circ p}\pi_2\right) p^\ast(\eta)\right),$  
  for each $t$. That is, 
  $$ d\left( \eta(\dot{\psi_t})\circ p\right) \wedge \pi_2^\ast(d\theta)  =  -d\left(\dot f_t\circ pe^{-f_t\circ p}\pi_2\right)\wedge p^\ast(\eta) ,$$  
  for each $t$. Taking the interior derivative in the above equality with respect to the vector field $ \frac{\partial}{\partial\theta}$, yields 
  $-d\left( \eta(\dot{\psi_t})\circ p\right)  = -\left(\dot f_t\circ pe^{-f_t\circ p}\right) p^\ast(\eta) ,$  \ \text{ for each} \ t. Finally, we compute 
  \begin{eqnarray} 
  d\left( \pi_2\eta(\dot{\psi_t})\circ p\right) & =& \left( \eta(\dot{\psi_t})\circ p\right) \pi_2^\ast(d\theta)  + d \left(\eta(\dot{\psi_t})\circ p\right)\pi_2 \nonumber\\
  &=&\left( \eta(\dot{\psi_t})\circ p\right) \pi_2^\ast(d\theta) + \left(\dot f_t\circ pe^{-f_t\circ p}\pi_2\right) p^\ast(\eta),
  \end{eqnarray}
  for each $t$. 
  \end{proof}
  Here is a corollary of Proposition \ref{Trasit-2}.
  
  \begin{proposition}\label{Trasit-3}
  	Let $(M, \eta, \omega)$ be a  cosymplectic manifold. 
  	If $\Psi = \{\psi_t\}$ is any  almost cosymplectic isotopy such that 
  	$\psi_t^\ast(\eta) = e^{f_t}\eta$  ( or $\mathcal{L}_{\dot{\psi}_t}\eta = \mu_t\eta$) for all $t$, then we have 
  	\begin{equation}
  	\boxed{  \mu_t   = \dot f_te^{-f_t}},
  	\end{equation}
  	or equivalently, 
  	\begin{equation}
  	\boxed{\dot f_t\circ \psi_t = \dot f_te^{f_t}},
  	\end{equation}
  	for each $t$. 
  \end{proposition}
  \begin{proof}
  Assume $\Psi = \{\psi_t\}$ to be any almost cosymplectic isotopy such that 	$\psi_t^\ast(\eta) = e^{f_t}\eta$  ( or $\mathcal{L}_{\dot{\psi}_t}\eta = \mu_t\eta$) for all $t$. From the proof of Proposition \ref{Trasit-2}, we derive that\\ 
  $d\left( \eta(\dot{\psi_t})\circ p\right)  = \left(\dot f_t\circ pe^{-f_t\circ p}\right) p^\ast(\eta) ,$ and composing the latter equality with any smooth section of the projection $p$, yields 
  $d\left( \eta(\dot{\psi_t})\right)  = \left(\dot f_te^{-f_t}\right) \eta ,$
  for each $t$.  On the other hand, from $d\left( \eta(\dot{\psi_t})\right) =  \mathcal{L}_{\dot{\psi}_t}\eta = \mu_t\eta,$ for all $t$
  we derive that $ \left(\dot f_te^{-f_t}\right) \eta = \mu_t\eta$, for each $t$. Applying the Reed vector field in both sides of the 
  latter equality implies $ \mu_t   = \dot f_te^{-f_t},$ for each $t$. From $f_t = \int_0^t\mu_s\circ\psi_s ds$, it follows that 
  $\dot f_t\circ \psi_t = \dot f_te^{f_t} ,$
  for each $t$.
  \end{proof}   
  
  \begin{proposition}\label{Trasit-4}
  	Let $(M, \eta, \omega)$ be a  cosymplectic manifold. 
  	If $\Psi = \{\psi_t\}$ is any  almost co-Hamiltonian isotopy such that 
  	$\psi_t^\ast(\eta) = e^{f_t}\eta$ for all $t$, and $\iota (\dot{\psi}_t)\omega = dH_t$ , then the isotopy $\tilde\Psi = \{\tilde\psi_t\}$ defined by 
  	$$ \tilde\psi_t: M\times \mathbb{S}^1 \longrightarrow M\times \mathbb{S}^1,(x, \theta) \mapsto (\psi_t(x), \theta e^{-f_t(x)} ),$$
  	is a Hamiltonian isotopy of  the symplectic manifold $(\tilde M, \tilde\omega)$ such that 
  	\begin{equation}
  	\boxed{\iota (\dot{\tilde{\psi}}_t)\tilde{\omega} = d\left(
  		H_t\circ p  +  \pi_2\eta(\dot{\psi_t})\circ p\right)},
  	\end{equation}
  	for each $t$.  
  \end{proposition}
  
  The following theorem shows that the Reeb vector field determines the almost cosymplectic nature of a uniform limit of a sequence of almost cosymplectic diffeomorphisms. 
  \begin{theorem}\label{Theo-Al1}
  	Let $(M, \eta, \omega)$ be a  compact cosymplectic manifold with Reeb's vector field $\xi$. 
  	Let $ \{\psi_i\}$ be any  sequence of almost co-symplectic diffeomorphisms that uniformly converges to a diffeomorphism $\psi$, and suppose that $\psi_i^\ast(\eta) = e^{f_i}\eta$. The following assertions hold.
  	\begin{enumerate}
  		\item The smooth function $\psi^\ast(\eta)(\xi)$, is non-negative.
  		\item  If the smooth function $\psi^\ast(\eta)(\xi)$, is positive, then the sequence $\{f_i\}$ uniformly converges to  
  		$ F_\psi^\xi: = \ln{\left( \psi^\ast(\eta)(\xi)\right) }$.
  		\item If the smooth function $\psi^\ast(\eta)(\xi)$ is equal to the constant function $1$, then $\psi$ is a  cosymplectic diffeomorphism: A flexibility result.
  		
  		\item If the smooth function $\psi^\ast(\eta)(\xi)$ is positive and different from the constant function $1$, then $\psi$ is an almost cosymplectic diffeomorphism
  		with $\psi^\ast(\eta) = e^{F_\psi^\xi}\eta $: A rigidity result.
  	\end{enumerate}
  \end{theorem}
  To prove the above theorem, we shall need the following lemma. 
  \begin{lemma}\label{Lem-Al1}
  	Let $(M, \eta, \omega)$ be a compact  cosymplectic manifold. If $ \{\psi_i\}$ is a  sequence of almost co-symplectic diffeomorphisms that uniformly converges to a diffeomorphism $\psi$, with $\psi_i^\ast(\eta) = e^{f_i}\eta$, then for each fixed $x\in M$, the sequence of positive real numbers 
  	$ \{e^{f_i(x)}\}$ converges to $(\psi^\ast(\eta)(\xi))(x)$, where $\xi$ is the Reeb vector field. 
  \end{lemma}
  The proof of the above lemma will need the following corollary. 
  
  \begin{corollary}\label{Cor-Al1}
  	Let $(M, \eta, \omega)$ be a compact  cosymplectic manifold. If $ \{\psi_i\}$ is a  sequence of almost co-symplectic diffeomorphisms that uniformly converges to a diffeomorphism $\psi$, with $\psi_i^\ast(\eta) = e^{f_i}\eta$, then for any smooth curve $\gamma \subset M$, we have 
  	$\lim_{i\rightarrow\infty} \int_\gamma e^{f_i}\eta = 
  	\int_\gamma \psi^\ast(\eta).$ 
  \end{corollary}
  
  \begin{proof}
  Assume $M$ to be equipped with a Riemannian metric $g$, with injectivity radius $r(g)$. Let $\gamma$ be any smooth curve in $M$. Since $\psi_i\xrightarrow{C^0}\psi$, then for $i$ sufficiently large we may assume that 
   $d_{C^0} (\psi_i,\psi)\leq r(g)/2$, and derive that, for each $t\in [0,1]$, the points $\psi_i(\gamma(t))$ to  $\psi(\gamma(t))$ can be connected 
   through a minimizing geodesic  $\chi_i^t$. This means that 
  the curves $\chi_i^0$, $\chi_i^1$, $\psi_i\circ \gamma$, and  $\psi\circ \gamma$ form the boundary of a smooth  $2-$chain $\square(\gamma, \psi_i, \psi)\subset M$. Since $d\eta = 0$, we derive from Stockes' theorem that $\int_{ \square(\gamma, \psi_i, \psi)}d\eta = 0,$ i.e., 
  $\int_{\psi_i\circ \gamma}\eta  - \int_{\psi_\circ \gamma}\eta  = \int_{\chi_i^1}\eta -  \int_{\chi_i^0}\eta,$
  for $i$ sufficiently large. On the other hand, 
  from 
  $ \psi_i^\ast(\eta) - \psi^\ast(\eta) = e^{f_i}\eta - \psi^\ast(\eta),$ we derive that 
  \begin{eqnarray} 
  |\int_{\gamma}\left( e^{f_i}\eta - \psi^\ast(\eta)\right)  | &=&  
  |\int_{\gamma}\left( \psi_i^\ast(\eta) - \psi^\ast(\eta) \right)|
  \nonumber\\
  &=&  |\int_{\chi_i^1}\eta -  \int_{\chi_i^0}\eta|,
  \nonumber\\
  &\leq& 2 |\eta |_{0} d_{C^0}(\psi_i, \psi),
  \nonumber\\
  \end{eqnarray}
  for all $i$ sufficiently large, where $|.|_0$ stands for the uniform sup norm on the space of $1-$form of a compact manifold \cite{Tchuiaga2}. The top right hand side of the above estimates tends to zero as $i$ tends to infinity.
   \end{proof}
  
  {\it Proof of Lemma \ref{Lem-Al1}.} Let $\{\phi_t\}$ be the cosymplectic flow generated by the Reeb vector field $\xi$. For each fixed $t\in ]0, 1[$, and each $x\in M$, consider the smooth curve $\bar \gamma_{x, t}: s \mapsto \phi_{st}(x)$, and derive from Corollary  \ref{Cor-Al1} that 
  $  \lim_{i\longrightarrow \infty} \int_{\bar \gamma_{x, t}}\left( e^{f_i}\eta\right) = \int_{\bar \gamma_{x, t}} \psi^\ast(\eta),$ for each fixed $t\in ]0, 1[$, i.e., 
  $ \lim_{i\longrightarrow \infty} \int_0^t\left( e^{f_i(\phi_{u}(x))} du\right) =  \int_0^t(\psi^\ast(\eta)(\xi))(\phi_{u}(x))du,$ for each fixed $t\in ]0, 1[$. Therefore, applying the mean valuer theorem for integral in a suitable way to the above equality, implies that  
  $  \lim_{i\longrightarrow \infty} e^{f_i(\phi_{t}(x))}  =  (\psi^\ast(\eta)(\xi))(\phi_{t}(x)),$ for each fixed $t\in ]0, 1[$, for all $x\in M$. This implies that 
  $ \lim_{i\longrightarrow \infty} e^{f_i(x)}  =  (\psi^\ast(\eta)(\xi))(x),$  for each $x\in M$. $\blacksquare$\\
  
  {\it Proof of Theorem \ref{Theo-Al1}.} 
  By Lemma \ref{Lem-Al1}, the smooth function $x \mapsto (\psi^\ast(\eta)(\xi))(x),$ is the uniform limit of a sequence of positive functions, hence the latter is non-negative. Assume that the smooth function 
  $x \mapsto (\psi^\ast(\eta)(\xi))(x),$ is positive. Therefore, as in Proposition \ref{Trasit-1}, we define a sequence of symplectic isotopies of the symplectic manifold $(\tilde M, \tilde \omega)$  by\\ 
  $\tilde \phi_i: = (\psi_i\circ p, \pi_2 e^{-f_i\circ p}),$
  for each $i$. Since by assumption,  $\psi_i\xrightarrow{C^0}\psi$, and by Lemma \ref{Lem-Al1} we have $ \lim_{i\longrightarrow \infty} e^{-f_i\circ p} = e^{ -\ln((\psi^\ast(\eta)(\xi)))\circ p},$
  then, it follows that 
  $\tilde \phi_i \xrightarrow{C^0} (\psi\circ p, \pi_2 e^{ -F_\xi^\psi\circ p}),$
  with $F_\xi^\psi := \ln((\psi^\ast(\eta)(\xi))).$ Since the map 
  $(\psi\circ p, \pi_2 e^{ -F_\xi^\psi\circ p})$ is a diffeomorphism, then it follows from the celebrated rigidity theorem of Elishberg-Gromov that the diffeomorphism $\Phi_\psi:= (\psi\circ p, \pi_2 e^{ -F_\xi^\psi\circ p}),$ is a symplectic diffeomorphism of $(\tilde M, \tilde \omega)$. Finally, the fact that $ \Phi_\psi^\ast(\tilde \omega) = \tilde \omega$, obviously implies that $ \psi^\ast(\omega) = \omega$, and $ \psi^\ast(\eta) = e^{F_\xi^\psi}\eta$, provided the positive function $x \mapsto (\psi^\ast(\eta)(\xi))(x)$ is different from the constant function $1$: That is, $\psi$ is an almost cosymplectic diffeomorphism. Otherwise, we have  $ \psi^\ast(\omega) = \omega$, and $ \psi^\ast(\eta) = \eta$: That is, $\psi$ is a cosymplectic diffeomorphism. 
  $\blacksquare$ 
  \subsubsection{Almost co-Hamiltonian dynamical systems (ACHDS)}
  Let $Y$ be an almost co-Hamiltonian vector field of $(M, \omega,\eta)$, and let $\Psi_Y$ be its flow. Since 
  $\imath_Y\omega = dH$, for some $H\in C^\infty(M,\mathbb{R})$, we derive that for each $p\in M$, we have 
  $$\frac{d}{ dt}H(\psi_Y^t(z)) = \omega(Y, Y)\circ\psi_Y^t(z) = 0, \ \text{ for each} \ t,$$
   and for all $z\in M$. So,  in almost cosymplectic geometry, along the orbit $t\mapsto\psi_Y^t(z)$, 
  the energy function $H$, is a constant function. Therefore, it could interesting to study the topology of orbits together 
  with  fix point theory from almost co-Hamiltonian dynamical system view point.\\
  
  Here is a consequence of Arnold's conjecture from symplectic geometry.
  
  \begin{proposition}\label{Trasit-5}
  	Let $(M, \eta, \omega)$ be a closed cosymplectic manifold. 
  	If $\Psi = \{\phi_t\}$ is any  almost co-Hamiltonian isotopy such that 
  	$\psi_t^\ast(\eta) = e^{f_t}\eta$ for all $t$, and $\iota (\dot{\psi}_t)\omega = dH_t$, then for each $t$, the map $\psi_t$ has at least one 
  	fixed point $x_t$ satisfying $f_t(x_t) = 0$.
  \end{proposition}
  \begin{proof}
   If $\Psi = \{\phi_t\}$ is any  almost co-Hamiltonian isotopy such that 
   $\psi_t^\ast(\eta) = e^{f_t}\eta$ for all $t$, then the isotopy $\tilde\Psi = \{\tilde\psi_t\}$ defined by  
   $ \tilde\psi_t: M\times \mathbb{S}^1 \longrightarrow M\times \mathbb{S}^1, (x, \theta) \mapsto (\psi_t(x), \theta e^{-f_t(x)} ),$
   is a  Hamiltonian isotopy of  the symplectic manifold $(\tilde M, \tilde\omega)$ (Proposition \ref{Trasit-4}).  Thus, 
   by Arnold's conjecture, for each $t$, the map $\tilde\psi_t $ has at least one fix point $(x_t, \theta_t)$. That is, 
   $\psi_t(x_t) = x_t$, and $\theta_t e^{-f_t(x_t)} = \theta_t$, i.e., 
   $\psi_t(x_t) = x_t$, and $f_t(x_t) = 0$, for each $t$.
  \end{proof}

  \begin{proposition}\label{Trasit-6}
  	Let $(M, \eta, \omega)$ be a closed cosymplectic manifold. Let $\psi$ be an almost co-Hamiltonian diffeomorphism such that 
  	$\psi^\ast(\eta)= e^f\eta$. 
  	If $\Psi = \{\phi_t\}$ is any  almost co-Hamiltonian isotopy with time-one map $\psi$, then  any fix point of $\psi$ is a critical 
  	point for the function $x\mapsto \left( \int_0^1\psi^\ast_s\left(\eta(\dot{\psi}_s) \right)ds\right)  (x).$
  \end{proposition}		 
  \begin{proof}
  From the formula 
  \begin{eqnarray} 
  e^f\eta -\eta &=& \psi^\ast(\eta)- \eta \nonumber\\ 
  &=& d\left( \int_0^1\psi^\ast_s\left(\eta(\dot{\psi}_s) \right)ds\right),
  \end{eqnarray}
  we derive that, for each $x\in \mathcal{F}ix(\psi)$,
  we have 
  \begin{eqnarray} 
  e^{f(x)}\eta_{|_x} -\eta_{|_x}  &=&  \left( e^f\eta -\eta\right)_{|_x} \nonumber\\ 
  &=&  d\left( \int_0^1\psi^\ast_s\left(\eta(\dot{\psi}_s) \right)ds\right)_{|_x}, 
  \end{eqnarray}   	
  i.e., $ 0 = d\left( \int_0^1\psi^\ast_s\left(\eta(\dot{\psi}_s) \right)ds\right)_{|_x} ,$
  since by Proposition \ref{Trasit-5}, we have 
  $f(y) = 0$, whenever  $y\in \mathcal{F}ix(\psi)$. 
  \end{proof}
  \subsubsection*{Almost co-Hamiltonian equations:}
  Let $(z,q_1,\dots,q_n, p_1,\dots,p_n)$ be the cosymplectic-Darboux coordinates system such that\\ $ \omega = \Sigma_{i = 1}^n dq_i\wedge dp_i$, and 
  $\eta = dz$. Let $\{H_t\}$ and $\{\mu_t\} $ be smooth families of smooth functions on $M$ such that 
  $$ 
  \left\{
  \begin{array}{l}
  \dot p_i =-\frac{\partial}{ \partial q_i}H_t, \\
  \dot q_i = \frac{\partial}{ \partial p_i}H_t, \\
  \frac{\partial}{ \partial z}(\dot z) = \mu^t.\\
  \end{array}
  \right. 
  $$  
  If $X_t$ is a smooth family of vector fields on $M$ such that $ \iota(X_t)\omega = d H_t$, and $d(\eta(X_t)) = \mu^t\eta$, then
  the one-parameter family of local diffeomorphisms $ \Phi : = \{\phi_t\}$ generated 
  by $ X_t$ consists of almost co-Hamiltonian diffeomorphisms. Conversely, let $ \Phi : = \{\phi_t\}$ be an almost 
  co-Hamiltonian isotopy such that 
  $ \iota(\dot \phi_t)\omega = d L_t$, and $d(\eta(\dot \phi_t)) = \nu^t\eta$, then 
  $$ 
  (ACHE) \left\{
  \begin{array}{l}
  \dot p_i =-\frac{\partial}{ \partial q_i}L_t, \\
  \dot q_i = \frac{\partial}{ \partial p_i}L_t, \\
  \frac{\partial}{ \partial z}(\dot z) = \nu^t.\\
  \end{array}
  \right. 
  $$  
  \section{Cosymplectic flux geometries}\label{Flux}
  Let $\{\phi_t\} $ be either an almost cosymplectic,  or a cosymplectic isotopy. Since the $1-$form $\widetilde I_{\eta, \omega}(\dot\phi_t)$ is closed for each $t$, then its defines a de Rham cohomology class\\
  $[\widetilde I_{\eta, \omega}(\dot\phi_t)]\in H^1(M,\mathbb{R})$. Therefore, to any such an isotopy $\{\phi_t\},$ corresponds a unique de Rham cohomology class 
  \begin{equation}
  \displaystyle \int_0^1[\phi_t^\ast(\widetilde I_{\eta, \omega}(\dot\phi_t))]dt\in H^1(M,\mathbb{R}),
  \end{equation}
  i.e., we have group homomorphisms 
  \begin{eqnarray}\begin{array}{rcccc}
  \mathcal{A}\widetilde{S}_{\eta, \omega} : & \mathcal{A}Iso_{\eta, \omega}(M) &\longrightarrow & H^1(M,\mathbb{R})\\
  & \{\psi_t\}& \longmapsto &\displaystyle \int_0^1[\psi_t^\ast(\widetilde I_{\eta, \omega}(\dot\psi_t))]dt, 
  \end{array}\nonumber
  \end{eqnarray}
  and
  \begin{eqnarray}\begin{array}{rcccc}
  \widetilde{S}_{\eta, \omega} : & Iso_{\eta, \omega}(M) &\longrightarrow & H^1(M,\mathbb{R})\\
  & \{\phi_t\}& \longmapsto &\displaystyle \int_0^1[\phi_t^\ast(\widetilde I_{\eta, \omega}(\dot\phi_t))]dt.
  \end{array}\nonumber
  \end{eqnarray}
  
  \begin{proposition}\label{prof1} 
  	The  map $\widetilde{S}_{\eta, \omega}$ is a continuous group homomorphism whose kernel contains  $ CH_{\eta, \omega}(M)$.
  \end{proposition}
  \begin{proof} The fact that $\widetilde{S}_{\eta, \omega}$ is a group homomorphism follows from Subsection \ref{rem1}. 
  \end{proof}
  \begin{proposition}\label{prof1A} 
  	The  map $ \mathcal{A}\widetilde{S}_{\eta, \omega}$ is a continuous group homomorphism.
  \end{proposition}
  \begin{proof} This follows from Subsection \ref{rem1}. 
  \end{proof}
  \begin{proposition}\label{prof2}
  	If $\{\phi_t\}, \{\psi_t\} \in Iso_{\eta, \omega}(M) $ are homotopic relatively to fixed endpoints, then  
  	$\widetilde{S}_{\eta, \omega}(\{\psi_t\}) = \widetilde{S}_{\eta, \omega}(\{\phi_t\})$.
  \end{proposition}
  \begin{proof} This is a consequence of similar result from symplectic case \cite{Banyaga78}.
  	\end{proof}
  \begin{proposition}\label{prof2A}
  	If $\{\phi_t\}, \{\psi_t\} \in  \mathcal{A}Iso_{\eta, \omega}(M) $ are homotopic relatively to fixed endpoints, then  $ \mathcal{A}\widetilde{S}_{\eta, \omega}(\{\psi_t\}) =  \mathcal{A}\widetilde{S}_{\eta, \omega}(\{\phi_t\})$.
  \end{proposition}
  
  Let $\sim$ be an equivalence relation on 
  $ Iso_{\eta, \omega}(M) $ (resp. $  \mathcal{A}Iso_{\eta, \omega}(M) $) defined by: $\Phi\sim \Psi$ if and only if, $\Phi$ and $\Psi$ are homotopic 
  relatively to fixed endpoints. Let $\widetilde{Iso_{\eta, \omega}(M)}$ (resp. $\widetilde{ \mathcal{A}Iso_{\eta, \omega}(M)}$ ) be the quotient space of the above equivalence relation. Let $\pi_1(G_{\eta, \omega}(M))$ (resp. $\pi_1( \mathcal{A}G_{\eta, \omega}(M))$) be the first fundamental group of $ G_{\eta, \omega}(M)$ (resp. $ \mathcal{A}G_{\eta, \omega}(M)$) , and set 
  $\Gamma_{\eta, \omega} = \widetilde{S}_{\eta, \omega}(\pi_1(G_{\eta, \omega}(M)))$ (resp. $ \mathcal{A}\Gamma_{\eta, \omega} =  \mathcal{A}\widetilde{S}_{\eta, \omega}(\pi_1( \mathcal{A}G_{\eta, \omega}(M)))$).

  The mapping $\widetilde{S}_{\eta, \omega} $ induces a map $S_{\eta, \omega}$ from $\widetilde{Iso_{\eta, \omega}(M)}$ onto 
  $H^1(M,\mathbb{R})/\Gamma_{\eta, \omega}$ such that the following diagram commutes:  
  $$
  \begin{array}{lcr}
  \widetilde{Iso_{\eta, \omega}(M)}  &\xrightarrow{\widetilde{S}_{\eta, \omega}} & {H}^1(M,\mathbb{R}) \\
  \hspace{0.4cm} ev_1 \downarrow &                                                  &\downarrow \pi \hspace{0.4cm}\\ 
  G_{\eta, \omega}(M) &      \xrightarrow{S_{\eta, \omega}}                        & {H}^1(M,\mathbb{R})/\Gamma_{\eta, \omega},
  \end{array}
  \hspace{1cm}(I)
  $$
  where $\pi$ is the quotient map, and $ev_1$ is the time-one map evaluation: 
  \begin{equation}
  \pi\circ\widetilde{S}_{\eta, \omega} = S_{\eta, \omega}\circ ev_1.
  \end{equation}
  Similarly, the mapping $ \mathcal{A}\widetilde{S}_{\eta, \omega}$ induces a map $ \mathcal{A}S_{\eta, \omega}$ from $\widetilde{ \mathcal{A}Iso_{\eta, \omega}(M)}$ onto 
  $H^1(M,\mathbb{R})/ \mathcal{A}\Gamma_{\eta, \omega}$ such that the following diagram commutes:  
  $$
  \begin{array}{lcr}
  \widetilde{ \mathcal{A}Iso_{\eta, \omega}(M)}  &\xrightarrow{ \mathcal{A}\widetilde{S}_{\eta, \omega}} & {H}^1(M,\mathbb{R}) \\
  \hspace{0.4cm}  ev_1 \downarrow &                                                  &\downarrow \pi \hspace{0.4cm}\\ 
  \mathcal{A}G_{\eta, \omega}(M) &      \xrightarrow{ \mathcal{A}S_{\eta, \omega}}                        & {H}^1(M,\mathbb{R})/ \mathcal{A}\Gamma_{\eta, \omega},
  \end{array}
  \hspace{1cm}(II)
  $$
  where $\pi$ is the quotient map, and $ev_1$ is the time-one map evaluation: 
  \begin{equation}
  \pi\circ \mathcal{A}\widetilde{S}_{\eta, \omega} =  \mathcal{A}S_{\eta, \omega}\circ ev_1.
  \end{equation}
  
  We have the following factorization result: 
  \begin{proposition}\label{fact}
  	Let $\Phi:= \{\phi_t\} \in Iso_{\eta, \omega}(M) $, and $\alpha$ be any closed $1-$form of $M$. Then
  	
  	\begin{equation}\label{fact1}
  	\int_M \Delta(\Phi,\eta)\alpha\wedge\omega^n - 	\int_M \Delta(\Phi,\alpha) \eta\wedge\omega^n = n!\langle[\alpha], [\omega^{(n-1)}\wedge \eta]\wedge \widetilde{S}_{\eta, \omega}(\Phi)\rangle
  	\end{equation}
  	where $\Delta(\Phi, \bullet) = \displaystyle   \int_0^1\bullet(\dot{\phi}_t)\circ\phi_t dt $ and $\langle\cdot, \cdot\rangle$ is the usual Poincar\'e pairing. 
  \end{proposition}
  \begin{proof} This is a direct consequence of the fact that for each $\Phi:= \{\phi_t\} \in Iso_{\eta, \omega}(M) $, and $\alpha$ any closed $1-$form of $M$, 
  	we have $\imath_{\dot \phi_t}(\alpha \wedge \eta\wedge \omega^n ) = 0$, for each $t$, combined with Stokes' theorem. 
  \end{proof}
  
  \begin{proposition}\label{factA}
  	Let $\Phi:= \{\phi_t\} \in  \mathcal{A}Iso_{\eta, \omega}(M) $ such that $\phi_t^\ast(\eta) = e^{f_t}\eta$ for each $t$, and $\alpha$ be any closed $1-$form of $M$. Then
  	\begin{equation}\label{fact1A}
  	\int_M \left( \int_0^1e^{-f_t}C_{\Phi, \eta}^tdt\right)\alpha\wedge\omega^n - 	\int_M \Delta(\Phi,\alpha) \eta\wedge\omega^n =  n!\langle[\alpha], [\omega^{(n-1)}\wedge \eta]\wedge  \mathcal{A}\widetilde{S}_{\eta, \omega}(\Phi)\rangle.
  	\end{equation}
  \end{proposition}
  
  Let $\Psi:= \{\psi_t\} \in  \mathcal{A}Iso_{\eta, \omega}(M) $ be a loop. From Proposition \ref{factA}, we derive that 
  \begin{equation}
  \int_M \left( \int_0^1e^{-f_t}C_{\Psi, \eta}^tdt\right)\alpha\wedge\omega^n - \left( \int_{\mathcal{O}_x^\Psi}\alpha\right)  \left( \int _M \eta \wedge\omega^n\right) 
  =  n!\langle[\alpha], [\omega^{(n-1)}\wedge \eta]\wedge  \mathcal{A}\widetilde{S}_{\eta, \omega}(\Phi)\rangle,
  \end{equation}
  for each closed $1-$form $\alpha$: In particular, if the isotopy $\Psi$ is an almost co-Hamiltonian loop, then
  \begin{equation}
  \int_{\mathcal{O}_x^\Psi}\alpha = \frac{1}{ Vol_{\eta, \omega}(M)}\int_M \left( \int_0^1e^{-f_t}C_{\Psi, \eta}^tdt\right)\alpha\wedge\omega^n,
  \end{equation}
  for all $x\in M$, and for each closed $1-$form $\alpha$. That is, $\Psi$ will have at least  one contractible orbit if and only if 
  \begin{equation}
  \int_M \left( \int_0^1e^{-f_t}C_{\Psi, \eta}^tdt\right)\alpha\wedge\omega^n = 0,
  \end{equation}
  for each closed $1-$form $\alpha$.\\ 
  
  We have the following facts. 
  
  \begin{proposition}\label{Decom-1}
  	Let $\Phi:= \{\phi_t\} \in Iso_{\eta, \omega}(M) $  such that $\Phi= \Phi_\omega\circ\Phi_\eta$.  Then, 
  	\begin{itemize}
  		\item $ \widetilde{S}_{\eta, \omega}(\Phi_\omega) = \int_0^1[\iota(\dot \phi_t)\omega]dt$,
  		\item $ \widetilde{S}_{\eta, \omega}(\Phi_\eta) = \left( \int_0^1 \eta(\dot \phi_t)\circ\phi_tdt\right) [\eta]$.
  	\end{itemize}
  \end{proposition}
  
  \begin{proposition}\label{ADecom-1}
  	Let $\Psi:= \{\psi_t\} \in \mathcal{A}Iso_{\eta, \omega}(M) $ such that 
  	$\Psi= \Psi_\omega\circ\Psi_\eta $. The following results hold.
  	\begin{itemize}
  		\item $\widetilde{S}_{\eta, \omega}(\Psi_\omega) = \int_0^1[\iota(\dot \psi_t)\omega]dt$. In particular, when $\Psi$ is an almost co-Hamiltonian isotopy.
  		then $\widetilde{S}_{\eta, \omega}(\Psi_\omega) = 0$. 
  		\item $ \mathcal{A}\widetilde{S}_{\eta, \omega}(\Psi_\eta) =  [\left( \int_0^1 \eta(\dot \psi_t)\circ\psi_tdt\right)\eta]$.
  	\end{itemize}
  \end{proposition}

  \subsection{Reeb dynamics} 
  
  Let $\Phi:= \{\phi_t\} \in Iso_{\eta, \omega}(M) $ be a loop. From Proposition \ref{fact}, we derive that for each $x\in M$, if $\mathcal{O}_x^\Phi$ stands for the orbit 
  of $x$ under $\Phi$, then 
  \begin{equation}
  \left( \int_{\mathcal{O}_x^\Phi}\eta\right)  \left( \int _M \alpha \wedge\omega^n\right) - \left( \int_{\mathcal{O}_x^\Phi}\alpha\right)  \left( \int _M \eta \wedge\omega^n\right) 
  =  n!\langle[\alpha], [\omega^{(n-1)}\wedge \eta]\wedge \widetilde{S}_{\eta, \omega}(\Phi)\rangle, 
  \end{equation}
  for each closed $1-$form $\alpha$. 
  \begin{itemize}
  	\item If some  $y\in M$ has a contractible orbit under $\Phi$, then $ [\omega^{(n-1)}\wedge \eta]\wedge \widetilde{S}_{\eta, \omega}(\Phi) = 0$.
  	\item If $ \widetilde{S}_{\eta, \omega}(\Phi) = 0$, then $
  	\displaystyle \int_{\mathcal{O}_x^\Phi}\alpha  = 
  	\left( \frac{\int _M \alpha \wedge\omega^n}{ \int _M \eta \wedge\omega^n}\right) \int_{\mathcal{O}_x^\Phi}\eta,
 $
  	for all closed $1-$form $\alpha$, for all $x\in M$. 
  \end{itemize}

  Here is a consequence of the previous items. 
  \begin{proposition}\label{prof2-B}
  	Let $(M,\eta,\omega)$ be a smooth closed cosymplectic manifold of dimension\\  $(2n + 1)$.  Assume that the flow $\Phi_\eta$  of the Reeb vector field 
  	is a loop. Then, for each closed $1-$form $\alpha$, the smooth map $x\mapsto \int_{\mathcal{O}_x^{\Phi_\eta}}\alpha,$ is constant, and it is completely determined by the integral $\int _M \alpha \wedge\omega^n.$
  \end{proposition}
  Proposition \ref{prof2}  seems to tell us that on such a cosymplectic manifold $M$, 
  one can describe some topological properties of $M$ by the mean of such a flow $\Phi_\eta $. 
  \begin{proposition}\label{prof3}
  	The group $ \ker S_{\eta, \omega} $ is path connected by smooth arcs.
  \end{proposition}
  
  \begin{proof} Let $\phi\in\ker S_{\eta, \omega} $, by definition, there exists  $\Phi:= \{\phi_t\} \in Iso_{\eta, \omega}(M) $ with $\phi_1 = h$ such that 
  	$\displaystyle \int_0^1 \widetilde I_{\eta, \omega}(\dot\phi_t)dt = df,$ for some smooth function on $M$. As in \cite{Banyaga78}, consider the $2-$parameters family 
  	of cosymplectic vector field defined by 
  	$ X_{s,t} (\phi_{st}(x)) = \dfrac{\partial}{\partial s}(\phi_{st}(x)),$
  	for all $x\in M$, then set $\displaystyle  \alpha_t = \int_0^1 \widetilde I_{\eta, \omega}(X_{s,t})ds,$
  	and define another smooth family $(Y_t) $ of cosymplectic vector field such that $ \widetilde I_{\eta, \omega}(Y_{t}) = \alpha_t -t\alpha_1 =: \beta_t,$ 
  	for each $t$. Compute 
  	\begin{equation}
  	\displaystyle  \int_0^1 \widetilde I_{\eta, \omega}(X_{s,t} - Y_t)ds = d(tf),
  	\end{equation}
  	for each $t$, and set $Z_{s,t}:= X_{s,t} - Y_t$, for each $s$, for each $t$. The 
  	$2-$parameters family 
  	of cosymplectic diffeomorphisms ${H_{s,t}}$ defined by 
  	$ Z_{s,t} (H_{s,t}(x)) = \dfrac{\partial}{\partial s}(H_{s,t}(x)),$
  	for all $x\in M$, satisfies $ H_{s, 1} = \phi_s$, for all $s$,  $ H_{1, t} \in \ker S_{\eta, \omega} $, for each $t$. So, $t\longmapsto H_{1, t} $ is a smooth path 
  	in $\ker S_{\eta, \omega} $ with time-one map $\phi$. 
  \end{proof}
  \section{Geometry of cosymplectic diffeomorphisms}\label{Hofer-N}
  In order to describe further structures of the groups $G_{\eta, \omega}(M) $ and $\mathcal{A}G_{\eta, \omega}(M)  $ we shall need the comparison 
  results established in the following subsection. 
  \subsection{\bf Comparison of norms} \label{Hofer-N-1}
  Let  $M$, and $N$ be two smooth compact manifolds, and put $ \widetilde{M}:= M\times N$. Let $q: \widetilde M\rightarrow N$ and $p: \widetilde M\rightarrow M$  be  projection maps. Let us equip $M$ with a Riemannian metric $g^M$, equip $N$ with its natural metric $g^N$ and  denote by $\tilde g$ the corresponding induced product metric on 
  $\widetilde M$.
 \subsubsection{Comparison of the norms $|p^\ast(\alpha)|_0$ and $|\alpha|_{0}$, for each $\alpha\in\Omega^1(M)$  } 
   Consider  a $1-$ form   $\alpha$ on $M$ and let us recall the definition of the supremum norm (i.e., the uniform sup norm)  of $\alpha$: 
   for each $x\in M$, we know  
  that $\alpha$ induces a linear map $\alpha_x : T_xM \rightarrow \mathbb{R},$ 
  whose norm is given by 
  \begin{equation}
  \|\alpha_x\|^{g^M} := \sup\bigg\{ |\alpha_x(X)| \ ; \ X\in T_xM,  \|X\|_{g^M} = 1\bigg\},
  \end{equation}
  where $\|\cdot\|_{g^M}$ is the norm induced 
  on each tangent space $ T_xM$ (at the point $x$) by the Riemannian metric $g^M$. 
  Therefore, the uniform sup norm of $\alpha$, 
  say $|\cdot|_{0}$, is defined as
  \begin{equation}
  |\alpha|_{0} := \sup_{x\in M}\|\alpha_x\|^{g^M}.
  \end{equation}
  On the other hand, since $p^\ast\alpha$ is a $1-$form on $\widetilde M$, then for each $(x, y)\in \widetilde M$, we have 
  \begin{eqnarray}\begin{array}{cclccccccc}
  \|p^\ast(\alpha)_{|(x,y)}\|^{\tilde g} &= & \sup\bigg\{ |\alpha_x(p_\ast(Y))|  \ ; \  Y\in T_{(x,y)}\tilde M,  \|Y\|_{\tilde g} = 1\bigg\}\nonumber\\
  &=&  \sup\bigg\{ |\alpha_x(Y_1)|  \ ; \  (Y_1 + Y_2) \in T_{x} M\oplus T_y N\ \text{ and }\  \|Y_1\|_{g^M} +  \|Y_2\|_{g^N}= 1\bigg\}\nonumber \\
  & \leqslant & \sup\bigg\{ |\alpha_x(Y_1)|  \ ; \  Y_1  \in T_{x} M, \|Y_1\|_{g^M} \leqslant 1\bigg\},
  \end{array}\end{eqnarray}
  where $\|\cdot\|_{g^M}$ (resp. $\|\cdot\|_{g^N}$) is the norm induced on each tangent space $ T_xM$ (resp. $T_y N$) by
  the Riemannian metric $g^M$ (resp. $g^N$). Therefore, 
  \begin{equation}\label{Compar1}
  \|p^\ast(\alpha)_{|(x,y)}\|^{\tilde g} \leqslant  \|\alpha_x\|^{g^M},
  \end{equation}
 for each $(x,y)\in \tilde M$, implies that
  \begin{equation}\label{Compar 2}
  |p^\ast(\alpha)|_{0}  \leqslant  |\alpha|_{0}. 
  \end{equation}
  
  \subsubsection{Splitting of closed $1$-forms and the uniform sup norm }\label{SC01}
  Let $H^1(M,\mathbb{R})$ (resp. $H^1(\widetilde M,\mathbb{R})$) denote the first de Rham cohomology group (with real coefficients) of $M$ (resp. $\widetilde M$)
  and let $\mathcal{Z}^1(M)$ (resp. $\mathcal{Z}^1(\widetilde M)$) denote the space of all closed $1-$forms on $M$ (resp. $\widetilde M$).  
  Consider the map 
  \begin{equation}\label{eq1}
  \mathcal{S} :   H^1(M,\mathbb{R})\longrightarrow \mathcal{Z}^1(M),
  \end{equation}
  to be a fixed linear section of the natural 
  projection 
  \begin{equation}\label{eq2}
  \mathcal{\pi}_M: \mathcal{Z}^1(M)\longrightarrow  H^1(M,\mathbb{R}).
  \end{equation}  
  Each $\alpha\in \mathcal{Z}^1(M)$ splits as :
  \begin{equation}\label{eq3}
  \alpha = \mathcal{S}(\mathcal{\pi}_M(\alpha)) + (\alpha - \mathcal{S}(\mathcal{\pi}_M(\alpha))).
  \end{equation}
  We shall call the $1-$form $(\alpha - \mathcal{S}(\mathcal{\pi}_M(\alpha)))$ the exact part of $\alpha$ and throughout 
  all the paper, for simplicity, when this will be necessary, the  latter $1-$form will be denoted $df_{\alpha,\mathcal{S} }$  to mean that it is the differential of a certain function  that depends on $\alpha$ and $\mathcal{S}$ ; while we shall  call the $1-$form $\mathcal{S}(\mathcal{\pi}_M(\alpha))$ the $\mathcal{S}-$form of $\alpha$.  Let $\mathbb H^1(M,\mathcal{S})$ denote the space of all $\mathcal{S}-$forms and define the set 
  $\mathbb{B}^1(M)$ as  :   $$\mathbb{B}^1(M):= (\mathcal{Z}^1(M) \smallsetminus  \mathbb H^1(M,\mathcal{S}))\cup\{0\}.$$
  We then  have the following direct sum : 
  \begin{equation}\label{eq4}
  \mathcal{Z}^1(M) = \mathbb H^1(M,\mathcal{S})\oplus_\mathcal{S}\mathbb{B}^1(M),
  \end{equation}
  with 
  \begin{equation}
  \dim(\mathbb H^1(M,\mathcal{S})) = \dim(H^1(M,\mathbb{R}))< \infty,
  \end{equation}
  for each  linear section $\mathcal{S} $ (see \cite{Tchuiaga2}).\\
  
  Denote by $ \mathcal P\mathbb H^1(M,\mathcal{S})$ the space of all smooth mappings 
  $\mathcal{H}: [0,1]\rightarrow\mathbb H^1(M,\mathcal{S})$. Since both spaces $\mathbb H^1(M,\mathcal{S})$ and $H^1(M,\mathbb{R})$ are isomorphic and  $H^1(M,\mathbb{R})$ is a finite dimensional vector space whose dimension is the first Betti number $b_1(M)$, then
  $\mathbb H^1(M,\mathcal{S})$ is of finite dimension \cite{DeRam}. Thus, there exists a positive constants 
  $K_1(g)$ and $k_2(g)$ which depend on the Riemannian metric $g$ on $M$ such that  
  \begin{equation}\label{sharp 1}
  k_1(g)\|\alpha\|_{L^2} \leqslant |\alpha|_{0} \leqslant k_2(g)\|\alpha\|_{L^2},
  \end{equation}
  for all $ \alpha\in \mathbb H^1(M,\mathcal{S})$. On the other hand, consider the projection $p : \widetilde M\longrightarrow M$, and let 
  \begin{equation}
  \mathcal{\pi}_{\widetilde{M}}: \mathcal{Z}^1(\widetilde M)\longrightarrow  H^1(\widetilde M,\mathbb{R}).
  \end{equation} 
  be the canonical projection, where  $\mathcal{Z}^1(\widetilde M) $ is the set of all closed $1-$forms on $\widetilde M$:  we have the commutative diagram 
  \begin{eqnarray}
  \begin{array}{cccccccccc}
  \mathcal{Z}^1( M) & \stackrel{p^\ast}{\longrightarrow} & \mathcal{Z}^1(\widetilde M)\\[.1cm]
  \mathcal{\pi}_{M}\downarrow &   & \downarrow \mathcal{\pi}_{\widetilde{M}} \\[.1cm]
  H^1(M,\mathbb{R}) & \stackrel{ p^\ast}{\longrightarrow}  &  H^1(\widetilde M,\mathbb{R})
  \end{array}\nonumber 
  \end{eqnarray}
  namely $ p^\ast\circ  \mathcal{\pi}_{M} =  \mathcal{\pi}_{\widetilde M}\circ p^\ast$. The following composition of linear mappings 
  \begin{eqnarray}\begin{array}{ccccccc}
  \mathbb H^1(M,\mathcal{S})& \xrightarrow{ \ \pi_M \ } &H^1(M,\mathbb{R})&\xrightarrow{\ p^\ast } &H^1(\widetilde M,\mathbb{R})\nonumber\\
  \alpha & \longmapsto& \pi_M(\alpha) &\longmapsto & p^\ast(\pi_M (\alpha)),
  \end{array}\end{eqnarray}
  is continuous, then there is a constant $\kappa_0$ such that 
  \begin{equation}
  \|  \pi_{\widetilde{M}}(p^\ast(\alpha))\|_{L^2} = \|  p^\ast(\pi_M (\alpha)) \|_{L^2} \leqslant \kappa_0 |\alpha|_0,
  \end{equation}
  since from the commutation of the previous diagram we have $ p^\ast\circ  \mathcal{\pi}_{M} =  \mathcal{\pi}_{\widetilde M}\circ p^\ast$. \\
  Let $\tilde S : H^1(\widetilde M,\mathbb{R})\longrightarrow \mathcal{Z}^1(\widetilde M)$ be any fixed linear section of $\pi_{\widetilde{M}}$, then 
  there exists a positive constant $\upsilon_0$ such that
  \begin{equation}
  |\tilde S(\pi_{\widetilde{M}}(\theta)) |_0\leqslant \upsilon_0  \|  \pi_{\widetilde{M}}(\theta) \|_{L^2}, \qquad \forall \ \theta \in \mathcal{Z}^1(\widetilde M).
  \end{equation}
  Summarizing the above inequalities gives:   
  \begin{equation}\label{sharp 2}
  |\tilde S\big(\pi_{\widetilde{M}}(p^\ast(\mathcal{S}(\pi_M(\alpha))))\big) |_0 \leqslant 
  \upsilon_0 \| \pi_{\widetilde{M}} \big(p^\ast(\mathcal{S}(\pi_M(\alpha)))\big) \|_{L^2}  
  \leqslant \upsilon_0\kappa_0 |\mathcal{S}(\pi_M(\alpha))|_0, \qquad \forall \ \alpha \in \mathcal{Z}^1( M).
  \end{equation}
 \subsection{Co-Hofer-like geometries} \label{Hofer-N-2}
  For any $X\in \chi_{\eta, \omega}(M)$, the closed $1-$forms $\imath_X\omega$ and $\eta(X)\eta$ split as  follows:
  \begin{equation}\label{equ5}
  \imath_X\omega = \mathcal{H}_\omega + dU_\omega,
  \qquad \text{ and } \qquad 
  \eta(X)\eta = \mathcal{K}_\eta + d V_\eta.
  \end{equation}
  Hence, the closed $1-$form, $\widetilde I_{\eta, \omega}(X) $ splits as : 
  \begin{equation}\label{equ6}
  \widetilde I_{\eta, \omega}(X) = \left(\mathcal{K}_\eta + \mathcal{H}_\omega \right) + d\left( U_\omega + V_\eta\right).
  \end{equation}
  From the above splitting, one defines a norm $\|\cdot\|_{C}^\mathcal{S} $ on 
  $\chi_{\eta, \omega}(M)$ as follows: for any  \\ $X\in \chi_{\eta, \omega}(M)$, 
  \begin{equation}\label{equ7}
  \|X\|_{C}^\mathcal{S} := \|\mathcal{K}_\eta + \mathcal{H}_\omega\|_{L^2} + \nu^B(d(U_\omega + V_\eta)) + | \eta(X)|,
  \end{equation}
  where $ \|\cdot\|_{L^2}$ is the $L^2-$Hodge norm and $ \nu^B$ is any norm on $ \mathbb{B}^1(M)$ which we assume to be  equivalent to the oscillation norm  (see \cite{Tchuiaga2}) : 
  $$osc(df)= \max_xf(x) -\min_xf(x), \qquad \forall \ f\in C^\infty(M). $$
  \begin{theorem}\label{thm2}
  	Let $(M, \omega, \eta)$ be a compact cosymplectic manifold, let $ \mathcal{S}$ and $\mathcal{T}$ be two linear sections of the projection 
  	$\mathcal{\pi} : \mathcal{Z}^1(M)\rightarrow  H^1(M,\mathbb{R}) $. Then, the two norms $\|\cdot \|_{C}^\mathcal{S} $ and $\|\cdot\|_{C}^\mathcal{T} $ are equivalent. 
  \end{theorem}
  
  \begin{proof} Let $X$ be a strict cosymplectic vector field such that 
  	$$\widetilde I_{\eta, \omega}(X) = \left(\mathcal{K}_\eta^\mathcal{S} + \mathcal{H}_\omega^\mathcal{S} \right) + d\left( U_\omega^\mathcal{S} + V_\eta^\mathcal{S}\right) ,$$ with respect to the  $\mathcal{S}-$decomposition, and 
  	$$\widetilde I_{\eta, \omega}(X) = \left(\mathcal{K}_\eta^\mathcal{T} + \mathcal{H}_\omega^\mathcal{T} \right) + d\left( U_\omega^\mathcal{T} + V_\eta^\mathcal{T}\right),$$ 
  	with respect to the  $\mathcal{T}-$decomposition. It is enough to show that there exists $C_1> 0$, and $C_2> 0$ such that 
  	$$C_1 \|X\|_{C}^\mathcal{T}  \leqslant \|X\|_{C}^\mathcal{S}  \leqslant C_2 \|X\|_{C}^\mathcal{T} .$$
  	Since $\dim(\mathbb H^1(M,\mathcal{S})) = \dim(H^1(M,\mathbb{R})) = \dim(\mathbb H^1(M,\mathcal{T}))< \infty$, then all the norms on each of the spaces 
  	$H^1(M,\mathcal{S})$ and $H^1(M,\mathcal{T})$ we shall equip  $H^1(M,\mathcal{S})$ with a basis $\mathbf{B}$ (resp. $H^1(M,\mathcal{T})$ with a basis $\mathbf{B}'$) 
  	and denote by  $\|\cdot\|_{\mathbf{B}}$ (resp.  $\|\cdot\|_{\mathbf{B}'}$)  the  corresponding norm. So, we only have to show that 
  	\begin{equation}\label{EQL1}
  	C_1 (\nu^B(d(U_\omega^\mathcal{T} + V_\eta^\mathcal{T})) + \|\mathcal{K}_\eta^\mathcal{T} + \mathcal{H}_\omega^\mathcal{T} \|_{\mathbf{B}'} + 
  	| \eta(X)|) \leqslant (\nu^B(d( U_\omega^\mathcal{S} + V_\eta^\mathcal{S})) + \lVert\mathcal{K}_\eta^\mathcal{S} + \mathcal{H}_\omega^\mathcal{S} \rVert_{\mathbf{B}} + | \eta(X)|),
  	\end{equation}
  	and
  	\begin{equation}\label{EQL2}
  	(\nu^B(d( U_\omega^\mathcal{S} + V_\eta^\mathcal{S})) + \|\mathcal{K}_\eta^\mathcal{S} + \mathcal{H}_\omega^\mathcal{S} \|_{\mathbf{B}} + | \eta(X)|) \leqslant C_2 (\nu^B(d(U_\omega^\mathcal{T} + V_\eta^\mathcal{T})) + \| \mathcal{K}_\eta^\mathcal{T} + \mathcal{H}_\omega^\mathcal{T}\|_{\mathbf{B}'} 
  	+ | \eta(X)|).
  	\end{equation}
  	
  	The inequalities (\ref{EQL1}) and  (\ref{EQL2}) follow from similar arguments to those used in Banyaga \cite{Banyaga08} for Hodge's decomposition. But, here 
  	the uniqueness of harmonic part in Hodge's decomposition is replaced by the fact that $\mathbb H^1(M,\mathcal{S}) \cap \mathbb{B}^1(M) = \{0\} $ 
  \\	(resp. $\mathbb H^1(M,\mathcal{T}) \cap \mathbb{B}^1(M) = \{0\} $). 
  \end{proof}
  
  Base on Theorem \ref{thm2}, we shall denote the norm $\|\cdot\|_{C}^\mathcal{S} $, simply by $\|\cdot\|_{C}$ no matter the choice of the 
  linear section  $\mathcal{S}$. 
  \subsubsection{Co-Hofer-like lengths} 
  Let $ \Phi = \{\phi_t\}\in Iso_{\eta, \omega}(M)$,
  for each $t$, we have  
  $$ \|\dot{\phi}_t\|_{C} := \|\mathcal{K}_\eta^t + \mathcal{H}_\omega^t\|_{L^2} + osc(U_\omega^t + V_\eta^t) + | C_{\Phi,\eta}^t|.$$ Therefore, we define the $L^{(1, \infty)}-$version of the co-Hofer-like length of $\Phi: =\{\phi_t\}$ as:
  \begin{equation}\label{equ8}
  l_{Co}^{(1,\infty)} (\Phi) : = \int_0^1\|\dot{\phi}_t\|_{C}dt,
  \end{equation}
  and, $L^{\infty}-$version of the co-Hofer-like length of $\Phi$ as  :
  \begin{equation}\label{equ9}
  l_{Co}^{\infty} (\Phi) : = \max_{t\in [0,1]}\|\dot{\phi}_t\|_{C}. 
  \end{equation}
  Since $$\widetilde I_{\eta, \omega}(\dot{\phi}_t) = \left(\mathcal{K}_\eta^t + \mathcal{H}_\omega^t \right) + d\left( U_\omega^t + V_\eta^t\right),$$ 
  for each $t$, we have 
  $\widetilde I_{\eta, \omega}(\dot{\phi}_{-t}) = -\phi_t^\ast\left(\mathcal{K}_\eta^t + \mathcal{H}_\omega^t \right) - d\left( U_\omega^t\circ\phi_t + V_\eta^t\circ\phi_t\right),$
  i.e.,
  $$\widetilde I_{\eta, \omega}(\dot{\phi}_{-t}) = -\left(\mathcal{K}_\eta^t + \mathcal{H}_\omega^t \right) - d\left( U_\omega^t\circ\phi_t + V_\eta^t\circ\phi_t +\Delta_t(\mathcal{K}_\eta + \mathcal{H}_\omega,\Phi)\right) ,$$ 
  with $\Delta_t(\alpha,\Phi)
  :=\displaystyle  \int_0^t\alpha_t(\dot{\phi}_{s})\circ\phi_{s}ds$. Hence, we see that in general, we  may have
  \begin{equation}\label{equ10}
  l_{Co}^{(1,\infty)} (\Phi) \ne  l_{Co}^{(1,\infty)} (\Phi^{-1}),
  \end{equation}
  or,
  \begin{equation}\label{equ11}
  l_{Co}^{\infty} (\Phi) \ne  l_{Co}^{\infty} (\Phi^{-1}). 
  \end{equation}
  The restriction of the above lengths to the group $ H_{\eta, \omega}(M)$ will be called co-Hofer lengths, and denoted $ l_{CH}^{\infty}$, and 
  $ l_{CH}^{(1,\infty)}$. Indeed, if  $\Phi_F = \{\phi_t\}$ is a co-Hamiltonian isotopy such that $\widetilde I_{\eta, \omega} (\dot\phi_t) = dF_t$, for all $t$, then 
  \begin{equation}\label{equ13}
  l_{CH}^{(1,\infty)} (\Phi_F) = \int_0^1 \left( osc(F_t) + | C_{\Phi_F,\eta}^t|\right) dt,
  \end{equation}
  and 
  \begin{equation}\label{equ14}
  l_{CH}^{\infty} (\Phi_F) = \max_t \left( osc(F_t) + | C_{\Phi_F,\eta}^t|\right). 
  \end{equation}
  Notice that the lengths $ l_{CH}^{\infty}$, and 
  $ l_{CH}^{(1,\infty)}$ are symmetric.

  \subsubsection{Displacement energy of fibers}
  Assume that $\Phi = \{\phi_t\}$ and $\tilde\Phi = \{\tilde\phi_t\}$ are as defined in the Subsection \ref{Trans} with $\phi_1\neq id_M$ ;  let $ l_{HL}$   denote the Hofer-like length and by $E_S$, we denote the symplectic displacement energy defined on the closed symplectic manifold $(\widetilde M, \widetilde\omega)$ \cite{Banyaga08, Tchuiaga2}. 
  Since $\phi_1\neq id_M$, then $ \tilde\phi_1 \neq id_{\tilde M}$, i.e., there exists a compact subset $\mathbf{B}_0\subset \widetilde M$ such that 
  $ \tilde\phi_1(\mathbf{B}_0)\cap\mathbf{B}_0 = \emptyset $. 
  We may assume that  $\mathbf{B}_0$ is  
  of the form $\mathbf{B}\times \mathbf{C}$, with $\mathbf{B}$ a compact subset of $M$ and $\mathbf{C}$  a compact subset of $\mathbb{S}^1$. Thus, for 
  each fixed $\theta\in \mathbf{C}$, the compact fiber $\mathbf{B}\times \{\theta\}$ is also completely displaced by $ \tilde\phi_1$. Therefore, 
  we have 
  \begin{equation}\label{Ene-trans1}
  0<  E_S(\mathbf{B}\times \{\theta\})\leqslant  l_{HL}(\tilde\Phi), \qquad \forall \, \theta\ \in \mathbf{C}.
  \end{equation}
  Since the map $\theta \longmapsto E_S(\mathbf{B}\times \{\theta\})$ is bounded and positive on $\mathbf{C}$, we derive that
  \begin{equation}\label{Ene-trans2}
  0<  \frac{1}{2\pi}\int_{\mathbf{C}} E_S(\mathbf{B}\times \{\theta\})d\theta\leqslant  l_{HL}(\tilde\Phi).
  \end{equation}
  On the other hand, from (\ref{Hofer-trans}), if $ \widetilde I_{\eta, \omega}(\dot{\Phi}_t) = \mathcal{H}^t_\omega  + \mathcal{K}_\eta^t + d(U_\eta^t + U_\omega^t) $, for all $t$, then we derive that
  \begin{eqnarray}\label{Hofer-trans1}
  \imath_{\dot{\tilde \Phi}_t}\tilde{\omega} &=& p^\ast( \imath_{\dot {\phi}_t }\omega)  +   p^\ast( C_{\Phi,\eta}^t) d\theta + p^\ast(C_{\Phi,\eta}^t\eta)\nonumber \\
  &=&  p^\ast(\mathcal{H}^t_\omega  + \mathcal{K}_\eta^t) + d(C_{\Phi,\eta}^t{\color{blue}\pi_2} + U_\eta^t + U^t_\omega ), \ \text{ for each } \ t.
  \end{eqnarray}
  Thus, 
  \begin{eqnarray}\label{Hofer-trans4}
  l_{HL}^\infty(\tilde{\Phi}) &=& \max_t \left( \lVert \pi_{\widetilde{M}} ( p^\ast(\mathcal{H}^t_\omega  + \mathcal{K}_\eta^t))\rVert_{L^2} + 
  osc ( C_{\Phi,\eta}^t\pi_2 + U_\eta^t + U^t_\omega)\right)\nonumber\\
  &\leqslant &  \max_t \left( \lVert \pi_{\widetilde{M}} ( p^\ast(\mathcal{H}^t_\omega  + \mathcal{K}_\eta^t))\rVert_{L^2} + 
  osc ( U_\eta^t + U^t_\omega) + \pi| C_{\Phi,\eta}^t|\right).
  \end{eqnarray}
  By (\ref{sharp 2}), we have $$ \| \pi_{\widetilde{M}}( p^\ast(\mathcal{H}^t_\omega  + \mathcal{K}_\eta^t))\|_{L^2}\leqslant \kappa_0 |\mathcal{H}^t_\omega  + \mathcal{K}_\eta^t |_0,$$ whereas by  (\ref{sharp 1}), we have
  $$|\mathcal{H}^t_\omega  + \mathcal{K}_\eta^t |_0\leqslant k_2(g)  \|\mathcal{H}^t_\omega  + \mathcal{K}_\eta^t\|_{L^2},$$ so, it follows that 
  \begin{equation}\label{Hofer-trans 5}
  l_{HL}^\infty(\tilde \Phi) \leqslant 2\max\{(1 + \kappa_0k_2(g)), \pi\} l_{Co}^\infty( F_j).
  \end{equation}
  Thus, (\ref{Ene-trans1}) and (\ref{Hofer-trans 5}) imply that 
  \begin{equation}\label{Ene-trans4}
  0<  \frac{1}{4\pi\max\{(1 + \kappa_0k_2(g)), \pi\}}\int_{\mathbf{C}} E_S(\mathbf{B}\times \{\theta\})d\theta\leqslant  l_{Co}^\infty( \Phi),
  \end{equation}
  In the rest of this paper, we shall refer to (\ref*{Ene-trans4}) as the Co-energy-inequality. \\
  
  \subsubsection{Gromov area of fibers}
  Assume that $\Phi_F = \{\phi_t\}$ is a co-Hamiltonian isotopy such that $\widetilde I_{\eta, \omega} (\dot\phi_t) = dF_t$, for all $t$, and 
  $\tilde\Phi_F = \{\tilde\phi_t\}$ is defined via $\Phi_F$ as in the Subsection \ref{Trans} with $\phi_1\neq id_M$, let
  $C_W(\overline{B})$ represents the Gromov area of a ball $\overline{B}$ on the closed symplectic manifold $(\tilde M, \tilde\omega)$ \cite{Lal-McD95}. 
  Since $\phi_1\neq id_M$, then $ \tilde\phi_1 \neq id_{\tilde M}$, i.e., there exists a compact subset $\mathbf{B}_0\subset \tilde M$ such that 
  $ \tilde\phi_1(\mathbf{B}_0)\cap\mathbf{B}_0 = \emptyset $. 
  We may assume that  $\mathbf{B}_0$ is  
  of the form $\mathbf{B}\times \mathbf{C}$, with $\mathbf{B}$ a compact subset of $M$, and $\mathbf{C}$ subset a compact of $\mathbb{S}^1$. Thus, for 
  each fixed $\theta\in \mathbf{C}$, the compact fiber $\mathbf{B}\times \{\theta\}$ is also completely displaced by $ \tilde\phi_1$. Therefore, 
  
  \begin{equation}\label{Ene-trans7-2}
  0<  \frac{1}{(2 \pi)^2}\int_{\mathbf{C}} C_W(\mathbf{B}\times \{\theta\})d\theta\leqslant  l_{CH}^\infty( \Phi_F).
  \end{equation}
  In the rest of this paper, we shall refer to (\ref*{Ene-trans7-2}) as the co-capacity-inequality. \\
  
  Here is the cosymplectic analogues of Theorem $6-$\cite{Hof-Zen94}. 
  \begin{theorem}\label{maint1} 
  	Let $(M,\eta, \omega)$ be a closed  
  	cosymplectic manifold.
  	Let $ \Phi = \{\phi_i^t\}$ be a sequence of cosymplectic isotopies, 
  	$\Psi = \{\psi^t\}$ be another cosymplectic isotopy, and  $\phi : M\rightarrow M$ be a map such that
  	\begin{itemize}
  		\item  $(\phi_i^1)$ converges 
  		uniformly to $\phi$, and 
  		\item $l_{C}^{\infty}(\Psi^{-1}\circ\{\phi_i^t\})\rightarrow0,i\rightarrow\infty$.
  	\end{itemize}
  	Then we must have $\phi = \psi^1.$
  \end{theorem} 
  \begin{proof} Let assume that $\phi \neq \psi^1,$ i.e., there exists a  compact subset $\mathbf{B}_0\subseteq M$ which is completely displaced by 
  	$ (\psi^1)^{-1}\circ \phi$, and since the convergence $ \phi_i^1\longrightarrow \phi$, is uniformly, then 
  	we may assume that $ (\psi^1)^{-1}\circ \phi_i^1$, completely displace $\mathbf{B}_0$, for all $i$ sufficiently large. 
  	Fix $i_0$ to be a sufficiently large natural number. We have 
  	a sequence of cosymplectic isotopies $\{ \Psi^{-1}\circ\{\phi_j^t\}\}_{j\geqslant i_0}$ with time-one map  $ (\psi^1)^{-1}\circ \phi_j^1$, for all 
  	$j\geqslant i_0$. Then, we derive from the Co-energy-inequality that 
  	\begin{equation}\label{Ene-trans7}
  	0<  \frac{1}{4\pi\max\{(1 + \kappa_0k_2(g)), \pi\}}\int_{\mathbf{C}_0} E_S(\mathbf{B}_0\times \{.\})d\theta\leqslant  l_{Co}^\infty( \Psi^{-1}\circ\{\phi_j^t\}),
  	\end{equation}
  	for some non-trivial compact subset $\mathbf{C}_0$ of the unit circle, and 
  	for all $j\geqslant i_0$. Since the right-hand side in (\ref{Ene-trans4}) tends to zero as $j$ tend to infinity, then 
  	(\ref{Ene-trans4}) yields a contradiction. 
  \end{proof} 
  
  The following result is an immediate consequence of Theorem \ref{maint1}. It can justify the existence of a $C^0-$counterpart of 
  cosymplectic geometry (see \cite{BanTchu}). 
  \begin{corollary}\label{Comaint}
  	Let $\Phi_i = \{\phi_i^t\}$ be a sequence of symplectic isotopies,
  	$\Psi = \{\psi_t\}$ be another symplectic isotopy, and  $\varXi:t\mapsto \varXi_t$ be a 
  	family of maps $\varXi_t : M\rightarrow M$, 
  	such that the sequence $\Phi_i$ converges  uniformly 
  	to $\varXi$ and  $l_{C}^{\infty}(\Psi^{-1}\circ\Phi_i)\rightarrow0,i\rightarrow\infty$. 
  	Then $\varXi = \Psi.$
  \end{corollary}
  \begin{proof} Assume the contrary, i.e., that $\Psi \neq \varXi$. This is equivalent to say that  
  	there exists $t\in ]0,1]$ such that $\varXi_t\neq\psi_t.$ Therefore, the sequence 
  	of symplectic paths $\Phi_{t,i}:s\mapsto \phi_i^{st}$ contradicts Theorem \ref{maint1}. 
  \end{proof}
  \subsubsection{Co-Hofer norm}\label{SC01-12}
  For any co-Hamiltonian diffeomorphism $\psi$, we define its
  the  $L^{(1,\infty)}-$co-Hofer norm and its $L^{\infty}-$co-Hofer norm respectively as follow:
  \begin{equation}\label{bener10}
  \Arrowvert \psi \Arrowvert _{CH}^{(1,\infty)} := \inf(l_{CH}^{(1,\infty)}(\Psi))
  \quad \text{ and }\quad 
  \Arrowvert \psi \Arrowvert_{CH}^{\infty} := \inf(l_{CH}^{\infty}(\Psi)),
  \end{equation}
  where each infimum is taken over the set of all co-Hamiltonian isotopies 
  $\Psi$ with time-one maps equal to $\psi$.
  \begin{theorem}\label{thm33}
  	Let $(M, \eta, \omega)$ be a compact cosymplectic manifold. Then, each of the rules $\|\cdot\|_{CH}^{(1,\infty)}$ and $\|\cdot\|_{CH}^\infty $ induces 
  	a bi-invariant norm on $Ham_{\eta, \omega}(M)$. 
  \end{theorem}
  \subsubsection{Co-Hofer-like energies}\label{SC012}
  For $\phi\in G_{\eta, \omega}(M)$, we define its $L^{(1,\infty)}-$energy and its $L^{\infty}-$energy as follows:
  \begin{equation}\label{bener1}
  e_C^{(1,\infty)}(\phi) := \inf(l_{C}^{(1,\infty)}(\Phi)),
  \end{equation}
  and
  \begin{equation}\label{bener2}
  e^{\infty}_C(\phi) := \inf(l_{C}^{\infty}(\Phi)),
  \end{equation}
  where each infimum is taken over the set of all cosymplectic isotopies 
  $\Phi$ with time-one maps equal to $\phi$.
  
  \subsubsection{Co-Hofer-like norms}
  The $L^{(1,\infty)}-$version and the $L^{\infty}-$version 
  of the co-Hofer-like norms of $\phi\in G_{\eta, \omega}(M)$  are respectively defined by,
  \begin{equation}\label{bny1}
  \|\phi\|_{C}^{(1,\infty)} 
  := \frac{1}{2}(e_{C}^{(1,\infty)}(\phi) + e_{C}^{(1,\infty)}(\phi^{-1})),
  \end{equation}
  and
  \begin{equation}\label{bny2}
  \|\phi\|_{C}^\infty 
  := \frac{1}{2}(e^{\infty}_{C}(\phi) + e^{\infty}_{C}(\phi^{-1})).
  \end{equation}
  
  \begin{theorem}\label{thm3}
  	Let $(M, \eta, \omega)$ be a compact cosymplectic manifold. Then, each of the rules $\|\cdot\|_{C}^{(1,\infty)}$ and $\|\cdot\|_{C}^\infty $ induces 
  	a right-invariant norm on $G_{\eta, \omega}(M)$. 
  \end{theorem}
  \begin{proof} Since checking the other properties of a norm are straight calculations, we shall just prove the non-degeneracy of the norm $ \|\cdot \|_{C}^\infty$ \ :  if $\phi\in G_{\eta, \omega}(M)$ such that  $\|\phi\|_{C}^\infty = 0,$ then  
  	from the definition of the norm $ \|\cdot\|_{C}^\infty$, we derive that there exists a sequence   
  	of cosymplectic isotopies $\{\Phi_{i}\}$, each of which with time-one map $\phi$ such that $
  	l_{C}^{\infty}(\Phi_{i})< \dfrac{1}{i} ,$ for each  positive integer $i$. That is, 
  	\begin{itemize}
  		\item $\lim_{C^0}(\Phi_{i}(1)) = \phi$ and
  		\item $l_{C}^{\infty}(\{Id\}^{-1}\circ\Phi_{i})\longrightarrow 0$,  as $i\longrightarrow\infty$,
  	\end{itemize}
  	where $Id$ is the constant path identity. 
  	Hence, by Theorem \ref{maint1}, we must have $\phi = id_M$. 
  \end{proof}
  Banyaga-Bikorimana \cite{Ban-Bi} have studied the cosymplectic analogue of the $C^0-$rigidity result of Eliashberg-Gromov  \cite{Elias87}
  using Lemma \ref{lem-3} together with a result found by Buhosky \cite{L-B}. Here is a direct proof of cosymplectic analogue of the $C^0-$rigidity result of Eliashberg-Gromov  \cite{Elias87}: This follows as a direct consequence of Theorem \ref{maint1}, and here we do not appeal to Buhosky's result \cite{L-B}. 
  \begin{theorem}\label{closure}
  	The group $G_{\eta, \omega}(M)$ is $C^0-$closed inside the group $Diff^\infty(M)$. 
  \end{theorem}
  We shall need the following rigidity lemma.
 \begin{lemma}\label{Rigid-lem}
 	Let $(M, \omega, \eta)$ be a compact connected cosymplectic manifold, and let $X_H$ be a co-Hamiltonian vector field 
 	such that $\widetilde  I_{\eta, \omega}(X_H) = dH$. Let $\{\psi_i\}\subset  G_{\eta, \omega}(M)$ such that $ \psi_i\xrightarrow{C^0}\psi$. If $\psi\in Diff^\infty(M)$, then
 	\begin{enumerate}
 		\item  $\psi_\ast(\xi)(H) = \xi(H)$, and 
 		\item $ \eta(\psi_\ast(\xi)) = 1$, 
 	\end{enumerate}	
 	
 	where $\xi$ is the Reeb vector field of $(M, \omega, \eta)$. 
 \end{lemma}
 \begin{proof} For $(2)$,  assume $M$ to be equipped with a Riemannian metric $g$, with injectivity radius $r(g)$. Pick any smooth curve $\gamma \in M$. 	Since $\psi_i\xrightarrow{C^0}\psi$, then for $i$ sufficiently large we may assume that 
 	$d_{C^0} (\psi_i,\psi)\leq r(g)/2$, and derive that, for each $t\in [0,1]$, the points $\psi_i(\gamma(t))$ to  $\psi(\gamma(t))$ can be connected 
 	through a minimizing geodesic  $\chi_i^t$. This implies that 
 	the curves $\chi_i^0$, $\chi_i^1$, $\psi_i\circ \gamma$, and  $\psi\circ \gamma$ form the boundary of a smooth  $2-$chain $\square(\gamma, \psi_i, \psi)\subset M$. Since $d\eta = 0$, we derive from Stockes' theorem that $\int_{ \square(\gamma, \psi_i, \psi)}d\eta = 0,$ i.e., 
 	$\int_{\psi_i\circ \gamma}\eta - \int_{\psi_\circ \gamma}\eta  = \int_{\chi_i^1} \eta -  \int_{\chi_i^0}\eta,$
 	for $i$ sufficiently large. That is,
 	$	| \int_{\psi_i\circ \gamma}\eta - \int_{\psi_\circ \gamma}\eta|\leq 2 |\eta|_{0}d_{C^0} (\psi_i,\psi),$
 	for $i$ sufficiently large because the length of any minimizing geodesic is bounded from above by the distance between its endpoints, i.e., 
 	$\lim_{i\longrightarrow \infty}\left( \int_{\psi_i\circ \gamma}\eta\right)  =  \int_{\psi_\circ \gamma}\eta $. On the other hand, 
 	 let $\{\phi_t\}$ be the cosymplectic flow generated by the Reeb vector field $\xi$. For each fixed $t\in ]0, 1[$, and each $x\in M$, consider the smooth curve $\bar \gamma_{x, t}: s \mapsto \phi_{st}(x)$, and derive from previous limit that\\ 
 	 $\lim_{i\longrightarrow \infty} \int_{\psi_i(\bar\gamma_{x, t})} \eta  = \int_{\psi(\bar \gamma_{x, t})} \eta,$ for each fixed $t\in ]0, 1[$, i.e.,\\ 
 	 $  t = \lim_{i\longrightarrow \infty} \int_0^t\left( du\right) =  \int_0^t(\psi^\ast(\eta)(\xi))(\phi_{u}(x))du,$ for each fixed $t\in ]0, 1[$, since by 
 	 Lemma \ref{pushfoward} we have $(\psi_i)_\ast(\xi) = \xi$, for all $i$. Therefore, taking the derivative of the equality with respect to $t$ gives: 
 	 $ 1  =  (\psi^\ast(\eta)(\xi))(\phi_{t}(x)), $ for each fixed $t\in ]0, 1[$, for all $x\in M$, which implies 
 	 $\eta(\psi_\ast(\xi)) = 1$. For $(1)$, since $\psi_i\xrightarrow{C^0}\psi$, and $H$ is continuous, we derive that\\ 
 	 $ \lim_{i\longrightarrow \infty}\left( H(\psi_i(\bar\gamma_{x, t}(1))) - H(\psi_i(x))\right) =   \left( H(\psi(\bar \gamma_{x, t}(1))) - H(\psi(x))\right),$ 
 	 for each fixed $t\in ]0, 1[$, for all $x\in M$. We also have\\ $ \lim_{i\longrightarrow \infty}\left( H(\psi_i(\bar\gamma_{x, t}(1))) - H(\psi_i(x))\right) = \lim_{i\longrightarrow \infty}\left(\int_{\psi_i\circ\bar\gamma_{x, t} }dH  )\right) = \int_0^t dH((\psi_i)_\ast(\xi))\circ\bar\gamma_{x, t}(s)ds $, i,e., 
 	  $$ \lim_{i\longrightarrow \infty}\left( H(\psi_i(\bar\gamma_{x, t}(1))) - H(\psi_i(x))\right)
 	   = \int_0^t dH((\psi_i)_\ast(\xi))(\psi_i(\phi_s(x))ds = \int_0^t dH(\xi)((\psi_i(\phi_s(x)))ds,$$
 	   which implies that $ \lim_{i\longrightarrow \infty}\left( H(\psi_i(\bar\gamma_{x, t}(1))) - H(\psi_i(x))\right) =  t\xi(H),$ for each fixed $t\in ]0, 1[$, for all $x\in M$ because the function $y\mapsto \xi(H)(y)$ is constant. Thus, we have just proved that
 	    $$  t\xi(H) = \lim_{i\longrightarrow \infty}\left( H(\psi_i(\bar\gamma_{x, t}(1))) - H(\psi_i(x))\right) = \left( H(\psi(\bar \gamma_{x, t}(1))) - H(\psi(x))\right) = \int_{\psi\circ\bar\gamma_{x, t} }dH,$$  for each fixed $t\in ]0, 1[$, for all $x\in M$, i.e., 
 	    $t\xi(H) = \int_0^t\left( \psi_\ast(\xi)(H)\right)(\psi(\phi_s(x))ds$, for each fixed $t\in ]0, 1[$, for all $x\in M$. 
 	     Thus, taking the derivative in the latter equality with respect to $t$, gives 
 	   $ \xi(H) =  \psi_\ast(\xi)(H)(\psi(\phi_t(x)),$ for each fixed $t\in ]0, 1[$, for all $x\in M$.
 \end{proof}
 
 \begin{remark}\label{Rigid-2}
 		Let $(M, \omega, \eta)$ be a compact connected cosymplectic manifold, and let $X_H$ be a co-Hamiltonian vector field 
 such that $\widetilde  I_{\eta, \omega}(X_H) = dH$.  For each smooth diffeomorphism $\psi$ of $M$, as in Remark \ref{flux-geo-rem}, define a vector field  $X_ {H\circ\psi}: = Y_{d(H\circ \psi)} - 
 \xi(H\circ \psi)\xi$, which satisfies 
 $  \widetilde I_{\eta, \omega}(X_{H\circ\psi}) =  d(H\circ \psi)$,   $\mathcal{L}_{X_{H\circ\psi}}\eta = 0$, and $\mathcal{L}_{X_{H\circ\psi}}\omega =  d\left( \xi(H\circ \psi) \right)$. If in addition, we have the information  that there exists a sequence $\{\psi_i\}\subset  G_{\eta, \omega}(M)$ such that $,  \psi_i\xrightarrow{C^0}\psi$, then with the help of Lemma \ref{Rigid-lem}, we can derive that $\psi_\ast(\xi)(H) = \xi(H) $. This combined together with\\ $\psi^\ast(d H)(\xi) = \left( dH(\psi_\ast(\xi))\right)\circ\psi$, implies that $\psi^\ast(d H)(\xi) = \xi(H)\circ \psi$. Since the smooth function 
 $\xi(H)$ is constant  by assumption, then we have $ \xi(H) = d(H\circ \psi)(\xi) = \xi(H\circ \psi)$. Thus, 
  $\mathcal{L}_{X_{H\circ\psi}}\omega =  d\left( \xi(H\circ \psi) \right) = d(\xi(H)) = 0$. Therefore, $X_{H\circ\psi} $ is a  co-Hamiltonian vector field 
  such that  $  \widetilde I_{\eta, \omega}(X_{H\circ\psi}) =  d(H\circ \psi)$, whenever $\psi$ is the $C^0-$limit of a sequence of cosymplectic diffeomorphisms.\\
 
 \end{remark}
  {\it Proof of Theorem \ref{closure}.} We shall adapt the proof given by Buhosky \cite{L-B} for similar result in symplectic geometry. Assume that 
  $M$ is equipped with a Riemannian metric $g$  with injectivity radius $r(g)$. Let $\{\varphi_i\}\subseteq G_{\eta, \omega}(M)$  be a sequence of cosymplectic diffeomorphisms such that $\varphi_i\xrightarrow{C^0}\psi\in Diff^\infty(M)$. Assume that $\psi$ is not a cosymplectic diffeomorphism. Then, for any 
  co-Hamiltonian vector field $X_H$ such that $\widetilde I_{\eta, \omega}(X_H) = dH$, we have\\ $ \varphi_\ast(X_H)\neq X_{H\circ\varphi^{-1}}$ (this is supported by Remark \ref{Rigid-2}). This implies that, if $\Psi_H$ 
  is  
  the cosymplectic flow generated by  $X_H$, then we must have 
  \begin{equation}\label{contradiction}
  \varphi \circ \Psi_H\circ \varphi^{-1} \neq \Psi_{H\circ \varphi^{-1}},
  \end{equation}
  where $ \Psi_{H\circ \varphi^{-1}}$ is the cosymplectic flow generated by $ Y_{H\circ \varphi^{-1}}: = \widetilde I_{\eta, \omega}^{-1}( d\left( H\circ \varphi^{-1}\right) ).$ The sequence 
  of co-Hamiltonian isotopies $ \varphi_i \circ \Psi_H\circ \varphi^{-1}_i$  converges uniformly to  $ \varphi \circ \Psi_H\circ \varphi^{-1}$, and we have
  $$ l_{C}^{\infty}(\Psi_{H\circ \varphi^{-1}}^{-1}\circ\{\varphi_i \circ \Psi_H\circ \varphi^{-1}_i\}) = osc(H\circ \varphi^{-1}_i - H\circ \varphi^{-1})  
  + |\eta(X_{H\circ \varphi^{-1}}) - \eta(X_{H\circ \varphi^{-1}_i})|,$$ 
  for each $i$. On the other hand, for each fixed $x\in M$, consider  the orbits $ \mathcal{C}_{x, i}: =\left( \varphi_i \circ \Psi_H\circ \varphi^{-1}_i\right)(x)$, and  $\mathcal{C}_{x}: = \left( \varphi \circ \Psi_H\circ \varphi^{-1}\right) (x)$.
   Since $ \varphi_i \circ \Psi_H\circ \varphi^{-1}_i$  converges uniformly to  $ \varphi \circ \Psi_H\circ \varphi^{-1}$, 
   then for $i$ sufficiently large we may assume that 
   $\bar d (\varphi_i \circ \Psi_H\circ \varphi^{-1}_i, \varphi \circ \Psi_H\circ \varphi^{-1})\leq r(g)/2$, and derive as in the proof of Lemma \ref{Rigid-lem},  that there exist two minimal geodesics $\gamma_i$ (with endpoints $x$ and $\left( \varphi_i \circ \Psi_H^1\circ \varphi^{-1}_i\right)(x)$) and $\gamma$ (with endpoints $x$ and $\left( \varphi \circ \Psi_H^1\circ \varphi^{-1}\right)(x)$) such that 
  $ \mathcal{C}_{x, i},  \mathcal{C}_{x}, \gamma_i$, and $\gamma $ delimit a $2-$chain in $\boxplus_{ \mathcal{C}_{x, i},  \mathcal{C}_{x}, \gamma_i, \gamma}\subset M$. Since $d\eta = 0$, it follows from Stokes' theorem that 
  $ \int_{\boxplus_{ \mathcal{C}_{x, i},  \mathcal{C}_{x}, \gamma_i, \gamma}}d\eta = 0$, i.e., 
  $\int_{\mathcal{C}_{x, i}} \eta - \int_{\mathcal{C}_{x}} \eta =  \int_{\gamma_i} \eta - \int_{\gamma} \eta,$
  for all $i$ sufficiently large. That is, for each $x\in M$, we have
  $$ |\eta(X_{H\circ \varphi^{-1}})(x) - \eta(X_{H\circ \varphi^{-1}_i})(x)| = | \int_{\mathcal{C}_{x, i}} \eta - \int_{\mathcal{C}_{x}} \eta|$$ 
  $$ = | \int_{\gamma_i} \eta - \int_{\gamma} \eta| $$ 
  $$\leq 2|\eta|_0\bar d(\varphi_i \circ \Psi_H\circ \varphi^{-1}_i,\varphi \circ \Psi_H\circ \varphi^{-1}),$$
   for all $i$ sufficiently large. Hence  
    $$ l_{C}^{\infty}(\Psi_{H\circ \varphi^{-1}}^{-1}\circ\{\varphi_i \circ \Psi_H\circ \varphi^{-1}_i\}) = osc(H\circ \varphi^{-1}_i - H\circ \varphi^{-1}) 
 + |\eta(X_{H\circ \varphi^{-1}}) - \eta(X_{H\circ \varphi^{-1}_i})|\longrightarrow 0, i\longrightarrow\infty.$$ Similarly, one 
 proves that: if we consider the reparametrized isotopies $ \Psi_{t, H}: s\mapsto \Psi_H^{ts}$, and $ \Psi_{t, H\circ \varphi^{-1}}: s\mapsto \Psi_{H\circ \varphi^{-1}}^{st}$, for each fixed 
 $t$, then  $l_{C}^{\infty}(\Psi_{t, H\circ \varphi^{-1}}^{-1}\circ\{\varphi_i \circ\Psi_{t, H}\circ \varphi^{-1}_i\})
  \rightarrow0,i\rightarrow\infty,$  for each fixed 
  $t$. 
 Finally, we have proved that for each fixed 
 $t$, we have 
 \begin{itemize}
 	\item $\varphi_i \circ \Psi_{t, H}\circ \varphi^{-1}_i\xrightarrow{C^0} \varphi \circ \Psi_{t, H}\circ \varphi^{-1}$, and 
 	\item $  l_{C}^{\infty}(\Psi_{t, H\circ \varphi^{-1}}^{-1}\circ\{\varphi_i \circ \Psi_H^t\circ \varphi^{-1}_i\}) \rightarrow0,i\rightarrow\infty$.
 \end{itemize}
  Thus, by Theorem \ref{maint1}
  we must have, $ \varphi \circ \Psi_H^t\circ \varphi^{-1} = \Psi_{H\circ \varphi^{-1}}^t$, for all $t$: This contradicts (\ref{contradiction}). $\square$\\
  \subsection{Almost co-Hofer-like geometries}
  For any $X\in  \mathcal{A}\chi_{\eta, \omega}(M)$, the closed $1-$form $\iota(X)\omega$ splits as:
  \begin{equation}\label{equ5-A}
  \iota(X)\omega = \mathcal{H}_\omega + dU_\omega.
  \end{equation}
  From the above splitting, one defines a norm $\|.\|_{\mathcal{A}C}^\mathcal{S} $ on 
  $\chi_{\eta, \omega}(M)$ as follows:\\ For any  $X\in \chi_{\eta, \omega}(M)$, 
  \begin{equation}\label{equ7-A}
  \|X\|_{\mathcal{A}C}^\mathcal{S} := \|\mathcal{H}_\omega\|_{L^2} + osc(U_\omega) + \Theta(X) ,
  \end{equation}
  with
  $$ \Theta(X): = \frac{1}{Vol_{\eta, \omega}(M)}\int_M|\eta(X)|\eta\wedge\omega^n,$$
  where $ \|.\|_{L^2}$ is the $L^2-$Hodge norm, $\dim(M) = 
  (2n +1)$, and $Vol_{\eta, \omega}(M): = \int_M\eta\wedge\omega^n $. One can derive from Theorem \ref{thm2}
  that if $ \mathcal{S}$ and $\mathcal{T}$ are two linear sections of the projection 
  $\mathcal{\pi}: \mathcal{Z}^1(M)\rightarrow  H^1(M,\mathbb{R}) $, then the two norms $\|.\|_{\mathcal{A}C}^\mathcal{S} $, and $\|.\|_{\mathcal{A}C}^\mathcal{T} $ are equivalent. 
  \subsubsection{Almost co-Hofer-like lengths} 
  Let $ \Phi = \{\phi_t\}\in  \mathcal{A}Iso_{\eta, \omega}(M)$, such that $\mathcal{L}_{\dot\phi_t}\eta = \mu_t\eta$, 
  for each $t$, we have  
  $$ \|\dot{\phi}_t\|_{\mathcal{A}C} := \|\mathcal{H}_\omega^t\|_{L^2} + osc(U_\omega^t) + \varTheta_t(\Phi),$$
  with 
  $$ \varTheta_t(\Phi): = \frac{1}{Vol_{\eta, \omega}(M)}\int_M |\eta(\dot{\phi}_t)|\eta\wedge\omega^n,$$
  for each $t$. Therefore, we define the $L^{(1, \infty)}-$version of the almost co-Hofer-like length of $\Phi: =\{\phi_t\}$ as:
  \begin{equation}\label{equ8-A}
  l_{\mathcal{A}co}^{(1,\infty)} (\Phi) : = \int_0^1\|\dot{\phi}_t\|_{\mathcal{A}C}dt,
  \end{equation}
  and, $L^{\infty}-$version of the almost co-Hofer-like length of $\Phi$ as:
  \begin{equation}\label{equ9-A}
  l_{\mathcal{A}co}^{\infty} (\Phi) : = \max_{t\in [0,1]}\|\dot{\phi}_t\|_{\mathcal{A}C}. 
  \end{equation}
  As in the case of co-Hofer-like length, it seems that in general, we have
  \begin{equation}\label{equ10-A}
  l_{\mathcal{A}co}^{(1,\infty)} (\Phi) \ne  l_{\mathcal{A}co}^{(1,\infty)} (\Phi^{-1}),
  \end{equation}
  or,
  \begin{equation}\label{equ11-A}
  l_{\mathcal{A}co}^{\infty} (\Phi) \ne  l_{\mathcal{A}co}^{\infty} (\Phi^{-1}). 
  \end{equation}
  \subsubsection{Almost co-Hamiltonian lengths}
  The restriction of the above lengths to the group $ \mathcal{A}H_{\eta, \omega}(M)$ will be called almost co-Hofer lengths, and denoted $ l_{\mathcal{A}H}^{\infty}$, and 
  $ l_{\mathcal{A}H}^{(1,\infty)}$. Indeed, if  $\Phi = \{\phi_t\}$ is an almost co-Hamiltonian isotopy such that 
  $ \iota(\dot\phi_t)\omega = dF_t$, for all $t$, then 
  \begin{equation}\label{equ13-A}
  l_{\mathcal{A}H}^{(1,\infty)} (\Phi) := \int_0^1 \left( osc(F_t) + \varTheta_t(\Phi)\right) dt,
  \end{equation}
  and,
  \begin{equation}\label{equ14-A}
  l_{\mathcal{A}H}^{\infty} (\Phi) := \max_t \left( osc(F_t) + \varTheta_t(\Phi)\right). 
  \end{equation}
  If  $\Phi = \{\phi_t\}$ is an almost co-Hamiltonian isotopy, then  it is  not hard to check that\\ $ l_{\mathcal{A}H}^{(1,\infty)} (\Phi^{-1}) = l_{\mathcal{A}H}^{(1,\infty)} (\Phi),$ and $ l_{\mathcal{A}H}^\infty (\Phi^{-1}) = l_{\mathcal{A}H}^\infty (\Phi),$ i.e.,  the lengths $ l_{\mathcal{A}H}^{\infty}$, and 
  $ l_{\mathcal{A}H}^{(1,\infty)}$ are symmetric.
  \subsubsection{Almost co-Hofer norms}\label{SC01-12A}
  For any almost co-Hamiltonian diffeomorphism $\psi$, we define 
  the  $L^{(1,\infty)}-$version of its almost co-Hofer norm and $L^{\infty}-$version of its almost co-Hofer norm respectively as follows:
  \begin{equation}\label{bener10A}
  \Arrowvert \psi \Arrowvert _{\mathcal{A}H}^{(1,\infty)} := \inf(l_{\mathcal{A}H}^{(1,\infty)}(\Psi)),
  \end{equation}
  and,
  \begin{equation}\label{bener20A}
  \Arrowvert \psi \Arrowvert_{\mathcal{A}H}^{\infty} := \inf(l_{\mathcal{A}H}^{\infty}(\Psi)),
  \end{equation}
  where each infimum is taken over the set of all almost co-Hamiltonian isotopies 
  $\Psi$ with time-one maps equal to $\psi$.
  \begin{theorem}\label{thm5A}
  	Let $(M, \eta, \omega)$ be a compact cosymplectic manifold. Then, each of the rules $\|.\|_{\mathcal{A}H}^{(1,\infty)}$, and $\|.\|_{\mathcal{A}H}^\infty $ induces 
  	a bi-invariant norm on $\mathcal{A}Ham_{\eta, \omega}(M)$. 
  \end{theorem}
  \subsubsection{Almost co-Hofer-like energies}\label{SC012A}
  For any $\phi\in  \mathcal{A}G_{\eta, \omega}(M)$, we define its  $L^{(1,\infty)}-$energy and $L^{\infty}-$energy respectively as follows:
  \begin{equation}\label{bener1A}
  e_{\mathcal{A}Co}^{(1,\infty)}(\phi) := \inf(l_{\mathcal{A}Co}^{(1,\infty)}(\Phi)),
  \end{equation}
  and,
  \begin{equation}\label{bener2A}
  e^{\infty}_{\mathcal{A}Co}(\phi) := \inf(l_{\mathcal{A}Co}^{\infty}(\Phi)),
  \end{equation}
  where each infimum is taken over the set of all almost  cosymplectic isotopies 
  $\Phi$ with time-one maps equal to $\phi$.
  
  \subsubsection{Almost co-Hofer-like norms}
  The $L^{(1,\infty)}-$version and the $L^{\infty}-$version 
  of the almost co-Hofer-like norms of $\phi\in  \mathcal{A}G_{\eta, \omega}(M)$  are respectively  defined as: 
  \begin{equation}\label{bny1A}
  \|\phi\|_{{\mathcal{A}Co}}^{(1,\infty)} 
  := (e_{\mathcal{A}Co}^{(1,\infty)}(\phi) + e_{\mathcal{A}Co}^{(1,\infty)}(\phi^{-1}))/2,
  \end{equation}
  and
  \begin{equation}\label{bny2A}
  \|\phi\|_{\mathcal{A}Co}^\infty 
  := (e^{\infty}_{\mathcal{A}Co}(\phi) + e^{\infty}_{\mathcal{A}Co}(\phi^{-1}))/2.
  \end{equation}
  
  \begin{theorem}\label{thm6A}
  	Let $(M, \eta, \omega)$ be a compact cosymplectic manifold. Then, each of the rules $\|.\|_{\mathcal{A}Co}^{(1,\infty)}$, and $\|.\|_{\mathcal{A}Co}^\infty $ induces 
  	a right-invariant norm on $ \mathcal{A}G_{\eta, \omega}(M)$. 
  \end{theorem}
  We shall first give a proof for Theorem \ref{thm5A}, and then derive that of Theorem \ref{thm6A}. First of all, note that the proofs of the 
  non-degeneracy of these norms are very technical: We shall prove this in several steps.\\

  {\it Proof of Theorem \ref{thm5A}.}  We shall proceed by contradiction. Assume that there is an almost co-Hamiltonian diffeomorphism $\psi$ with $\psi \neq id_M$ such 
  that $\|\psi\|_{\mathcal{A}H}^\infty = 0 $. 
  \begin{itemize}
  	\item {\bf Step $(1)$}: By definition of $ \|\psi\|_{\mathcal{A}H}^\infty$, there exists a 
  	sequence $\Psi_j : =\{\psi_{j, t}\}$ of almost co-Hamiltonian isotopies each of which  
  	with time-one map $\psi$ such that 
  	\begin{equation}\label{AC-1}
  	l_{\mathcal{A}H}^{\infty}(\Psi_j) < \frac{1}{j},
  	\end{equation}
  	for all positive integer $j$.  Assume that for each positive integer $j$, we have $\psi_{j, t}^\ast(\eta) = e^{f_t^j}\eta$,  and 
  	$ \iota(\dot\psi_{j, t}) \omega = dH_j^t$, for each $t$. It follows from (\ref{AC-1}) that  
  	\begin{equation}\label{AC-2}
  	\max_t \left( osc(H_j^t) + \varTheta_t(\Phi_j)\right) < \frac{1}{j},
  	\end{equation}
  	for all positive integer $j$.  That is, 
  	\begin{equation}\label{AC-3}
  	\max\left( \max_t osc(H_j^t) ,  \max_t\varTheta_t(\Phi_j)\right) < \frac{1}{j},
  	\end{equation}
  	for all positive integer $j$. On the other hand, consider the following sequence of symplectic 
  	isotopies $\tilde\Psi_j := \{\tilde\psi_t^j\}$ defined by 
  	$$ \tilde\psi_t^j: M\times \mathbb{S}^1 \longrightarrow M\times \mathbb{S}^1,$$
  	$$(x, \theta) \mapsto (\psi_{j, t}(x), \theta e^{-f_t^j(x)} ),$$
  	as in Proposition \ref{Trasit-1}. Derive from Proposition \ref{Trasit-4} that  $\tilde\Psi_j = \{\tilde\psi_t^j\}$ is a sequence of 
  	Hamiltonian isotopies of the symplectic manifold $(\tilde M, \tilde\omega)$ such that  
  	$$\iota (\dot{\tilde{\psi}}_{t}^j)\tilde{\omega} = p^\ast(d H_j^t) + d\left( \eta(\dot\psi_{j, t})\circ p\pi_2\right),$$
  	for each $j$, and for all $t$. Therefore, the $L^\infty-$ Hofer length of $\tilde\Psi_j = \{\tilde\psi_t^j\}$ satisfies 
  	\begin{equation}\label{AC-4}
  	l_H^\infty(\tilde\Psi_j) =  \max_t\left(  osc(H_j^t\circ p +  \eta(\dot\psi_{j, t})\circ p\pi_2)\right) ,
  	\end{equation}
  	$$ \leq \max_t osc(H_j^t) +  2\pi \max_t osc(\eta(\dot\psi_{j, t})),$$
  	for each $j$. 
  	\item {\bf Step $(2)$}:  Without the loss of generality, assume that $Vol_{\eta, \omega}(M) = 1$, and derive from (\ref{AC-3}) that 
  	$ \lim_{j\longrightarrow\infty} \int_M | \eta(\dot\psi_{j, t})|\eta\wedge\omega^n = 0,$ 
  	for each $t$. If there exist $s\in [0,1]$, and a nonempty open subset $U_\ast\subset M$ such that 
  	$ \left( | \eta(\dot\psi_{l, s})|_{|_{U_\ast}}\right) > \delta$, for some positive real number 
  	$\delta$, and for some integer $l$ sufficiently large, then  we could  complete $U_\ast$ with opens subsets $W_\alpha\subset M$ to obtain an open cover of $M$: 
  	By the compactness of $M$, we could consider a partition of unity $\{\phi_\alpha\}$ subordinate to the open cover 
  	$ M = U_\ast \cup\left( \cup_\alpha W_\alpha\right) $ with $\phi_{\alpha_\ast}$ supported in $U_\ast$. Therefore, 
  	we could compute 
  	$$ \int_M | \eta(\dot\psi_{l, s})|\eta\wedge\omega^n \geq \int_{\bar U_\ast}\phi_{\alpha_\ast} | \eta(\dot\psi_{l, s})|\eta\wedge\omega^n $$ 
  	$$\geq \int_{supp(\phi_{\alpha_\ast})} | \eta(\dot\psi_{l, s})|\eta\wedge\omega^n$$
  	$$ > \delta\int_{supp(\phi_{\alpha_\ast})} \eta\wedge\omega^n, $$
  	for some $l$ sufficiently large.  
  	That is, 
  	$ \frac{1}{l}
  	\geq \delta \left( \int_{supp(\phi_{\alpha_\ast})} \eta\wedge\omega^n\right) > 0 ,$
  	for $l$ sufficiently large,  which is a contradiction.
  	Hence, it follows from the above arguments that the condition $ \lim_{j\longrightarrow\infty} \int_M | \eta(\dot\psi_{j, t})|\eta\wedge\omega^n = 0,$ suggests  that for each $t$, we have 
  	$ \lim_{j\longrightarrow\infty}  | \eta(\dot\psi_{j, t})| = 0,$ uniformly. Furthermore, if 
  	$ \widetilde{| \eta(\dot\psi_{j, t})|}$ is the normalized function of  $ | \eta(\dot\psi_{j, t})|$, then we may assume that for each $t$, 
  	\begin{equation}\label{AC-44}
  	\lim_{j\longrightarrow\infty} \widetilde{| \eta(\dot\psi_{j, t})|} = 0,
  	\end{equation}
  	uniformly.  
  	\item {\bf Step $(3)$}: Since by assumption $\psi\neq id_M$, then there exists a nonempty open ball $B\subset M $ such that $\psi(\bar B)\cap \bar B = \emptyset$.
  	Therefore, the closed subset $\mathbb{B}:= \bar B\times \{0\} \subset M\times \mathbb{S}^1$ is completely displaced by 
  	$\tilde\psi_1^j:=(\psi\circ p, \pi_2e^{-f_t^j\circ p} )$ because 
  	$ \tilde\psi_1^j (\mathbb{B}) = \psi(\bar B)\times \{0\},$ 
  	implies  $ \tilde\psi_1^j (\mathbb{B})\cap \mathbb{B} = \emptyset,$ 
  	for each $j$. Hence, we derive from the energy-capacity-inequality theorem that 
  	\begin{equation}\label{AC-5}
  	0< \frac{C_W(\mathbb{B})}{2} \leq l_H^\infty(\tilde\Psi_j),
  	\end{equation}
  	for each $j$, where 
  	$C_W(\mathbb{B})$ represents the Gromov area of a ball $\mathbb{B}$ with 
  	respect to the symplectic manifold $(\tilde M, \tilde\omega)$ \cite{Lal-McD95}. From {\bf Step $(1)$} and  {\bf Step $(2)$}, it follows that 
  	\begin{eqnarray}
  	0&<& \frac{C_W(\mathbb{B})}{2} \nonumber\\
  	& \leq &  l_H^\infty(\tilde\Psi_j)
  	\nonumber\\
  	& \leq & \max_t osc(H_j^t) +  4\pi \sup_{x, t}
  	\widetilde{| \eta(\dot\psi_{j, t})(x)|},
  	\end{eqnarray}
  	for all $j$. That is, 
  	\begin{equation}\label{AC-7}
  	0< \frac{C_W(\mathbb{B})}{2} \leq  \frac{1}{j}+  4\pi \sup_{x, t} \widetilde{| \eta(\dot\psi_{j, t})(x)|},
  	\end{equation}
  	for each $j$. 
  \end{itemize} 
  Finally, (\ref{AC-44}) together with (\ref{AC-7}) implies that 
  \begin{equation}\label{AC-8}
  0< \frac{C_W(\mathbb{B})}{2} \leq \lim_{j\longrightarrow\infty}\left(  \frac{1}{j}+  4\pi \sup_{x, t} \widetilde{| \eta(\dot\psi_{j, t})(x)|}\right) = 0, 
  \end{equation}
  which is a contradiction. Thus, we must have  $\psi = id_M $.  $ \blacksquare$\\
  
  {\it Proof of Theorem \ref{thm6A}.} This proof follows from the fact that if $\psi\in \mathcal{A}G_{\eta, \omega}(M)$ 
  such 
  that $\|\psi\|_{\mathcal{A}H}^\infty = 0 $, then  any 
  sequence $\Psi_j : =\{\psi_{j, t}\}$ of almost symplectic isotopies   each of which  
  with time-one map $\psi$ such that $ l_{\mathcal{A}Co}^{\infty}(\Psi_j) < \frac{1}{j},$ 
  for all positive integer $j$, can be assumed to be an almost co-Hamiltonian isotopy for each $j$ sufficiently large. Therefore, the conclusion 
  follows as in the proof of Theorem \ref{thm5A}.   $ \blacksquare$
  
  \begin{theorem}\label{maint1-A} 
  	Let $(M,\eta, \omega)$ be a closed  
  	cosymplectic manifold.
  	Let $ \Phi = \{\phi_i^t\}$ be a sequence of almost cosymplectic isotopies, 
  	$\Psi = \{\psi_t\}$ be another almost cosymplectic isotopy, and  $\phi : M\rightarrow M$ be a map such that
  	\begin{itemize}
  		\item  $(\phi_i^1)$ converges 
  		uniformly to $\phi$, and 
  		\item $l_{\mathcal{A}Co}^{\infty}(\Psi^{-1}\circ\{\phi_i^t\})\rightarrow0,i\rightarrow\infty$.
  	\end{itemize}
  	Then we must have $\phi = \psi_1$. 
  \end{theorem} 
  \begin{proof} Let assume that $\phi \neq \psi_1,$ i.e., there exists a  compact subset $\mathbf{B}\subseteq M$ which is completely displaced by 
  	$ \psi^{-1}_1\circ \phi$, and since the convergence $ \phi_i^1\longrightarrow \phi$, is uniform, then 
  	we may assume that $ \psi_1^{-1}\circ \phi_i^1$, completely displace  $\mathbf{B}$, for all $i$ sufficiently large. 
  	Fix $i_0$ to be a sufficiently large natural number. Consider  
  	the sequence of almost cosymplectic isotopies  $\{\Psi^{-1}\circ\{\phi_j^t\}\}_{j\geqslant i_0}$ with time-one map  $ \psi_1^{-1}\circ \phi_j^1$, for all 
  	$j\geqslant i_0$, and define a sequence of symplectic isotopies $\tilde \Theta_j$ of the symplectic manifold $(\tilde M, \tilde \omega)$ by setting: 
  	\begin{equation}
  	\tilde \Theta_j := (\psi_1^{-1}\circ \phi_j^1\circ p, \pi_2e^{F_i\circ p} ),
  	\end{equation}
  	and derive that $\tilde \Theta_j^1$ completely displace  
  	$\mathbf{B}_0:= \mathbf{B}\times \{0\}$, i.e., 
  	\begin{equation}
  	0< E_S(\mathbf{B}_0) \leq l_{HL}^{\infty}( \tilde \Theta_j),
  	\end{equation}
  	for all $j$ sufficiently large, where $E_S$ is the symplectic displacement energy, and $l_{HL}^{\infty}$ is the $L^\infty-$version of the Hofer-like length. Finally, the contradiction follows similarly as in the proof of the non-degeneracy of the almost co-Hofer norm. $ \blacksquare$
  \end{proof}
  
  \section{Illustrations}\label{Illus}
  \subsection{Cartesian product of a $2-$disk and a unit circle}
  Let $\mathbb{D}^2$ denote the closed unit disk centered at the origin, equipped with polar coordinates system $(r,\theta)$, with 
  the symplectic form $\Omega := rdr\wedge d\theta$, and let   $\mathbb{S}^1$ denote a unit circle, with coordinate function $s$. 
  \subsubsection{A rotation on $\mathbb{D}^2$} Let $\rho : (0,1]\longrightarrow \mathbb{R}$ be a non-negative $C^1$ map that can be continuously extended 
  to a non-negative $C^1$ map $\bar \rho : [0,1]\longrightarrow \mathbb{R}$. Consider that map, 
  $ \psi_\rho : \mathbb{D}^2\longrightarrow \mathbb{D}^2, (r, \theta) \mapsto (r,\theta + \rho(r) ).$
  Note that $\psi_{\rho}$ is $C^1$ everywhere on $\mathbb{D}^2$, and $ \psi_{\rho}^\ast(\Omega) = \Omega$.
  Let $N$ denote the Cartesian product $ \mathbb{D}^2\times \mathbb{S}^1$ equipped with the cosymplectic structure $\omega := p^\ast(\Omega)$, and 
  $ \eta := \pi_2^\ast(ds)$, where $p:  N\rightarrow \mathbb{D}^2,$ and $\pi_2: N\rightarrow \mathbb{S}^1,$ are projection maps. Let $f : \mathbb{D}^2\longrightarrow \mathbb{R}$ be a smooth function. It is clear that $(N, \omega, \eta)$ is a cosymplectic manifold, and the following map:\\ 
  $ \phi_{\rho, f}: N\longrightarrow N, ((r, \theta), s)\mapsto (\psi_\rho(r, \theta), s e^{f(r, \theta)}),$
  is an almost co-Hamiltonian diffeomorphism of $(N, \omega, \eta)$: we have 
  $\phi_{\rho, f}^\ast(\omega) = p^\ast\left( \psi_\rho^\ast(\Omega)\right)  = p^\ast\left( \Omega\right) ,$ and 
  $$ \phi_{\rho, f}^\ast(\eta) = (\pi_2\circ \psi_\rho)^\ast(ds)  = (\pi_2e^{f\circ p})^\ast\left( ds \right) = e^{f\circ p}\eta.$$
  So, the smooth map $\Phi_{\bar\rho, f}: t\mapsto \phi_{t\bar\rho, tf}$, defines a smooth isotopy in $\mathcal{A}H_{\eta, \omega}(\mathbb{D}^2\times \mathbb{S}^1) $, with $ \Phi_{\rho, f}^0 = (id_{\mathbb{D}^2}, id_{\mathbb{S}^1})$, $ \Phi_{\rho, f}^1 = \phi_{\rho, f}$, and  
  $ \dot \Phi_{\bar \rho, f}^t = \dot \psi_{t\rho} + f\circ p\pi_2 \frac{\partial}{\partial s} 
   = \bar \rho\frac{\partial}{\partial \theta} + f\circ p\pi_2 \frac{\partial}{\partial s},$ for all $t$.  Also, 
  we can compute $\iota(\dot \Phi_{\bar \rho, f}^t )\omega = -r\bar \rho dr = d\left( \int_r^0u\bar \rho(u)du\right)$, and 
  $\eta(\dot \Phi_{\bar \rho, f}^t ) =  f\circ p\pi_2 $, for all $t$. If  $\xi$ is the Reeb vector field on $N$, then we have 
  $\mathcal{L}_{\dot \Phi_{\bar \rho, f}^t }\eta =  \mu\eta,$ 
  with $\mu := \xi(f\circ p\pi_2 )$. 
  \subsubsection{Almost co-Hofer lengths of $  \Phi_{\bar \rho, f}$}
  The $L^{(1, \infty)}-$version of the almost co-Hofer length of $  \Phi_{\bar \rho, f}$ is given by: 
  \begin{equation}\label{equ8-I}
  l_{\mathcal{A}H}^{(1,\infty)} ( \Phi_{\bar \rho, f}) = osc\left(  \int_r^0u\bar \rho(u)du\right) + 
  \frac{1}{2\pi^2}\int_{\mathbb{D}^2\times \mathbb{S}^1}| f\circ p\pi_2|\eta\wedge\omega
  \end{equation}
  $$ = \left( \int_0^1 u\bar \rho(u)du\right)  + \frac{1}{2\pi^2}\left( \int_{\mathbb{D}^2}| f|\Omega\right)
  \left( \int_{-\pi}^\pi|s| ds\right) $$
  $$ = \left( \int_0^1 u\bar \rho(u)du\right)  + \frac{1}{2}\left( \int_{\mathbb{D}^2}| f|\Omega\right).$$
  We derive from the above calculations that the $L^{\infty}-$version of the almost co-Hofer length of $ \Phi_{\bar \rho, f} $  is given by:
$l_{\mathcal{A}H}^{\infty} ( \Phi_{\bar \rho, f})  = \left( \int_0^1 u\bar \rho(u)du\right)  + \frac{1}{2}\left( \int_{\mathbb{D}^2}| f|\Omega\right). 
 $
  It follows from the above facts that, the following estimates of almost co-Hofer norms of $ \phi_{\rho, f} $  hold:\\ 
  $\Arrowvert  \phi_{\rho, f} \Arrowvert _{\mathcal{A}H}^{(1,\infty)} \leq \Arrowvert \phi_{\rho, f} \Arrowvert _{\mathcal{A}H}^{\infty} \leq  \left( \int_0^1 u\bar \rho(u)du\right)  + \frac{1}{2}\left( \int_{\mathbb{D}^2}| f|\Omega\right). $
  \subsubsection{Almost flux of $\Phi_{\bar \rho, f}$}
  By definition of the homomorphism $\mathcal{A}\widetilde{S}_{\eta, \omega}$, we have\\  
  $ \mathcal{A}\widetilde{S}_{\eta, \omega}(\Phi_{\bar \rho, f}) = \int_0^1[\iota(\dot \Phi_{\bar \rho, f}^t )\omega + \eta(\dot \Phi_{\bar \rho, f}^t )\eta]
  dt  = [f\circ p\pi_2\eta].$\\
  Hence, $\mathcal{A}\widetilde{S}_{\eta, \omega}(\Phi_{\bar \rho, f}) =  [f\circ p\pi_2\eta]\in H^1(\mathbb{D}^2\times \mathbb{S}^1, \mathbb{R})$. 
  \subsection{Cartesian product of a $2l-$torus and a unit circle}
  \subsubsection{A rotation on $\mathbb{T}^{2l}$}
  Consider the torus $\mathbb T^{2l}$ with coordinates $(\theta_1,\dots,\theta_{2l})$ and 
  equip it with the flat Riemannian metric $g_0$. 
  Note that all the $1-$forms $d\theta_i$, $i= 1,\dots, 2l$ are harmonic. 
  Take the $1-$forms $d\theta_i$ for $i= 1,\dots, 2l$ as basis for the space of harmonic 
  $1-$forms and consider the symplectic form
  $
  \Omega := \sum_{i=1}^l d\theta_i \wedge d\theta_{i +l}.
  $ 
  Given $v = (a_1,\dots, a_l, b_1,\dots, b_l) \in \mathbb R^{2l}$, 
  the translation $x \mapsto x + v$ on $ \mathbb R^{2l}$ induces a rotation $R_v$ on $\mathbb T ^{2l}$, 
  which is a symplectic diffeomorphism.
  Therefore, the smooth mapping $\{R_{v}^t\} : t \mapsto R_{tv}$ defines a symplectic isotopy 
  such that $\iota(\dot R_{v}^t )\omega = \mathcal{H}$ with $\mathcal{H} = \sum_{i=1}^l \left(a_i d\theta_{i+l} - b_i d\theta_i\right) $.\\ 
  Let $M$ denote the Cartesian product $\mathbb{T}^{2l} \times \mathbb{S}^1$ equipped with the cosymplectic structure $\omega := p^\ast(\Omega)$, and 
  $ \eta := \pi_2^\ast(ds)$, where $p:  M\rightarrow \mathbb{T}^{2l},$ and $\pi_2: M\rightarrow \mathbb{S}^1,$ are projection maps. Let $h : \mathbb{T}^{2l}\longrightarrow \mathbb{R}$ be a smooth function. We have that $(M, \omega, \eta)$ is a cosymplectic manifold, and the following map 
  $ \phi_{ R_{v}, h}: M\longrightarrow M, (\theta_1,\dots,\theta_{2l}, s)\mapsto ( R_{v}(\theta_1,\dots,\theta_{2l}), s e^{h(\theta_1,\dots,\theta_{2l} )}),$
  is an almost co-symplectic diffeomorphism of $(M, \omega, \eta)$: Compute 
  $ \phi_{R_{v}, h}^\ast(\omega) = p^\ast\left( \psi_\rho^\ast(\Omega)\right)  = \omega,$ and 
  $ \phi_{R_{v}, h}^\ast(\eta) = (\pi_2\circ \phi_{R_{v},  h})^\ast(ds)  = (\pi_2e^{h\circ p})^\ast\left( ds \right) = e^{h\circ p}\eta.$ 
  So, the smooth map\\ $\Phi_{R_{v}, h}: t\mapsto \phi_{R_{tv}, th}$, defines a smooth isotopy in 
  $\mathcal{A}G_{\eta, \omega}(\mathbb{T}^{2l}\times \mathbb{S}^1) $, with $ \Phi_{R_{v}, h}^0 = (id_{\mathbb{T}^{2l}}, id_{\mathbb{S}^1})$, 
  $ \Phi_{R_{v}, h}^1 = \phi_{R_{v}, h}$, and 
  $ \dot \Phi_{R_{v}, h}^t =   \sum_{i=1}^l\left( a_i \frac{\partial}{\partial \theta_i} + b_i\frac{\partial}{\partial \theta_{i+l}}\right)  + h\circ p\pi_2 \frac{\partial}{\partial s},$ for all $t$.  Also, 
  we can compute $\iota(\dot \Phi_{R_{v}, h}^t )\omega =  p^\ast\left( \sum_{i=1}^l \left(a_i d\theta_{i+l} - b_i d\theta_i\right) \right) $, and  
  $\eta(\dot \Phi_{R_{v}, h}^t ) =  h\circ p\pi_2 $, for all $t$. Let $\xi$ be the Reeb vector field on $M$, then 
  $\mathcal{L}_{\dot \Phi_{R_{v}, h}^t}\eta =  \nu\eta,$ 
  with $\nu := \xi(h\circ p\pi_2 )$. 
  \subsubsection{Almost co-Hofer-like lengths of $ \Phi_{R_{v}, h} $}
  Let assume that $Vol_{\eta,\omega}(\mathbb{T}^{2l}\times \mathbb{S}^1) = 1 $. 
  The $L^{(1, \infty)}-$version of the almost co-Hofer-like length of $\Phi_{R_{v}, h} $ is given by:\\ 
$
  l_{\mathcal{A}Co}^{(1,\infty)} (\Phi_{R_{v}, h} ) =  \sum_{i=1}^l \left(|a_i| + |b_i|\right)  + 
  \int_{\mathbb{T}^{2l}\times \mathbb{S}^1}| h\circ p\pi_2|\eta\wedge\omega^l
 = \sum_{i=1}^l \left(|a_i| + |b_i|\right)  +  \pi^2\int_{\mathbb{T}^{2l}}| h|\Omega^l.$ \\
 
  We derive from the above calculations that the $L^{\infty}-$version of the almost co-Hofer-like length of $\Phi_{R_{v}, h}  $  is given by:
  $
  l_{\mathcal{A}Co}^{\infty} (\Phi_{R_{v}, h} )  =  \sum_{i=1}^l \left(|a_i| + |b_i|\right)  + \pi^2 \int_{\mathbb{T}^{2l}}| h|\Omega^l.
  $
  It follows from the above facts that, the following estimates of almost co-Hofer-like energies of $ \phi_{R_{v}, h} $ hold,\\
  $e_{\mathcal{A}Co}^{(1,\infty)}(\phi_{R_{v}, h}) \leq e_{\mathcal{A}Co}^{\infty}(\phi_{R_{v}, h})\leq  
  \sum_{i=1}^l \left(|a_i| + |b_i|\right)  + \pi^2\int_{\mathbb{T}^{2l}}| h|\Omega^l.$
  \subsubsection{Almost flux of $\Phi_{R_{v}, h}$}
  By definition, we have  
  $ \mathcal{A}\widetilde{S}_{\eta, \omega}(\Phi_{R_{v}, h}) = [ p^\ast\left( \sum_{i=1}^l \left(a_i d\theta_{i+l} - b_i d\theta_i\right)\right)  +   h\circ p\pi_2\eta].$\\
  Thus, $\mathcal{A}\widetilde{S}_{\eta, \omega}(\Phi_{R_{v}, h}) =  [ \sum_{i=1}^l \left(a_i d\theta_{i+l} - b_i d\theta_i\right) +   h\circ p\pi_2\eta]\in H^1(\mathbb{T}^{2l}\times \mathbb{S}^1, \mathbb{R})$.

\end{document}